\documentclass[final,DIV = calc,headings = normal]{scrartcl}
  \usepackage[USenglish]{babel}

\RequirePackage{iflang}
\RequirePackage{ifdraft}

\RequirePackage[T1]{fontenc} 
\RequirePackage[utf8]{inputenc} 

\makeatletter
  \@ifpackageloaded{babel}{}{\RequirePackage{babel}}
\makeatother
\PassOptionsToPackage{math = normal}{babel}

\RequirePackage[babel]{microtype} 

\RequirePackage[autostyle]{csquotes}  
\MakeOuterQuote{"}

\PassOptionsToPackage{dvipsnames}{xcolor} 
\PassOptionsToPackage{pdfusetitle,hypertexnames = false}{hyperref}	
\ifdraft{%
  \PassOptionsToPackage{draft}{bookmark}%
}{}
\RequirePackage{bookmark}	
\hypersetup{
  colorlinks = true,        
  citecolor = Sepia,        
  linkcolor = Sepia,        
  urlcolor = Sepia,         
  pdfstartview = ,      
  pdfpagelayout = OneColumn,
  pdfborder = 0 0 0,		
  linktoc = page
}

\RequirePackage{etex}

\RequirePackage{amsthm}
\RequirePackage{amsmath}
\ifdraft{
  \RequirePackage[inline,inner]{showlabels} 
}{}

\RequirePackage[shortlabels,inline]{enumitem}

\setlist[itemize,3]{label = $\triangleright$}

\setlist[enumerate,1]{label = \textup{(\roman*)}, ref = \textup{(\roman*)}}
\setlist[enumerate,2]{label = \textup{(\alph*)}, ref = \textup{(\alph*)}}
\setlist[enumerate,3]{label = (\arabic*), ref = (\arabic*)}

\usepackage[capitalize,noabbrev]{cleveref}    

\RequirePackage{autonum}  

\def\nobreakbefore{%
  \relax\ifvmode\else
    \ifhmode
      \ifdim\lastskip > 0pt\relax
        \unskip\nobreakspace
      \fi
    \fi
  \fi
}
\let\oldcite\cite
\renewcommand\cite{\nobreakbefore\oldcite}

\makeatletter
\AtBeginEnvironment{proof}{
  \ifnum\catcode`: = \active
  \expandafter\patchcmd\csname\string\proof\endcsname{\@addpunct{.}}{\@addpunct{ :}}{}{}
\else
  \expandafter\patchcmd\csname\string\proof\endcsname{\@addpunct{.}}{\@addpunct{:}}{}{}
\fi
}
\makeatother

\makeatletter
\def\blfootnote{\gdef\@thefnmark{}\@footnotetext}
\makeatother

\newcommand{\MSC}[1]{\blfootnote{\textup{2010} \textit{Mathematics Subject Classification}: \textup{#1}}}

\KOMAoption{toc}{flat}
\BeforeTOCHead[toc]{%
  \renewcommand\TOCLineLeaderFill[1][]{\hfill}
}

\setkomafont{disposition}{\normalcolor\sffamily}

\recalctypearea

\ifdraft{
  \RequirePackage[scaled = .98,sups]{XCharter}
  }{
  \RequirePackage{Baskervaldx}
  \linespread{1.05}
}

\RequirePackage{FiraSans} 
\RequirePackage[varqu,varl]{zi4}
\ifdraft{
  \RequirePackage[scaled = 1.07,libertine,bigdelims,frenchmath,vvarbb]{newtxmath}
  \linespread{1.02}
  }{
  \RequirePackage[baskervaldx,bigdelims,frenchmath,vvarbb]{newtxmath} 
}
\RequirePackage[cal = boondoxo]{mathalfa}
 
\RequirePackage{bm}	
\useosf

\usepackage{graphics}
\usepackage{amscd}

\numberwithin{equation}{section}

\theoremstyle{plain}
\newtheorem{theorem}{Theorem}[section]
\newtheorem*{theorem*}{Theorem}
\newtheorem{lemma}[theorem]{Lemma}
\newtheorem*{lemma*}{Lemma}

\newtheorem*{corollary*}{Corollary}
\newtheorem{proposition}[theorem]{Proposition}
\newtheorem*{proposition*}{Proposition}

\theoremstyle{remark}

\theoremstyle{definition}

\RequirePackage{etoolbox} 

\PassOptionsToPackage{dvipsnames}{xcolor} 
\ifdraft{\newcommand*{\todo}[1]{\par{\textcolor{red}{TODO: #1}}\PackageWarning{TODO:}{#1}\par}}{\newcommand*{\todo}[1]{}}

\RequirePackage{mathtools}
\AtBeginDocument{
  \ifnum\catcode`: = \active
  \else
    \mathtoolsset{centercolon}
  \fi
}

\RequirePackage{tikz-cd}

\newenvironment{centered}{%
  \begin{list}{}{%
      \topsep0pt
    }
    \centering
  \item[]
  }
  {\end{list}}

\numberwithin{equation}{section}

\IfLanguagePatterns{english}{

  \theoremstyle{plain}
  \newtheorem{prop}[theorem]{Proposition}
  \newtheorem*{prop*}{Proposition}
  \newtheorem{lem}[theorem]{Lemma}
  \newtheorem*{lem*}{Lemma}
  \newtheorem{cor}[theorem]{Corollary}
  \newtheorem*{cor*}{Corollary}
  
  \newtheorem*{fact*}{Fact}
  
  \newtheorem*{conj*}{Conjecture}
  
  \newtheorem*{qst*}{Question}

  \theoremstyle{definition}
  
  \newtheorem*{defn*}{Definition}
  
  \newtheorem*{ex*}{Example}

  \theoremstyle{remark}
  
  \newtheorem*{rem*}{Remark}
  
  \newtheorem*{obs*}{Observation}
  
  \newtheorem*{ass*}{Assumption}
  
  \newtheorem*{ntt*}{Notation}

}{}

  \newcommand*{\iso}{\xrightarrow{\,\smash{\raisebox{-0.65ex}{\ensuremath{\scriptstyle\sim}}}\,}}
  \newcommand*{\opiso}{\xleftarrow{\,\smash{\raisebox{-0.65ex}{\ensuremath{\scriptstyle\sim}}}\,}}

  \newcommand*{\from}{\colon} 

  \newcommand*{\rst}[1]{\ensuremath{{\mathbin\upharpoonright}\raise-.5ex\hbox{$#1$}}}

  \DeclarePairedDelimiter{\norm}{\lvert}{\rvert}

  \newcommand*{\id}{{\mathop{}\mathopen{}\mathrm{id}}}  

  \DeclarePairedDelimiter{\gauss}{\lfloor}{\rfloor}

  \providecommand*{\x}{\times}				  
  				  
  \providecommand*{\ox}{\otimes}

  \newcommand*{\op}{\oplus}

  \DeclareMathOperator{\Gal}{Gal}

  \DeclareMathOperator{\GL}{GL}

  \DeclareMathOperator{\Sym}{Sym}
  
  \DeclareMathOperator{\im}{im}
  
  \DeclareMathOperator{\ind}{ind}

  \newcommand*{\ISO}[1]{\ensuremath{{\mathbb{#1}}}}

  \newcommand*{\ProvideMathOperator}[2]{\providecommand*{#1}{\operatorname{#2}}}

  \ProvideMathOperator{\O}{{\mathcal{O}}}	
  \ProvideMathOperator{\P}{\ISO{P}}

  \newcommand*{\CreateBunchOfLetterCmds}[4]{%
    \renewcommand*{\do}[1]{
      \expandafter\providecommand\csname #2##1#3\endcsname{{#1{##1}}}%
    }%
    \docsvlist{#4}
  }

  \newcommand*{\FixedDomain}[1]{{\mathbf{#1}}}
  \newcommand*{\E}{\FixedDomain{E}}

  \newcommand*{\FixedDomainb}[1]{{\overline{\FixedDomain{#1}}}}
  \CreateBunchOfLetterCmds{\FixedDomainb}{}{b}{E,K,L,M,o,k}
  \newcommand*{\FixedDomainh}[1]{{\widehat{\FixedDomain{#1}}}}
  \CreateBunchOfLetterCmds{\FixedDomainh}{}{h}{E,F,K,L,M,o,k}
  \newcommand*{\FixedDomainbh}[1]{{\widehat{\overline{\FixedDomain{#1}}}}}
  \CreateBunchOfLetterCmds{\FixedDomainbh}{}{bh}{E,F,K,L,M,o,k}

  \newcommand*{\F}{\ISO{F}} 	  		              
                         
  \newcommand*{\N}{\ISO{N}}                       
                         
  \newcommand*{\Q}{\ISO{Q}}                       
  \newcommand*{\Z}{\ISO{Z}}                       

  \newcommand*{\ISOb}[1]{{\overline{\ISO{#1}}}}
  \CreateBunchOfLetterCmds{\ISOb}{}{b}{F,Z,Q}
  \newcommand*{\ISOh}[1]{{\widehat{\ISO{#1}}}}
  \CreateBunchOfLetterCmds{\ISOh}{}{h}{Z,Q}
  \newcommand*{\ISObh}[1]{{\widehat{\overline{\ISO{#1}}}}}
  \CreateBunchOfLetterCmds{\ISObh}{}{bh}{Z,Q}

  \CreateBunchOfLetterCmds{\widetilde}{}{t}{A,B,C,D,E,F,G,H,I,J,K,L,M,N,O,P,Q,R,S,T,U,V,W,X,Y,Z}
  \CreateBunchOfLetterCmds{\bar}{}{b}{A,B,C,D,E,F,G,H,I,J,K,L,M,N,O,P,Q,R,S,T,U,V,W,X,Y,Z}
  \CreateBunchOfLetterCmds{\widehat}{}{h}{A,B,C,D,E,F,G,H,I,J,K,L,M,N,O,P,Q,R,S,T,U,V,W,X,Y,Z}
  \CreateBunchOfLetterCmds{\widehat\overline}{}{bh}{A,B,C,D,E,F,G,H,I,J,K,L,M,N,O,P,Q,R,S,T,U,V,W,X,Y,Z}

  \CreateBunchOfLetterCmds{\mathbf}{b}{}{A,B,C,D,E,F,G,H,I,J,K,L,M,N,O,P,Q,R,S,T,U,V,W,X,Y,Z}
  \CreateBunchOfLetterCmds{\mathbb}{d}{}{A,B,C,D,E,F,G,H,I,J,K,L,M,N,O,P,Q,R,S,T,U,V,W,X,Y,Z}
  \CreateBunchOfLetterCmds{\mathrm}{r}{}{A,B,C,D,E,F,G,H,I,J,K,L,M,N,O,P,Q,R,S,T,U,V,W,X,Y,Z}
  \CreateBunchOfLetterCmds{\mathcal}{c}{}{A,B,C,D,E,F,G,H,I,J,K,L,M,N,O,P,Q,R,S,T,U,V,W,X,Y,Z}
  \CreateBunchOfLetterCmds{\mathfrak}{f}{}{A,B,C,D,E,F,G,H,I,J,K,L,M,N,O,P,Q,R,S,T,U,V,W,X,Y,Z}

  \newcommand*{\BfT}[1]{{\widetilde{\mathbf{#1}}}}
  \newcommand*{\BfB}[1]{{\bar{\mathbf{#1}}}}
  \newcommand*{\BfH}[1]{{\widehat{\mathbf{#1}}}}
  \newcommand*{\BfBH}[1]{{\widehat{\overline{\mathbf{#1}}}}}
  \newcommand*{\DfT}[1]{{\widetilde{\mathbf{#1}}}}
  \newcommand*{\DfB}[1]{{\bar{\mathbb{#1}}}}
  \newcommand*{\DfH}[1]{{\widehat{\mathbb{#1}}}}
  \newcommand*{\DfBH}[1]{{\widehat{\overline{\mathbb{#1}}}}}
  \newcommand*{\RfT}[1]{{\widetilde{\mathbf{#1}}}}
  \newcommand*{\RmB}[1]{{\bar{\mathrm{#1}}}}
  \newcommand*{\RmH}[1]{{\widehat{\mathrm{#1}}}}
  \newcommand*{\RfBH}[1]{{\widehat{\overline{\mathrm{#1}}}}}
  \newcommand*{\CfT}[1]{{\widetilde{\mathbf{#1}}}}
  \newcommand*{\ClB}[1]{{\bar{\mathcal{#1}}}}
  \newcommand*{\ClH}[1]{{\widehat{\mathcal{#1}}}}
  \newcommand*{\CfBH}[1]{{\widehat{\overline{\mathcal{#1}}}}}
  \newcommand*{\FfT}[1]{{\widetilde{\mathbf{#1}}}}
  \newcommand*{\FlB}[1]{{\bar{\mathfrak{#1}}}}
  \newcommand*{\FlH}[1]{{\widehat{\mathfrak{#1}}}}
  \newcommand*{\FfBH}[1]{{\widehat{\overline{\mathfrak{#1}}}}}

  \CreateBunchOfLetterCmds{\BfT}{b}{t}{A,B,C,D,E,F,G,H,I,J,K,L,M,N,O,P,Q,R,S,T,U,V,W,X,Y,Z}
  \CreateBunchOfLetterCmds{\BfB}{b}{b}{A,B,C,D,E,F,G,H,I,J,K,L,M,N,O,P,Q,R,S,T,U,V,W,X,Y,Z}
  \CreateBunchOfLetterCmds{\BfH}{b}{h}{A,B,C,D,E,F,G,H,I,J,K,L,M,N,O,P,Q,R,S,T,U,V,W,X,Y,Z}
  \CreateBunchOfLetterCmds{\BfBH}{b}{bh}{A,B,C,D,E,F,G,H,I,J,K,L,M,N,O,P,Q,R,S,T,U,V,W,X,Y,Z}
  \CreateBunchOfLetterCmds{\DfT}{b}{t}{A,B,C,D,E,F,G,H,I,J,K,L,M,N,O,P,Q,R,S,T,U,V,W,X,Y,Z}
  \CreateBunchOfLetterCmds{\DfB}{d}{b}{A,B,C,D,E,F,G,H,I,J,K,L,M,N,O,P,Q,R,S,T,U,V,W,X,Y,Z}
  \CreateBunchOfLetterCmds{\DfH}{d}{h}{A,B,C,D,E,F,G,H,I,J,K,L,M,N,O,P,Q,R,S,T,U,V,W,X,Y,Z}
  \CreateBunchOfLetterCmds{\DfBH}{d}{bh}{A,B,C,D,E,F,G,H,I,J,K,L,M,N,O,P,Q,R,S,T,U,V,W,X,Y,Z}
  \CreateBunchOfLetterCmds{\RfT}{b}{t}{A,B,C,D,E,F,G,H,I,J,K,L,M,N,O,P,Q,R,S,T,U,V,W,X,Y,Z}
  \CreateBunchOfLetterCmds{\RmB}{r}{b}{A,B,C,D,E,F,G,H,I,J,K,L,M,N,O,P,Q,R,S,T,U,V,W,X,Y,Z}
  \CreateBunchOfLetterCmds{\RmH}{r}{h}{A,B,C,D,E,F,G,H,I,J,K,L,M,N,O,P,Q,R,S,T,U,V,W,X,Y,Z}
  \CreateBunchOfLetterCmds{\RfBH}{r}{bh}{A,B,C,D,E,F,G,H,I,J,K,L,M,N,O,P,Q,R,S,T,U,V,W,X,Y,Z}
  \CreateBunchOfLetterCmds{\CfT}{b}{t}{A,B,C,D,E,F,G,H,I,J,K,L,M,N,O,P,Q,R,S,T,U,V,W,X,Y,Z}
  \CreateBunchOfLetterCmds{\ClB}{c}{b}{A,B,C,D,E,F,G,H,I,J,K,L,M,N,O,P,Q,R,S,T,U,V,W,X,Y,Z}
  \CreateBunchOfLetterCmds{\ClH}{c}{h}{A,B,C,D,E,F,G,H,I,J,K,L,M,N,O,P,Q,R,S,T,U,V,W,X,Y,Z}
  \CreateBunchOfLetterCmds{\CfBH}{c}{bh}{A,B,C,D,E,F,G,H,I,J,K,L,M,N,O,P,Q,R,S,T,U,V,W,X,Y,Z}
  \CreateBunchOfLetterCmds{\FfT}{b}{t}{A,B,C,D,E,F,G,H,I,J,K,L,M,N,O,P,Q,R,S,T,U,V,W,X,Y,Z}
  \CreateBunchOfLetterCmds{\FlB}{f}{b}{A,B,C,D,E,F,G,H,I,J,K,L,M,N,O,P,Q,R,S,T,U,V,W,X,Y,Z}
  \CreateBunchOfLetterCmds{\FlH}{f}{h}{A,B,C,D,E,F,G,H,I,J,K,L,M,N,O,P,Q,R,S,T,U,V,W,X,Y,Z}
  \CreateBunchOfLetterCmds{\FfBH}{f}{bh}{A,B,C,D,E,F,G,H,I,J,K,L,M,N,O,P,Q,R,S,T,U,V,W,X,Y,Z}

  \RequirePackage[xspace]{ellipsis}
  \renewcommand*\mod[1]{\allowbreak\mkern5mu \mathrm{mod}\,\,#1}

\csname endofdump \endcsname

\renewcommand{\ss}{\textnormal{ss}}

\newcommand{\unr}{\textnormal{u}}

\usepackage[affil-it]{authblk}

\author[1]{\href{mailto:enno.nagel+math@gmail.com}{\texttt{Enno Nagel}}}
\author[2]{\href{mailto:aftab.pande@gmail.com}{\texttt{Aftab Pande}}}
\affil[1]{enno.nagel+math@gmail.com, Instituto de Matemática, UFAL, Maceió}
\affil[2]{aftab.pande@gmail.com, Instituto de Matemática, UFRJ, Rio de Janeiro}

\date{}

\begin{document}
  \title{Reductions of modular Galois representations of Slope (2,3)}
  \vspace{-2ex}
  \maketitle

  \begin{abstract}
    \noindent\textsc{Abstract}.\;
    We compute, via the $p$-adic Langlands correspondence, the semisimplifications of the mod-$p$ reductions of $2$-dimensional crystalline representations of $\Gal(\Qb_p/\Q_p)$ of slope $(2,3)$.
  \end{abstract}

  \MSC{11F85, 22E50, 11F80}

  \setcounter{secnumdepth}{5}
  \setcounter{tocdepth}{5}
  \setcounter{section}{-1}
  \tableofcontents


\section{Introduction}

Let $p$ be a prime number.
What is the (local two-dimensional crystalline) mod-$p$ Galois representation attached to a modular form of weight $k$, an integer $\geq 2$, and Hecke-operator eigenvalue $a_p$, a point in the $p$-adic open unit disc?
There is no general answer yet.
To conjure a conjecture, several authors computed the more accessible cases near the boundary of the disc, that is, the cases of lower \emph{slope}, $p$-adic valuation of $a_p$, (and generic weight $k \geq 2$) via the mod-$p$ local Langlands correspondence (as first conceived in \cite{B03}, proved in \cite{BLZ} for small weights with respect to the slope, that is, $v(a_p) > \lfloor\frac{k-1}{p}\rfloor$, and recently improved upon in \cite{BL}):
\begin{itemize}
  \item for slope $0 < v(a_p) < 1$ and weight $k > 2p + 2$ (with $k \nequiv 3 \mod (p-1)$ in \cite{BG09} and $k \equiv 3 \mod (p-1)$ and $p > 2$ in \cite{BG13},) and
  \item for $p > 3$, slope  $v(a_p) = 1$ (and weight $k \geq 2p+2$) in \cite{BGR}
  \item for $p \geq 3$, slope  $1 < v(a_p) < 2$ (with a condition on $a_p$ when $v(a_p) = 3/2$) and weight $2p + 2 \leq k \leq p^2 - p$ in \cite{GG};
    then for all weights in \cite{BG} and for $v(a_p) = 3/2$ (and $p > 3$) in \cite{GhateRai}.
\end{itemize}

In this article, we extend these results to slope $2 < v(a_p) < 3$ (with a condition on $a_p$ when $v(a_p) = 5/2$).

\subsection{Parametrizations of $p$-adic Galois representations mod $p$}

We will follow the notation of \cite{GG} and \cite{BG}.
Let $\E$ be a finite extension of $\Q_p$ and let $v$ be the additive valuation on $\E$ satisfying $v(p) = 1$.

Let $\cG_{\Q_p}$ be the absolute Galois group $\Gal(\Qb_p / \Q_p)$ of ${\Q_p}$.
A \emph{$p$-adic Galois representation} is a continuous action of $\cG_{\Q_p}$ on a finite-dimensional vector space defined over $\E$.

Among all $p$-adic Galois representations the \emph{crystalline} Galois representations admit an explicit parameterization:
Every crystalline representation $V$ of dimension $2$ is uniquely determined (up to twist by a crystalline character) by
\begin{itemize}
  \item
    a \emph{weight}, an integer $k \geq 2$, and
  \item
    an \emph{eigenvalue} $a_p$ in $\E$ with $v(a_p) > 0$.
\end{itemize}
The rational number $v(a_p)$ is called the \emph{slope} of $V$.

Inside $V$ the compact group $\cG_{\Q_p}$ stabilizes a lattice.
The (induced) representation of $\cG_{\Q_p}$ on the \emph{semisimplified} mod $p$ reduction $\Vb$ of $V$ over $\Fb_p$, by the Brauer-Nesbitt principle, is independent of the choice of this lattice.
Let $V_{k,a_p}$ be the crystalline representation of weight $k$ and eigenvalue $a_p$, that is, the crystalline representation attached to the (admissible) $\phi$-module of basis $\{e_1, e_2\}$ whose Frobenius $\phi$ and filtration $V_\bullet$ is given
by
\[
  \phi
  =
  \begin{pmatrix}
    0 & -1 \\
    p^{k-1} & a_p
  \end{pmatrix}
  \quad \text{ and } \,
  \ldots = V_0 = V \supset V_1 = \ldots = V_{k-1} = \E \cdot e_1 \supset 0 = V_k = \ldots
\]
We will denote by $\bar{V}_{k,a_p}$ the semisimplified mod $p$ reduction of $V_{k,a_p}$.

The finite-dimensional irreducible Galois representation over $\Fb_p$ are classified and, up to twists by unramified characters, parametrized by integers, as follows:
For $n$ in $\N$, let $\Q_{p^n}$ (respectively $\Q_{p^{-n}}$) be the smallest field extension of $\Q_p$ that contains a primitive $(p^n-1)$-th root $\zeta_n$ (respectively $p_n$) of $1$ (respectively of $-p$).
The \emph{fundamental character} $\omega_n \from \Gal(\Q_{p^{-n}} / \Q_{p^n}) \to \F_{p^n}^*$ is defined by
\[
  \sigma \mapsto \zeta_n
  \quad \text{ where $\zeta_n$ is determined by } \sigma(p_n) = \zeta_n \cdot p_n.
\]
Let $\omega := \omega_1$.
For $\lambda$ in $\Fb_p$, let $\unr(\lambda) \from \cG_{\Q_p} \to \Fb_p^*$ be the unramified character that sends the (arithmetic) Frobenius to $\lambda$.
For $a$ in $\Z$, let
\[
  \ind_{\cG_{\Q_{p^n}}}^{\cG_{\Q_p}} \omega_n^a
  :=
  \Fb_p[\cG_{\Q_p}] \otimes_{\Fb_p[\cG_{\Q_{p^n}}]} \omega_n^a
\]
be the induction of $\omega_n^a$ from $\cG_{\Q_{p^n}}$ to $\cG_{\Q_p}$.
The conjugated characters $\omega_n(g \cdot g^{-1})$ for $g$ in $\cG_{\Q_p}$ are $\omega_n, \omega_n^p, \ldots, \omega_n^{p^{n-1}}$ and all distinct;
therefore, by Mackey's criterion,    $\ind_{\cG_{\Q_{p^n}}}^{\cG_{\Q_p}} \omega_n^a$ is irreducible and its determinant is $\omega^a$ on $\cG_{\Q_{p^n}}$.
Let $\ind(\omega_n^a)$ denote the twist of $\ind_{\cG_{\Q_{p^n}}}^{\cG_{\Q_p}} \omega_n^a$ by the unramified character that turns its determinant into $\omega^a$ on all of $\cG_{\Q_p}$.

Every irreducible $n$-dimensional representation of $\cG_{\Q_p}$ over $\Fb_p$ is of the form $\ind(\omega_n^a) \otimes \unr (\lambda)$ for some $a$ in $\Z$ and $\lambda$ in $\Fb_p^*$ (cf. [op. cit., Paragraph 1.1]).
In particular, every mod $p$ reduction of dimension $2$ is either of the form
\[
  \ind(\omega_2^a) \otimes \unr (\lambda) \text{ \quad or \quad } (\omega_1^a \otimes \unr (\lambda)) \oplus (\omega_1^b \otimes \unr (\mu))
\]
for some $a,b$ in $\Z$ and $\lambda,\mu$ in $\Fb_p^*$.

The powers $a$ and $b$ of the fundamental character $\omega_2$ are not unique in $\Z$ but satisfy the following congruences:
$\omega_2$ has order $p^2-1$, so $\omega_2^{p^2} = \omega_2$, and $\omega_2^i$ and $\omega_2^{ip}$ are conjugate under $\cG_{\Q_p}$, thus have isomorphic inductions.

There are also restrictions on the exponents occurring in the mod $p$ reduction:
We recall that the Galois representation $V_{k, a_p}$ is obtained from a filtered $\phi$-module by a functor;
which is a tensor functor, in particular, it is compatible with taking the determinant.
This way, the determinant of the Galois representation $V_{k, a_p}$ is known and can be made explicit, and so its mod $p$ reduction.
It is $\omega^{k - 1}$.
At the same time, we recall that the determinant of $\ind(\omega_2^i)$ is (by definition) $\omega^i$.

\subsection{Main Theorem}

For a weight $k$ and an eigenvalue $a_p$ that parametrize a crystalline representation $V_{k,a_p}$, we compute $a$ in $\Z$ and $\lambda$ in $\Fb_p^*$ that parametrize the mod $p$ reduction $\Vb_{k,a_p}$ for
\begin{itemize}
  \item
    a weight $k$ in certain mod $(p-1)$ and mod $p$ congruence classes, and
  \item
    a slope $2 < v(a_p) < 3$.
\end{itemize}

Applying \cite[Lemma 3.3]{BG09} to the results of \cref{sec:ElimJH} and \cref{sec:sepRedIrred} yields:

\begin{theorem}
  Let $r := k-2$ and $a$ in $\{ 3, \ldots, p + 1 \}$ such that $r \equiv a \mod (p-1)$.
    If $p \geq 5$, $r \geq 3p + 2$ and $v(a_p)$ in $]2,3[$ (and, if for $a=5$ or $p=5$,  $v(a_p) = 5/2$, then $v(a_p^2 - p^5) = 5$), then
    \[
      \overline{V}_{k,a_p} \cong
      \begin{cases}
        \ind(\omega_2^{a + 1}),                               & \quad \text{ for } a = 3 \text{ and } r \not \equiv 0,1,2 \mod p\\
        \ind(\omega_2^{a+p}),                              & \quad \text{ for } a = 3 \text{ and } r \equiv 0 \mod p\\
        \ind(\omega_2^{a+p}),                               & \quad \text{ for } a = 4 \text{ and } r \not\equiv 1,2,3,4 \mod p\\
        \ind(\omega_2^{a+2p-1}),                               & \quad \text{ for } a = 4 \text{ and } r \equiv 1 \mod p\\
        \ind(\omega_2^{a + 1}),                               & \quad \text{ for } a = 4 \text{ and } r  \equiv 4 \mod p\\
        \ind(\omega_2^{a + 2p-1}),                            & \quad \text{ for } a = 5 \text{ and } r \equiv 2,3 \mod p\\
        \ind(\omega_2^{a + 2p-1}),                            & \quad \text{ for } a = 5 \text{ and } r \not\equiv 2,3,4,5 \mod p, v(a_p^2) \neq 5\\
        \ind(\omega_2^{a + p}),                               & \quad \text{ for } a = 5, \ldots, p-1 \text{ and } p \parallel r - a  \\
        \ind(\omega_2^{a + 1}),                               & \quad \text{ for } a = 5, \ldots, p-1 \text{ and } r \equiv a \mod p^2\\
        \ind(\omega_2^{a + 2p-1}),                            & \quad \text{ for } a = 6, \ldots, p \text{ and } r \not\equiv a, a-1 \mod p\\
        \ind(\omega_2^{a + p}),                               & \quad \text{ for } a = 5, \ldots, p  \text{ and } p | r-a+1\\
        \ind(\omega_2^{a+p}),                                 & \quad \text{ for }  a = p \text{ and } p | r - p \text{ but not } p^3 | r - p   \\
        \unr(\sqrt{-1})\omega \oplus \unr(-\sqrt{-1})\omega , & \quad \text{ for } a = p \text{ and } r \equiv p \mod p^3 \\
        \ind(\omega_2^{a+2p-1}),                              & \quad \text{ for } a =p+1 \text{ and } r \not \equiv 0, 1 \mod p \\
      \end{cases}
    \]
    where $\parallel$ denotes exact divisibility.
\end{theorem}

This result is as predicted by the main theorem of \cite{BG}:
Since the slope increases by a unit, here the reducible case occurs when $p^3 \mid p-r$ (whereas, in \cite{BG}, when $p^2 \mid p-r$).
In \cite{A}, Arsovski examines whether the representation is irreducible or not, for a large class of slopes (integral and non-integral) and even weights, but does not specify it.
In the cases where he eliminates certain factors (as in our \cref{sec:ElimJH}) his results are compatible with ours.
Our results agree with the results of \cite[Section 4.2 Case (ii)]{GhateKumar} (where our $a = b+2$).
Our results in \cref{sec:Xr-2} and \ref{sec:singular-quotients-x-r-2} are also compatible with those of \cite{GhateVangala}.

Here we deal with all weights and compute the exact shape of the representation, but we could not address:
\begin{itemize}

  	\item the case $p \mid r -(a-1) \mod p$ for $a = p+1  \text{ or } 2$, $a= 3  \text{ or } p+2$ and $a= 4  \text{ or } p+3$.
        \item the case $a = 5$ and $p \mid r - 5, r-4$, (to determine reducibility in \cref{sec:sepRedIrred}), and
    \item the case $v(a_p) = 5/2$ when $v(a_p^2 - p^5) \ne 5$.
\end{itemize}

The latter two cases are part of Ghate's zig-zag conjecture (see \cite{GhateZigZag}), which has been addressed in recent work (see \cite{GhateRai}) for $a=3$ and $v(a_p) = 3/2$.
The cases addressed give further evidence for:

\begin{conj*}[{\cite[Conjecture 4.1.1]{BG14}}]
  Let $\bar{V}_{k,a_p}$ be the semisimplified mod $p$ reduction of \, $V_{k,a_p}$.
  If $p$ is odd, $k$ is even and $v(a_p) \not \in \mathbb{Z}$, then $\bar{V}_{ k,a_p}$ is irreducible.
\end{conj*}

\subsection{Outline}

We refer to \cite{BG09}, \cite{GG} and \cite{BG} for a more detailed exposition.
Let $\bL$ be the \emph{$2$-dimensional mod $p$ local Langlands correspondence}, an injection
\begin{center}
  \setlength{\tabcolsep}{2pt}
  \begin{tabular}{lcr}
    $\left\{
      \begin{tabular}{c}
        continuous actions of $\Gal(\Qb_p / \Q_p)$\\
        on $2$-dimensional $\Fb_p$-vector spaces
      \end{tabular}
    \right\}$
      & $\hookrightarrow$ &
      $\left\{
        \begin{tabular}{c}
          semisimple \emph{smooth}\\
          actions of $\GL_2(\F_p)$ on\\
          $\Fb_p$-vector spaces
        \end{tabular}
      \right\}$
  \end{tabular}
\end{center}
Since $\bL$ is injective, to determine $\bar{V}_{k,a_p}$, it suffices to determine $\bL(\bar{V}_{k,a_p})$.
As $\bL$ and the $p$-adic local Langlands correspondence (the analog of the mod $p$ local Langlands correspondence that attaches actions of $\Gal(\Qb_p / \Q_p)$ on $2$-dimensional $\Q_p$-vector spaces to actions of $\GL_2(\Q_p)$ on Banach spaces) are compatible with taking the mod $p$ reduction,
\[
  \bL(\bar{V}_{k,a_p}) = \bar{\Theta}_{k,a_p}^\ss
\]
where the right-hand side is the representation of $\GL_2(\Q_p)$ over the (infinite dimensional) $\Fb_p$-vector space given by
\begin{itemize}
  \item
    the semisimplification $\bar{\Theta}_{k,a_p}^\ss$ of
  \item
    the reduction modulo $p$ $\bar{\Theta}_{k,a_p}$ of the canonical lattice $\Zb_p$-lattice $\Theta_{k,a_p}$ of the base extension $\Pi_{k,a_p}$ from $\E$ to $\Qb_p$ of
  \item
    the representation $\Pi_{k,a_p}$ of $\GL_2(\Q_p)$ that corresponds to $V_{k,a_p}$ under the $p$-adic local Langlands correspondence;
    explicitly, with $r = k-2$,
    \[
      \Pi_{k,a_p} = \ind_{KZ}^G \Sym^r \Qb_p^2 / (T - a_p)
    \]
    where
    \begin{itemize}
      \item
        $G = \GL_2(\Q_p)$, $K = \GL_2(\Z_p)$ and $Z = \Q_p^*$ is the center of $G$,
      \item
        $\Sym^r \Qb_p^2$ is the representation of $\GL_2(\Q_p)$ given by all homogeneous polynomials of total degree $r$, and
      \item
        $T$ is the \emph{Hecke operator} that generates the endomorphism algebra of all $\Qb_p[G]$-linear maps on $\ind_{KZ}^G \Sym^r \Qb_p^2$.
    \end{itemize}
\end{itemize}

The canonical $\Zb_p$-lattice $\Theta_{k,a_p}$ of $\Pi_{k,a_p}$ is given by the image
\[
  \Theta_{k,a_p} := \im (\ind_{KZ}^G \Sym^r \Zb_p^2 \to \Pi_{k,a_p})
\]
and the mod $p$-reduction $\bar\Theta_{k,a_p}$ by $\Theta_{k,a_p} / p \Theta_{k,a_p}$.

Let $V_r := \Sym^r \Fb_p^2$.
It is a representation of $\GL_2(\F_p)$ that extends to one of $KZ$ by letting $p \in Z$ act trivially.
We note that there is a natural $\Fb_p[G]$-linear surjection
\[
  \ind_{KZ}^G V_r \twoheadrightarrow \bar{\Theta}_{k,a_p}. \tag{$*$}
\]
Our main result will be that, generally, there is a single Jordan-Hölder factor $J$ of $V_r$ whose induction surjects onto the right-hand side.
Then \cite[Proposition 3.3]{BG09} uniquely determines $\bar V_{k,a_p}$.

To find the Jordan-Hölder factor $J$ of $V_r$, we first define a quotient $Q$ of $V_r$ whose induction surjects onto the right-hand side.
For this, let $X(k,a_p)$ denote the kernel of the above epimorphism.
Put $\Gamma := \GL_2(\F_p)$.

Let $\theta := X^p Y - X Y^p \in V_{p + 1}$ and let $V_r^{***}$ be the image of the map from $V_{r-3p-3}$ to $V_r$ given by multiplication with $\theta^3$.
For $i = 0, \ldots, r$, let
\[
  X_{r-i} := \text{ the $\Fb_p[\Gamma]$-submodule of $V_r$ generated by $X^i Y^{r-i}$ }.
\]
\small
\begin{obs*}
  Mistakably, the notation $X_{r-i}$ involves \emph{two} parameters, $r$ in $\N$ for the surrounding, and $i$ in $\{ 0, \ldots, r \}$ for the inner submodule:
  For example, put $r' = r-1$.
  Then $X_{r'}$ is the submodule of $V_{r-1}$, homogeneous polynomials of two variables of total degree $r-1$, generated by $Y^{r-1}$;
  whereas $X_{r-1}$ is the submodule of $V_r$, homogeneous polynomials of two variables of total degree $r$, generated by $X Y^{r-1}$.
\end{obs*}
\normalsize

By \cite[Remark 4.4]{BG09},
\begin{itemize}
  \item
    if $2 < v(a_p)$, then $\ind_{KZ}^G X_{r-2} \subseteq X(k,a_p)$, and
  \item
    if $v(a_p) < 3$, then $\ind^G_{KZ} V_r^{***} \subseteq X(k,a_p)$.
\end{itemize}
Finally put
\[
  Q := V_r / (X_{r-2} + V_r^{***})
\]
Thence, if $2 < v(a_p) < 3$, then the epimorphism $(*)$ induces an epimorphism
\[
  \ind_{KZ}^G Q \twoheadrightarrow \bar\Theta_{k,a_p}.
\]
Thus we need to understand the modules $X_{r-2}$, $V_r^{***}$ and their intersection $X_{r-2}^{***} := X_{r-2} \cap V_r^{***}$:
In \cref{lem:VrAstQuotients}, the Jordan-Hölder series of $V_r / V_r^{***}$ is computed.
In \cref{sec:Xr-2}, the Jordan-Hölder series of $X_{r-2}$ and $X_{r-2} / X_{r-2}^*$ is computed (where $X_{r-2}^* := X_{r-2} \cap V_r^*$), and
in \cref{sec:singular-quotients-x-r-2} that of $X_{r-2}^* / X_{r-2}^{***}$.
The computed modules depend on the congruence classes of $r$ modulo $p-1 = \#\F_p^*$ respectively $p = \#\F_p$,
as well as (the sum of) the digits of the $p$-adic expansion of $r$.

We then compute in \cref{sec:Q} the Jordan-Hölder factors of $Q$:
A priori, $Q$ has at most $6$ Jordan-Hölder factors.
If $Q$ happens to have a \emph{single} Jordan-Hölder factor, that is, if there is a homomorphism of an irreducible module onto $\bar{\Theta}_{k,a_p}$, then \cite[Proposition 3.3]{BG09} describes $\bar{\Theta}_{k,a_p}$ completely.

Otherwise, that is, if $Q$ happens to have \emph{more than one} Jordan-Hölder factor $J$, then in \cref{sec:ElimJH} we show, for all but a single Jordan-Hölder factor $J_0$ of $Q$, there are functions $f_J$ in $\ind_{KZ}^G \Sym^r \Qb_p^2$ such that
\begin{itemize}
  \item its image $(T - a_p)(f_J)$ under the Hecke operator lies in $\ind_{KZ}^G \Sym^r \Zb_p^2$, and
  \item its mod $p$ reduction $\bar{f}_J$  lies in $\ind_{KZ}^G J$, and generates the entire $\Fb_p[G]$-module $\ind_{KZ}^G J$ (this holds, for example, when it is supported on a single coset of $G / KZ$).
\end{itemize}
Then \cite[Proposition 3.3]{BG09} applied to
$\ind_{KZ}^G J_0 \twoheadrightarrow \bar\Theta_{k,a_p}$ describes $\bar{\Theta}_{k,a_p}$ completely.

In \cref{sec:sepRedIrred}, if the only remaining Jordan-Hölder factor is $V_{p-2}\ox \rD^n$ for some $n$, we need to distinguish between the irreducible and reducible case:
To this end we construct additional functions and observe whether the map $\ind_{KZ}^G V_{p-2} \ox \rD^n \rightarrow \overline{\Theta}_{k,a_p}$ factors through the cokernel of either $T$ (in which case irreducibility holds) or of $T^2 - cT +1$ for some $c \in \overline{\mathbb{F}}_p$ (in which case reducibility holds).


\section{Groundwork}
\label{sec:lemmas}

We restate key results of \cite{G} in our notation (which follows that of \cite{GG}, \cite{BG} and \cite{BGR}).
Let $M$ be the multiplicative monoid of all $2 \x 2$-matrices with coefficients in $\F_p$.
Inside the $M$-representation of all homogeneous polynomials of two variables,
\begin{itemize}
  \item
    here, as in \textit{op. cit.}, $V_r$ denotes the subrepresentation given by all those of (total) degree $r$, a vector space of dimension $r + 1$,
  \item
     whereas in \cite{G}, it denotes the subrepresentation given by all those of (total) degree $r-1$, a vector space of dimension $r$.
\end{itemize}
That is, there is a one-dimensional offset.

  \subsection{The Jordan-Hölder series of $V_m \ox V_n$ for $m = 2,3$}
  \label{sec:the-jordan-hölder-series-of-v-m-ox-v-n-for-m-23}

  For an $M$-representation $U$, let $\sigma U$ and $\varphi U$ denote the socle and cosocle of $U$.

  \begin{lemma}[The Jordan-Hölder series of a Tensor product of two irreducible modules as given in {\cite[(5.5) and (5.9)]{G}}.]
  \label{lem:JHTensor}
    Let $0 \leq m \leq n \leq p-1$.
    \begin{enumerate}
      \item
        \label{en:JHTensor0p-1}
        If $0 \leq m + n \leq p-1$, then
        \[
          V_m \otimes V_n \cong \bigoplus_{i = 0, \ldots, m} V_{m + n -2i} \otimes \rD^i.
        \]
      \item
        \label{en:JHTensorp2p-1}
        If $p \leq m + n \leq 2p-2$, then
        \[
          V_m \otimes V_n \cong V_{p(m + n + 2 - p)-1} \oplus (V_{p-n - 2} \otimes V_{p - m - 2} \otimes \rD^{m + n + 2-p})
        \]
        where the second summand equals
        \[
          (V_{p - n - 2} \otimes V_{p - m - 2} \otimes \rD^{m + n + 2 - p})
          \cong
          \bigoplus_{i = 0, \ldots, p - n - 2} V_{(p - m - 2) + (p - n - 2) - 2i} \otimes \rD^{m + n + 2 - p + i}
        \]
        and the first summand $V = V_{(k + 1)p - 1}$ for $k$ in $\{ 1, \ldots, p - 1 \}$ is a direct sum
        \[
          V = \bigoplus_{m = 0, \ldots, \gauss{k/2}} U_{k - 2m} \ox \rD^m
        \]
        where $U_0 = V_{p-1}$ and $U = U_l$ for $l$ in $\{ 1, \ldots, p \}$ has Jordan-Hölder series
        \[
          0 \subset \sigma U \subset \varphi U \subset U
        \]
        whose successive semisimple Jordan-Hölder factors $\Ub = \sigma U$, $\Ub' = \varphi U / \sigma U$ and $\Ub^{\prime\prime} = U / \varphi U$ are
        \begin{itemize}
          \item
            $\Ub = \Ub^{\prime\prime} = V_{p - l - 1} \ox \rD^l$, and
          \item
            $\Ub' = (V_{l - 2} \otimes \rD) \op V_l$.
        \end{itemize}
        with the convention that $V_k = 0$ for $k < 0$.
    \end{enumerate}
  \end{lemma}

  \begin{cor}[of {\cref{lem:JHTensor}}]
  \label{cor:JHVxV2}
    As $\F_p[M]$-modules
        we have $V_2 \ox V_{p-2} = V_{p-4} \ox \rD^2 \op V_{2p-1}$ where $V_{2p-1}$ has successive semisimple Jordan-Hölder factors $V_{p-2} \otimes \rD$, $V_1$ and $V_{p-2} \ox \rD$.
  \end{cor}

  \subsection{The singular submodules of $V_r$}
  \label{sec:singular-submodules-of-v-r}

  We recall that $\Gamma := \GL_2(\F_p)$.

  \begin{lemma}[Extension of {\cite[Propositions 2.1 and 2.2]{BG}}]
  \label{lem:VrAstQuotients}
    Let $p > 2$.
    The short exact sequence of $\F_p[\Gamma]$-modules
    \begin{enumerate}
      \item
        of $V_r/V_r^*$, for $r \geq p$, and $r \equiv a \mod (p-1)$ with $a \in \{1, \ldots, p-1\}$ is
        \[
          0 \to V_a \to V_r / V_r^* \to V_{p-a-1} \ox \rD^a \to 0,
        \]
        and this sequence splits if and only if $a = p-1$;
        \label{en:VrModVrAst}
      \item
        of $V_r^*/V_r^{**}$ for $r \geq 2p + 1$, and $r \equiv a \mod (p-1)$ with $a \in \{3, \ldots, p + 1\}$ is
        \[
          0 \to V_{a-2} \ox \rD \to V_r^* / V_r^{**} \to V_{p-a + 1} \ox \rD^{a-1} \to 0
        \]
        and this sequence splits if and only if $a = p + 1$;
        \label{en:VrAstModVrAstAst}
      \item
        of $V_r^{**}/V_r^{***}$, for $r \geq 3p + 2$, and $r \equiv a \mod (p-1)$ with $a \in \{5, \ldots, p + 3\}$ is
        \[
          0 \rightarrow V_{a-4} \otimes \rD^2 \rightarrow V_r^{**}/V_r^{***} \rightarrow V_{p-a + 3} \otimes \rD^{a-2} \rightarrow 0
        \]
        and this sequence splits if and only if $a = p + 3$.
        \label{en:VrAstAstModVrAstAstAst}
    \end{enumerate}
  \end{lemma}
  \begin{proof}
    See \cite[Proposition 2.1 and 2.2]{BG} for (i) and (ii) respectively.
    For (iii), follow the proof of Proposition 2.2 in loc.cit. and use $ V_r^{**}/V_r^{***} \cong  (V_{r-p-1}^{*}/V_{r-p-1}^{**})\otimes \rD$.

    The sequences in (i), (ii), (iii) split for $a = p-1, p + 1, p + 3$ respectively because $V_{p-1}$ is an injective module over $\F_p[\Gamma]$.
  \end{proof}

  \begin{lemma}[Extension of {\cite[Lemma 2.3]{BG}}]
    \label{lem:VrAstCriteria}
    Let $F(X,Y) = \sum_{0 \leq j \leq r} c_j X^{r-j}Y^j$ in $V_r$.
    If the indices of all nonzero coefficients are congruent mod $(p-1)$, that is, $c_j, c_k \not= 0$ implies $j \equiv k \mod (p-1)$, then
    \begin{enumerate}
      \item
        $F \in V_r^*$ if and only if $c_0 = 0 = c_r$ and $\sum c_j = 0$,
        \label{en:VrAstCriterion}
      \item
        \label{en:VrAstAstCriterion}
        $F \in V_r^{**}$ if and only if
        $
          c_0 = c_1 = 0 = c_{r-1} = c_r \text{ and } \sum c_j = \sum j c_j = 0,
        $
      \item
        \label{en:VrAstAstAstCriterion}
        For $p > 2$, $F \in V_r^{***}$ if and only if
        \[
          c_0 = c_1 = c_2 = 0 = c_{r-2} = c_{r-1} = c_r \text{ and } \sum c_j = \sum j c_j = \sum j(j-1) c_j =0.
        \]
    \end{enumerate}
  \end{lemma}

  \subsection{Some combinatorial Lemmas}

  The following lemma, known as Lucas' Theorem, is a key combinatorial lemma used throughout the paper.

  \begin{lem}[Lucas' Theorem]
    Let $r$ and $n$ be natural numbers and $r = r_0 + r_1 p + r_2 p^2 + \dotsb$ and $n = n_0 + n_1 p + n_2 p^2 + \dotsb$ be their $p$-adic expansions.
    Then
    \[
      \binom{r}{n}
      \equiv
      \binom{r_0}{n_0} \binom{r_1}{n_1} \binom{r_2}{n_2} \dotsm
      \mod p.
    \]
  \end{lem}

  \begin{lem}[Extension of {\cite[Lemmas 2.5 and 2.6]{BG}}]
    \label{lem:binomialsumaeq2}
    For $i = 0,1,2$, let $a$ in $\{ 1 + i, \ldots, p-1 + i \}$ be such that $r \equiv a \mod (p-1)$.
    Then
    \[
      \sum_{\substack{j \equiv a-i \mod (p-1)\\ 0 < j < r-i}} \binom{r}{j} \equiv
      \begin{cases}
        0 \mod p,                           & \text{\quad if $i = 0$ }\\
        a-r \mod p,                         & \text{\quad if $i = 1$ }\\
        \dfrac{(a-r)(a + r-1)}{2}   \mod p, & \text{\quad if $i = 2$.}
      \end{cases}
    \]
  \end{lem}
  \begin{proof}
    For $i = 0,1$, see \cite[Lemmas 2.5 and 2.6]{BG}.
    For $i = 2$, we apply induction on $r$.
    We have
    \[
      \binom{x + 2}{n}
      =
      \binom{x}{n-2}+
      2 \binom{x}{n-1}+
      \binom{x}{n}.
    \]
    Applying this identity for $i = 2$, and using the known cases ($i = 0,1$) and the induction hypothesis,
    \begin{align}
      \sum_{\substack{j \equiv a-2 \mod (p-1)\\ 0 < j < r-2}} \binom{r}{j}
      = & \sum_{\substack{j \equiv a-2 \mod (p-1)\\ 0 < j < r-2}} \binom{r-2}{j-2} + 2 \sum_{\substack{j \equiv a-2 \mod (p-1)\\ 0 < j < r-2}} \binom{r-2}{j-1} + \sum_{\substack{j \equiv a-2 \mod (p-1)\\ 0 < j < r-2}} \binom{r-2}{j}\\
      \equiv &  \frac{(a-r)(a + r-5)}{2} + 2(a-r) + 0 \mod p\\
      \equiv & \frac{(a-r)(a+r-5 +4)}{2} =\frac{(a-r)(a + r-1)}{2} \mod p.
      \qedhere
    \end{align}
  \end{proof}

\begin{rem*}
  More generally
  \[
    \sum_{\substack{j \equiv a-i \mod (p-1)\\ 0 < j < r-i}} \binom{r}{j} \equiv \binom{a}{i} - \binom{r}{i} \mod p.
  \]
  Since we do not go beyond $i=2$, we will not prove the above identity.
\end{rem*}

  \begin{lem}
    [Analog of {\cite[Lemma 2.5]{BG}} and {\cite[Proposition 2.8]{BGR}}]
    \label{lem:analog25bg}
    Let $p > 2$.
    For $i = 0, 1, \ldots, p-1$, if $r \equiv a \mod (p-1)$ and $a$ in $\{ i+1, \ldots ,p-1+i \}$, then we have
    \[
      \sum_{\substack{j \equiv a \mod (p-1)\\ i < j < r}} \binom{j}{i} \binom{r}{j}
      \equiv
      p \binom{r}{i} \frac{a-r}{a-i} \mod p^2.
    \]
  \end{lem}
  \begin{proof}
    By \cite[The latter statement of Lemma 2.5]{BG}
    \[
      \frac{1}{p} \sum_{0 < j \equiv a < r} \binom{r}{j}
      \equiv
      \frac{a-r}{a} \mod p.
    \]
    First replacing $r$ with $r-i$ and $a$ with $a-i$ yields
    \[
      \frac{1}{p} \sum_{j-i \equiv a-i} \binom{r-i}{j-i}
      \equiv
      \frac{a-r}{a-i} \mod p.
    \]
    Thus
    \[
      \frac{1}{p} \sum_{j \equiv a} \binom{j}{i} \binom{r}{j}
      =
      \frac{1}{p} \binom{r}{i} \sum_{j-i \equiv a-i} \binom{r-i}{j-i}
      \equiv
      \binom{r}{i} \frac{a-r}{a-i}  \mod p.
    \]

  \end{proof}

  \begin{cor}
    \label{cor:analog25bg}
    If $r \equiv p \mod (p-1)$ and $p^3 \mid p-r$, then for $i = 1, \ldots, p-1$, we have
    \[
      \sum_{\substack{j \equiv p \mod (p-1)\\ 1 < j < r}} \binom{j}{i} \binom{r}{j}
      \equiv
      0
      \mod p^3.
    \]
  \end{cor}
  \begin{proof}
    We first prove it for $i=1$ by observing
    \[
      \sum_{\substack{j \equiv p \mod (p-1)\\ 1 < j < r}} j \binom{r}{j} =     r  \sum_{\substack{j \equiv p \mod (p-1)\\ 1 < j < r}} \binom{r-1}{j-1}
    \]
    As $r \equiv p \mod p^3$,
    \[
      \sum_{\substack{j \equiv p \mod (p-1)\\ 1 < j < r}} j \binom{r}{j} \equiv     p  \sum_{\substack{j \equiv p \mod (p-1)\\ 1 < j < r}} \binom{r-1}{j-1} \mod p^3
    \]

    By \cite[Lemma 2.5]{BG}, with $r$ replaced by $r-1$,
    \[
      \sum_{\substack{j \equiv p \mod (p-1)\\ 1 < j < r}} \binom{r-1}{j-1}
    \equiv
    p  \frac{p-r}{p-1} \mod p^2.
    \]
    Multiplying by $p$ gives
    \[
      p \sum_{\substack{j \equiv p \mod (p-1)\\ 1 < j < r}} \binom{r-1}{j-1}
    \equiv
    p^2  \frac{p-r}{p-1} \mod p^3.
    \]

    As $r \equiv p \mod p^3$, the last expression above is zero $\mod p^3$.
    Hence
    \[
      \sum_{\substack{j \equiv p \mod (p-1)\\ 1 < j < r}} j \binom{r}{j} \equiv 0 \mod p^3.
    \]

    The proof for the general expression is similar as we use the condition $r \equiv p \mod p^3$ to show $\sum \binom{j}{i} \binom{r}{j} \equiv  \binom{p}{i} \sum \binom{r-i}{p-i} \mod p^3$ and then apply  \cref{lem:analog25bg}.
  \end{proof}

  \begin{lemma}
    \label{lem:analog7ag}
    Let $p > 2$.
    Let $r \equiv a \mod (p-1)$ with $a$ in $\{3, \ldots, p + 1 \}$.
    There are integers $\{\alpha_j : a \leq j < r \text{ and } j \equiv a \mod (p-1)\}$ such that
    \begin{enumerate}
       \item we have $\alpha_j \equiv \binom{r}{j} \mod p$, and
       \item for $n = 0, 1, 2$, we have $\sum_{j \geq n} \binom{j}{n} \alpha_j \equiv 0 \mod p^{4-n}$ and, for $n = 3$, we have
         \begin{itemize}
           \item
             if $a = 4, \ldots, p + 1$, then
             $\sum_{j \geq 3} \binom{j}{3} \alpha_j \equiv 0 \mod p$, and
           \item
             if $a = 3$, then
             $\sum_{j \geq 3} \binom{j}{3} \alpha_j \equiv \binom{r}{3} \mod p$.
         \end{itemize}
    \end{enumerate}
  \end{lemma}
  \begin{proof}
    If $r \leq ap$, then $\binom{r}{j} \equiv 0 \mod p$ for all $0 < j <r$ such that $j \equiv a \mod (p-1)$.
    Therefore, we can put $\alpha_j = 0$, and the proposition trivially holds true.

    Let $r > ap$.
    By \cref{lem:binomialsumaeq2} and noting that $j(j-1)(j-2)\binom{r}{j} = r(r-1)(r-2) \binom{r-3}{j-3}$ we see that
    \[
      \sum_{\substack{j \equiv a \mod (p-1) \\ 3 \leq j < r}} \binom{j}{3} \binom{r}{j}
      \equiv
      \begin{cases}
        \binom{r}{3} \mod p,	& \text{ \quad for } a = 3 \\
        0 \mod p,	& \text{ \quad otherwise.}
      \end{cases}
    \]
    This solves the case $n=3$.

    By \cref{lem:binomialsumaeq2} again, $\sum_{j \geq 2} \binom{j}{2} \binom{r}{j}$, $\sum_{j \geq 1} j \binom{r}{j}$ and $\sum_{j} \binom{r}{j} \equiv 0 \mod p$ for $j \equiv a \mod (p-1)$.
    Put
    \[
      s_0 = - p^{-1} \sum_{j } \binom{r}{j},
      \quad
      s_1 = - p^{-1} \sum_{j \geq 1} j \binom{r}{j}
      \quad \text{ and } \quad
      s_2 = - p^{-1} \sum_{j \geq 2} \binom{j}{2} \binom{r}{j}.
    \]
    and $\alpha_j = \binom{r}{j} +  p \delta_j$.

    Thus we have to solve for $3$ equations $(n=0,1,2)$ in $\delta_j's$. So we can take all but three $\delta_j$'s to be 0. Thus we need to choose $3$ $ j$'s wisely so that such a solution exists.

    There are $\delta_j$ such that
    \[
      \sum \alpha_j \equiv 0 \mod p^4, \quad
      \sum j \alpha_j \equiv 0 \mod p^3,
      \quad \text{ and } \quad
      \sum \binom{j}{2} \alpha_j \equiv 0 \mod p^2
    \]
    if and only if the following system of linear equations $(*)$ in the three unknowns $\delta_k$, $\delta_l$ and $\delta_m$ is solvable:
    \begin{align}
      1 & & 1 & & 1 & \equiv  s_0 \mod p^3,\\
      k & & l & & m & \equiv s_1 \mod p^2, \tag{$*$}\\
      \binom{k}{2} & & \binom{l}{2} & & \binom{m}{2} & \equiv s_2 \mod p.
    \end{align}
    It suffices to solve all equations modulo $p^3$.
    For this, we show that there are $k,l$ and $m$ in $\{ a, a + (p-1), \ldots, r - (p-1)\}$ such that the determinant of $(*)$ is invertible in $\Z / p^3 \Z$, or equivalently, that it is nonzero $\mod p$.

    Since $r > ap$, we can put $k = ap$.
    Then $(*)$ is modulo $p$ given by an upper triangular matrix whose upper left coefficient is $1$, and therefore its determinant equals that of its lower right $2 \x 2$-matrix
    \[
      \begin{pmatrix}
        l & m\\
        \binom{l}{2} & \binom{m}{2}
      \end{pmatrix}.
    \]
    Since this determinant is $(k - l)(l - m)(m - k)/2$, the system of linear equations $(*)$ can be made solvable by choosing $k$, $l$ and $m$ in different congruence classes.
  \end{proof}

  For \cref{prop:elimrEqAp}, we need a strengthened version of \cref{lem:analog7ag} for $a \geq 5$ when $r \equiv a \mod p$:

  \begin{lemma*}[\ref{lem:analog7ag}']
    Let $p > 2$.
    Let $r \equiv a \mod p(p-1)$ with $a$ in $\{5, \ldots, p + 1 \}$.
    There are integers $\{\alpha_j : a \leq j < r \text{ and } j \equiv a \mod (p-1)\}$ such that
    \begin{enumerate}
       \item $\alpha_j \equiv \binom{r}{j} \mod p^2$, and
       \item $\sum_{j \geq n} \binom{j}{n} \alpha_j \equiv 0 \mod p^{5-n}$ for $n = 0, 1, 2, 3$.
    \end{enumerate}
  \end{lemma*}
  \begin{proof}
    If $r \leq ap$, then necessarily $r = a$ and hence $\{j : a \leq j < r \text{ and } j \equiv a \mod (p-1)\} = \emptyset$ and the proposition trivially holds true.

    Let $r > ap$.
    By \cref{lem:analog25bg} for $i = 0, 1, 2, 3$ and noting that $r - a \equiv 0 \mod p$, we have $\sum_{j \geq 3} \binom{j}{3} \binom{r}{j}$, $\sum_{j \geq 2} \binom{j}{2} \binom{r}{j}$, $\sum_{j \geq 1} j \binom{r}{j}$ and $\sum_{j} \binom{r}{j} \equiv 0 \mod p^2$ (where the sums run over all $0 < j < r$ with $j \equiv a \mod (p-1)$).

    Therefore, we are in a situation analogous to that of the proof of \cref{lem:analog7ag}, and we can proceed analogously,
    putting
    \[
      s_0 = - p^{-2} \sum_{j \geq 0} \binom{r}{j},
      \quad
      s_1 = - p^{-2} \sum_{j \geq 1} j \binom{r}{j},
      \quad \text{ and } \quad
      s_2 = - p^{-2} \sum_{j \geq 2} \binom{j}{2} \binom{r}{j}
    \]
    and $\alpha_j = \binom{r}{j} +  p^2 \delta_j$
  \end{proof}

  \begin{lemma}
    \label{lem:analog7bg}
    Let $p \geq 5$.
    Let $a$ in $ \{ 4, \ldots, p + 1 \}$ such that $r \equiv a \mod (p-1)$.
    If $r \equiv a \mod p$, then there are integers $\{\beta_j : a-1 \leq j < r-1 \text{ and } j \equiv a-1 \mod (p-1)\}$ such that
      \begin{enumerate}
        \item we have $\beta_j \equiv \binom{r}{j} \mod p$, and
        \item for $n = 0, 1, 2, 3$, we have $\sum_{j \geq n} \binom{j}{n} \beta_j \equiv 0 \mod p^{4-n}$.
      \end{enumerate}
  \end{lemma}
  \begin{proof}
    If $r \leq (a-1)p$ and $r \equiv a \mod (p-1)$, then $\Sigma(r) = a$.
    Therefore, because $r \equiv a \mod p$, we have $r = a$.
    Hence, $\{j : a-1 \leq j < r-1 \text{ and } j \equiv a-1 \mod (p-1)\} = \emptyset$ and the proposition trivially holds true.

    Let $r > (a-1)p$.
    By \cref{lem:binomialsumaeq2} for $i = 1$ and noting that $r - a \equiv 0 \mod p$, we have $\sum_{j \geq 3} \binom{j}{3} \binom{r}{j}$, $\sum_{j \geq 2} \binom{j}{2} \binom{r}{j}$, $\sum_{j \geq 1} j \binom{r}{j}$ and $\sum_{j} \binom{r}{j} \equiv 0 \mod p$ (where the sums run over all $j < r-2$ with $j \equiv a-1 \mod (p-1)$) for $a \geq 5$.

    We now show the argument for $a=4$ and $j \equiv a - 1 = 3 \mod (p-1)$ and $n=3$.

    As $\beta_j \equiv \binom{r}{j} \mod p$ we see that:

    $\sum_{j \geq 3} \binom{j}{3} \beta_j \equiv \sum_{j \geq 3} \binom{j}{3} \binom{r}{j} \equiv \binom{r}{3} \sum_{j \geq 3} \binom{r-3}{j-3} \equiv \binom{r}{3} (1 + \sum_{j > 3} \binom{r-3}{j-3})\mod p$.

    If we let $r' = r-3 \equiv p = a' \mod (p-1)$ and $j' = j-3 \equiv p \mod (p-1)$ then by Lemma 1.6 and noting that $r \equiv 4 \mod p$, the sum $ \sum_{j > 3} \binom{r-3}{j-3} = \sum_{j' > 0} \binom{r'}{j'}\equiv  a' - r' = p - (r-3) \equiv (p-1) \equiv -1 \mod p$. Hence $\sum_{j \geq 3} \binom{j}{3} \binom{r}{j}$ vanishes.

    Therefore, we are in a situation analogous to that of the proof of \cref{lem:analog7ag}, and we can proceed analogously (where we put $k = (a-1)p$ instead of $k = ap$):
    Put
    \[
      s_0 = - p^{-1} \sum_{j \geq 0} \binom{r}{j},
      \quad
      s_1 = - p^{-1} \sum_{j \geq 1} j \binom{r}{j},
      \quad \text{ and } \quad
      s_2 = - p^{-1} \sum_{j \geq 2} \binom{j}{2} \binom{r}{j}
    \]
    and $\beta_j = \binom{r}{j} +  p \delta_j$.
    There are $\delta_j$ such that
    \[
      \sum \beta_j \equiv 0 \mod p^4, \quad
      \sum j \beta_j \equiv 0 \mod p^3,
      \quad \text{ and } \quad
      \sum \binom{j}{2} \beta_j \equiv 0 \mod p^2
    \]
    if the following system of linear equations $(*)$ in the three unknowns $\delta_k$, $\delta_l$ and $\delta_m$ is solvable:
    \begin{align}
      1 & & 1 & & 1 & \equiv  s_0 \mod p^3,\\
      k & & l & & m & \equiv s_1 \mod p^2, \tag{$*$}\\
      \binom{k}{2} & & \binom{l}{2} & & \binom{m}{2} & \equiv s_2 \mod p^1.
    \end{align}
    It suffices to solve all equations modulo $p^3$.
    For this, we show that there are $k,l$ and $m$ in $\{ a-1, a + (p-2), \ldots, r-p \}$ such that the determinant of $(*)$ is invertible in $\Z / p^4 \Z$, or equivalently, that it is nonzero $\mod p$.

    Because $r > (a-1)p$, we may put $k = (a-1)p$.
    Then $(*)$ is modulo $p$ given by an upper triangular matrix whose upper left coefficient is $1$, and therefore its determinant equals that of its lower right $2 \x 2$-matrix
    \[
      \begin{pmatrix}
        l & m\\
        \binom{l}{2} & \binom{m}{2}
      \end{pmatrix}.
    \]
    Since this determinant is $(k - l)(l - m)(m - k)/2$, the system of linear equations $(*)$ can be made solvable by choosing $k$, $l$ and $m$ in different congruence classes.
  \end{proof}

  For \cref{prop:elimrEqA}, we need a strengthened version of \cref{lem:analog7bg} when $r \equiv a \mod p^2$:

  \begin{lemma*}[\ref{lem:analog7bg}']
    Let $p \geq 5$.
    Let $a$ in $ \{ 5, \ldots, p + 1 \}$ such that $r \equiv a \mod (p-1)$.
    If $r \equiv a \mod p^2$, then there are integers $\{\beta_j : a-1 \leq j < r-1 \text{ and } j \equiv a-1 \mod (p-1)\}$ such that
    \begin{enumerate}
      \item we have $\beta_j \equiv \binom{r}{j} \mod p^2$, and
      \item for $n = 0, 1, 2, 3, 4$, we have $\sum_{j \geq n} \binom{j}{n} \beta_j \equiv 0 \mod p^{5-n}$.
    \end{enumerate}
  \end{lemma*}
  \begin{proof}
    If $r \leq (a-1)p$ and $r \equiv a \mod (p-1)$, then $\Sigma(r) = a$.
    Therefore, because $r \equiv a \mod p$, we have $r = a$.
    Hence, $\{j : a-1 \leq j < r-1 \text{ and } j \equiv a-1 \mod (p-1)\} = \emptyset$ and the proposition trivially holds true.

    Let $r > (a-1)p$.
    By \cite[last line of Lemma 3.3]{BhaLocConst} for $c = 0$ and $t = 2$, we have $\sum_{j \geq 4} \binom{j}{4} \binom{r}{j}$, $\sum_{j \geq 3} \binom{j}{3} \binom{r}{j}$, $\sum_{j \geq 2} \binom{j}{2} \binom{r}{j}$, $\sum_{j \geq 1} j \binom{r}{j}$ and $\sum_{j} \binom{r}{j} \equiv 0 \mod p^2$ (where the sums run over all $j < r-2$ with $j \equiv a-1 \mod (p-1)$).

    We now show the argument for $a=5$ and $j \equiv a - 1 = 4 \mod (p-1)$ and $n=4$.

    As $\beta_j \equiv \binom{r}{j} \mod p$ we see that:

    $\sum_{j \geq 4} \binom{j}{4} \beta_j \equiv \sum_{j \geq 4} \binom{j}{4} \binom{r}{j} \equiv \binom{r}{4} \sum_{j \geq 4} \binom{r-4}{j-4} \equiv \binom{r}{4} (1 + \sum_{j > 4} \binom{r-4}{j-4})\mod p$.

    If we let $r' = r-4 \equiv p = a' \mod (p-1)$ and $j' = j-4 \equiv p \mod (p-1)$ then by Lemma 1.6 and noting that $r \equiv 5 \mod p$, the sum $ \sum_{j > 4} \binom{r-4}{j-4} = \sum_{j' > 0} \binom{r'}{j'}\equiv  a' - r' = p - (r-4) \equiv (p-1) \equiv -1 \mod p$. Hence $\sum_{j \geq 4} \binom{j}{4} \binom{r}{j}$ vanishes.

    Therefore, we are in a situation similar to that of the proof of \cref{lem:analog7bg}, and we can proceed similarly,
    putting
    \[
      s_0 = - p^{-2} \sum_{j \geq 0} \binom{r}{j},
      \quad
      s_1 = - p^{-2} \sum_{j \geq 1} j \binom{r}{j},
      \quad \text{ and } \quad
      s_2 = - p^{-2} \sum_{j \geq 2} \binom{j}{2} \binom{r}{j}.
    \]
    and $\beta_j = \binom{r}{j} +  p^2 \delta_j$.
  \end{proof}

 \begin{lemma}
    \label{lem:analog7eg}
    Let $p \geq 5$.
    Let $a$ in $ \{ 6, \ldots, p \}$ such that $r \equiv a \mod (p-1)$.
    If $r \equiv a-1 \mod p$, then there are integers $\{\alpha'_j : a-2 \leq j < r-2 \text{ and } j \equiv a-2 \mod (p-1)\}$ such that
      \begin{enumerate}
        \item we have $\alpha'_j \equiv \binom{r}{j} - r \binom{r-1}{j} \mod p$, and
        \item for $n = 0, 1, 2, 3$, we have $\sum_{j \geq n} \binom{j}{n} \alpha'_j \equiv 0 \mod p^{4-n}$.
      \end{enumerate}
  \end{lemma}
  \begin{proof}
    We first use \cref{lem:binomialsumaeq2} for $a \geq 6$, yielding
    \begin{align}
        & \sum_{\substack{j \equiv a-2 \mod (p-1)\\ 0 < j < r-2}} \binom{j}{n} \binom{r}{j}\\
      = & \binom{r}{n}\sum_{\substack{j \equiv a-2 \mod (p-1)\\ 0 < j < r-2}} \binom{r-n}{j-n} \equiv \dfrac{(a-r)(a+r-1-2n)}{2}\binom{r}{n} \mod p.
    \end{align}
    As $r \equiv a-1 \mod p$, we have
    \[
      \sum_{\substack{j \equiv a-2 \mod (p-1)\\ 0 < j < r-2}} \binom{j}{n} \binom{r}{j}\equiv (a-1-n)\binom{a-1}{n} \equiv (a-1) \binom{a-2}{n} \mod p
    \]
    We also have
    \begin{align}
      & \sum_{\substack{j \equiv a-2 \mod (p-1)\\ 0 < j < r-2}} \binom{j}{n} r\binom{r-1}{j}\\
      = & r \binom{r-1}{n}\sum_{\substack{j \equiv a-2 \mod (p-1)\\ 0 < j < r-2}} \binom{r-1-n}{j-n} \equiv (a-r)r\binom{r-1}{n} \mod p
    \end{align}
    As $r \equiv a-1 \mod p$ we obtain
    \[
      \sum_{\substack{j \equiv a-2 \mod (p-1)\\ 0 < j < r-2}} \binom{j}{n} r\binom{r-1}{j} \equiv (a-1)\binom{a-2}{n}.
    \]
    For $n = 0, 1, 2, 3$, since $n < p$, we have
    \[
      \sum_{\substack{j \equiv a-2 \mod (p-1)\\ 0 < j < r-2}}  \bigg( \binom{r}{j} - r \binom{r-1}{j} \bigg) \equiv 0 \mod p.
    \]
    If $r \leq (a-2)p$ and $r \equiv a-1 \mod p$ and $r \equiv a \mod (p-1)$, then $r = p + a-1$ and $\{j : a-2 \leq j < r-2 \text{ and } j \equiv a-2 \mod (p-1)\} = \{ a-2 \}$, and, for $j = a-2$, we have
    \[
      \binom{r}{j} - r \binom{r-1}{j} = (a-1) - (p+a-1) \equiv 0 \mod p;
    \]
    therefore we may put $\alpha'_j = 0$.

    For the boundary case of $a=5$ and $n=3$, one can mimic the boundary cases from the previous lemmas. Using the fact that $\alpha'_j \equiv \binom{r}{j} - r \binom{r-1}{j} \mod p$, $r \equiv 4 \mod p$ and by \cref{lem:binomialsumaeq2} we see that:

    $ \sum_{j \geq 3} \binom{j}{3} \binom{r}{j} \equiv \binom{r}{3} \sum_{j \geq 3} \binom{r-3}{j-3} \equiv \binom{r}{3} (1 + \sum_{j > 3} \binom{r-3}{j-3}) \equiv 4 \mod p$.

    and

    $\sum_{j \geq 3} \binom{j}{3}r \binom{r-1}{j} \equiv r \binom{r-1}{3} \sum_{j \geq 3} \binom{r-4}{j-3} \equiv r \binom{r-1}{3}(1 + \sum_{j > 3} \binom{r-4}{j-3}) \equiv 4 \mod p$.

    Thus, we see that:

    $$\sum_{\substack{j \equiv 3 \mod (p-1)\\ 0 < j < r-2}} \binom{j}{3} \alpha'_j \equiv 0 \mod p$$.

    Let $r > (a-2)p$.
    Therefore, we are in a situation analogous to that of the proof of \cref{lem:analog7ag}, and we can proceed analogously:
    Put
    \begin{align}
      s_0 = - p^{-1} \sum_{j \geq 0} \binom{r}{j} - r \binom{r-1}{j},\\
      s_1 = - p^{-1} \sum_{j \geq 1} j (\binom{r}{j} - r \binom{r-1}{j}), \quad \text{ and }\\
      s_2 = - p^{-1} \sum_{j \geq 2} \binom{j}{2} (\binom{r}{j} - r \binom{r-1}{j}).
    \end{align}
    and $\alpha'_j = \binom{r}{j} - r \binom{r-1}{j} +  p \delta_j$.

    There are $\delta_j$ such that
    \[
      \sum \alpha'_j \equiv 0 \mod p^4, \quad
      \sum j \alpha'_j \equiv 0 \mod p^3
      \quad \text{ and } \quad
      \sum \binom{j}{2} \alpha'_j \equiv 0 \mod p^2
    \]
    if the following system of linear equations $(*)$ in the three unknowns $\delta_k$, $\delta_l$ and $\delta_m$ is solvable:
    \begin{align}
      1 & & 1 & & 1 & \equiv  s_0 \mod p^3,\\
      k & & l & & m & \equiv s_1 \mod p^2, \tag{$*$}\\
      \binom{k}{2} & & \binom{l}{2} & & \binom{m}{2} & \equiv s_2 \mod p.
    \end{align}
    It suffices to solve all equations modulo $p^3$.
    For this, we show that there are $k,l$ and $m$ in $\{ a-1, a + (p-2), \ldots, r-p \}$ such that the determinant of $(*)$ is invertible in $\Z / p^4 \Z$, or equivalently, that it is nonzero $\mod p$.

    Because $r > (a-2)p$, we may put $k = (a-2)p$.
    Then $(*)$ is modulo $p$ given by an upper triangular matrix whose upper left coefficient is $1$, and therefore its determinant equals that of its lower right $2 \x 2$-matrix
    \[
      \begin{pmatrix}
        l & m\\
        \binom{l}{2} & \binom{m}{2}
      \end{pmatrix}.
    \]
    Since this determinant is $(k - l)(l - m)(m - k)/2$, the system of linear equations $(*)$ can be made solvable by choosing $k$, $l$ and $m$ in different congruence classes.
  \end{proof}

  \begin{lemma}
    \label{lem:analog7cg}
    Let $p \geq 5$.
    Let $a$ in $ \{ 5, \ldots, p+1 \}$ such that $r \equiv a \mod (p-1)$.
    If $r \equiv a \mod p$, then there are integers $\{\gamma_j : a-2 \leq j < r-2 \text{ and } j \equiv a-2 \mod (p-1)\}$ such that
    \begin{enumerate}
      \item we have $ \gamma_j \equiv \binom{r}{j} \mod p$, and
      \item for $n = 0, 1, 2, 3$, we have $\sum_{j \geq n} \binom{j}{n} \gamma_j \equiv 0 \mod p^{4-n}$.
    \end{enumerate}
  \end{lemma}
  \begin{proof}
    If $r \leq (a-2)p$ and $r \equiv a \mod (p-1)$, then $\Sigma(r) = a$.
    Therefore, because $r \equiv a \mod p$, we have $r = a$.
    Therefore, $\{j : a-2 \leq j < r-2 \text{ and } j \equiv a-2 \mod (p-1)\} = \emptyset$ and the proposition trivially holds true.

    Let $r > (a-2)p$.
    Let us first show $\sum_{j \geq i} \binom{j}{i} \binom{r}{j} \equiv 0 \mod p$ in the edge case $i = 3$ and $a=5$.
    We have
    \[
      \sum_{j \geq 3} \binom{j}{3} \binom{r}{j} \equiv \binom{r}{3} \sum_{j \geq 3} \binom{r-3}{j-3}\mod p.
    \]
    We split up the latter sum as
    \[
      \sum_{j \geq 3} \binom{r-3}{j-3} = \sum_{j =3} \binom{r-3}{j-3}+ \sum_{j > 3} \binom{r-3}{j-3}= 1 + \sum_{j > 3} \binom{r-3}{j-3}.
    \]
    Letting $r' = r-3 \equiv p+1 = a' \mod (p-1)$ and $j' \equiv a'-2 \equiv (p-1) \mod (p-1)$, by \cref{lem:binomialsumaeq2},
    \begin{align}
      \sum_{j > 3} \binom{r-3}{j-3} & = \sum_{j' > 0, j' \equiv p-1} \binom{r'}{j'}\\
                                    & = \frac{(a' - r')(a' + r' -1 )}{2}\\
                                    & = \frac{(p+1 - (r-3))(p+1 + r-3 -1)}{2}=\frac{(p-1)(p+2)}{2} \equiv - 1 \mod p.
    \end{align}
    We conclude $\sum_{j \geq 3} \binom{j}{3} \binom{r}{j} \equiv 0 \mod p$.
    Therefore for $a \geq 5$, by \cref{lem:binomialsumaeq2} for $i = 2$ and as $a - r \equiv 0 \mod p$, we find all sums (running over all $j < r-2$ with $j \equiv a-2 \mod (p-1)$) given by $\sum_{j \geq 3} \binom{j}{3} \binom{r}{j}$, $\sum_{j \geq 2} \binom{j}{2} \binom{r}{j}$, $\sum_{j \geq 1} j \binom{r}{j}$ and $\sum_{j \geq 0} \binom{r}{j}$ to vanish $\mod p$.

    Therefore, we are in a situation analogous to that of the proof of \cref{lem:analog7ag}, and we can proceed analogously (where we put $k = (a-2)p$ instead of $k = ap$):
    Put
    \[
      s_0 = - p^{-1} \sum_{j \geq 0} \binom{r}{j},
      \quad
      s_1 = - p^{-1} \sum_{j \geq 1} j \binom{r}{j}
      \quad \text{ and } \quad
      s_2 = - p^{-1} \sum_{j \geq 2} \binom{j}{2} \binom{r}{j}.
    \]
    and $\gamma_j = \binom{r}{j} +  p \delta_j$.
    There are $\delta_j$ such that
    \[
      \sum \gamma_j \equiv 0 \mod p^4, \sum j \gamma_j \equiv 0 \mod p^3, \text{ and } \sum \binom{j}{2} \gamma_j \equiv 0 \mod p^2
    \]
    if the following system of linear equations $(*)$ in the three unknowns $\delta_k$, $\delta_l$ and $\delta_m$ is solvable:
    \begin{align}
      1 & & 1 & & 1 & \equiv  s_0 \mod p^3,\\
      k & & l & & m & \equiv s_1 \mod p^2, \tag{$*$}\\
      \binom{k}{2} & & \binom{l}{2} & & \binom{m}{2} & \equiv s_2 \mod p.
    \end{align}
    It suffices to solve all equations modulo $p^3$.
    For this, we show that there are $k,l$ and $m$ in $\{ a-2, a-2 + (p-1), \ldots, r-p-1 \}$ such that the determinant of $(*)$ is invertible in $\Z / p^4 \Z$, or equivalently, that it is nonzero $\mod p$.

    Because $r > (a-2)p$, we have $0 < (a-2)p < r$;
    we may, and will, therefore put $k = (a-2)p$.
    Then $(*)$ is modulo $p$ given by an upper triangular matrix whose upper left coefficient is $1$, and therefore its determinant equals that of its lower right $2 \x 2$-matrix
    \[
      \begin{pmatrix}
        l & m\\
        \binom{l}{2} & \binom{m}{2}
      \end{pmatrix}.
    \]
    Since this determinant is $(k - l)(l - m)(m - k)/2$, the system of linear equations $(*)$ can be made solvable by choosing $k$, $l$ and $m$ in different congruence classes.
  \end{proof}

  \begin{lemma}
    \label{lem:analog7dg}
    Let $p \geq 5$.
    Let $a = p$ and $r \equiv a \mod (p-1)$.
    \begin{enumerate}
      \item
        If $p^2 \mid p-r$, then there are integers $\{\gamma_j : p-1 \leq j < r-1 \text{ and } j \equiv 0  \mod (p-1)\}$ such that
        \begin{itemize}
          \item we have $ \gamma_j \equiv \binom{r}{j} \mod p^2$, and
          \item for $ 0 \leq n \leq 4$, we have  $\sum_{j \geq n} \binom{j}{n} \gamma_j \equiv 0 \mod p^{5-n}$.
        \end{itemize}
      \item
        If $p^2 \mid p-r$, then there are integers $\{\gamma_j : p \leq j < r \text{ and } j \equiv 1  \mod (p-1)\}$ such that
        \begin{itemize}
          \item
            we have $ \gamma_j \equiv \binom{r}{j} \mod p^2$, and
          \item
            for $0 \leq n \leq 4$, we have $\sum_{j \geq n} \binom{j}{n} \gamma_j \equiv 0 \mod p^{5-n}$.
        \end{itemize}
      \item
        If $p^3 \mid p-r$, then there are integers $\{\gamma_j : p \leq j < r \text{ and } j \equiv 1  \mod (p-1)\}$ such that
        \begin{itemize}
          \item
            we have $\gamma_j \equiv \binom{r}{j} \mod p^3$, and
          \item
            for $ 0 \leq n \leq 4$, we have $\sum_{j \geq n} \binom{j}{n} \gamma_j \equiv 0 \mod p^{6-n}$ , and
          \item  $\sum_{j \geq 5} \binom{j}{5} \gamma_j  \equiv
            \begin{cases}
              0 \mod p,                           & \text{\quad if $p \geq 7$ }\\
              1  \mod p, & \text{\quad if $p = 5$.}
            \end{cases}$
        \end{itemize}
    \end{enumerate}
  \end{lemma}
  \begin{proof}
    Ad (i): This is a special case of \cref{lem:analog7bg}'
    \bigskip

    Ad (ii):
    Similarly this follows from \cref{lem:analog7bg}', as follows:

    Let the integers $\{\beta_j : a-1 \leq j < r-1 \text{ and } j \equiv a-1 \mod (p-1)\}$ be as in \cref{lem:analog7bg}', that is
    \begin{enumerate}
      \item we have $\beta_j \equiv \binom{r}{j} \mod p^2$, and
      \item for $0 \leq n \leq 4$, we have $\sum_{j \geq n} \binom{j}{n} \beta_j \equiv 0 \mod p^{5-n}$.
    \end{enumerate}
    Since \( j \mapsto r - j \) for \( 0 \leq j \leq r \) is a bijection between
    \[
      \{ j : p \leq j < r \text{ and } j \equiv 1 \pmod{p - 1} \}
    \]
    and
    \[
      \{ j' : p - 1 \leq j' < r - 1 \text{ and } j' \equiv 0 \pmod{p - 1} \},
    \]
    the integers \( \gamma_j = \beta_{r-j} \) satisfy the conditions.
    \bigskip

    Ad (iii):
    We adapt \cref{lem:analog7ag} by referring to \cref{lem:analog25bg}:
    Let $a = p$ and $r \equiv a \mod (p-1)$.
    Because $p^3 \mid p-r$, we have $r > ap$.
    By \cref{cor:analog25bg}, we have $\sum_{j \geq 3} \binom{j}{3} \binom{r}{j}$, $\sum_{j \geq 2} \binom{j}{2} \binom{r}{j}$, $\sum_{j \geq 1} j \binom{r}{j}$ and $\sum_{j \geq 0} \binom{r}{j} \equiv 0 \mod p^3$. We note however that for $p=5$ one gets that  $\sum_{j \geq 5} \binom{j}{5} \binom{r}{j} \equiv 1 \mod p$ while for $p \geq 7$ we have that  $\sum_{j \geq 5} \binom{j}{3} \binom{r}{j} \equiv 0 \mod p$ as in \cite[Lemma 7.3]{BG}.

    Therefore, we are in a situation analogous to that of the proof of \cref{lem:analog7ag}, and we can proceed analogously:
    Put
    \[
      s_0 = - p^{-3} \sum_{j \geq 0} \binom{r}{j},
      \quad
      s_1 = - p^{-3} \sum_{j \geq 1} j \binom{r}{j}
      \quad \text{ and } \quad
      s_2 = - p^{-3} \sum_{j \geq 2} \binom{j}{2} \binom{r}{j}.
    \]
    and $\gamma_j = \binom{r}{j} +  p^3 \delta_j$.
    There are $\delta_j$ such that
    \[
      \sum \gamma_j \equiv 0 \mod p^6, \sum j \gamma_j \equiv 0 \mod p^5, \text{ and } \sum \binom{j}{2} \gamma_j \equiv 0 \mod p^4
    \]
    if the following system of linear equations $(*)$ in the three unknowns $\delta_k$, $\delta_l$ and $\delta_m$ is solvable:
    \begin{align}
      1 & & 1 & & 1 & \equiv  s_0 \mod p^3,\\
      k & & l & & m & \equiv s_1 \mod p^2, \tag{$*$}\\
      \binom{k}{2} & & \binom{l}{2} & & \binom{m}{2} & \equiv s_2 \mod p^1.
    \end{align}
    It suffices to solve all equations modulo $p^3$.
    For this, we show that there are $k,l$ and $m$ in $\{ a, a + (p-1), \ldots, r - (p-1) \}$ such that the determinant of $(*)$ is invertible in $\Z / p^6 \Z$, or equivalently, that it is nonzero $\mod p$.

    Because $r > ap$, we have $a \leq ap < r$;
    we may, and will, therefore put $k = ap$.
    Then $(*)$ is modulo $p$ given by an upper triangular matrix whose upper left coefficient is $1$, and therefore its determinant equals that of its lower right $2 \x 2$-matrix
    \[
      \begin{pmatrix}
        l & m\\
        \binom{l}{2} & \binom{m}{2}
      \end{pmatrix}.
    \]
    Since this determinant is $(k - l)(l - m)(m - k)/2$, the system of linear equations $(*)$ can be made solvable by choosing $k$, $l$ and $m$ in different congruence classes.
  \end{proof}


\section{The Jordan-Hölder series of $X_{r-2}$}
\label{sec:Xr-2}

Knowing under which conditions $X_{r-2} \supset X_{r-1}$ is a proper inclusion helps finding an additional Jordan-Hölder factor in $X_{r-2}$.
In contrast to the inclusion $X_{r-1} \supseteq X_r$, however, not always $X_{r-2} \neq X_{r-1}$ for $r$ sufficiently big.
To give an example, there is by \cref{lem:Xr2V2xXrpp} the natural epimorphism
\[
  X_{r''} \otimes V_2 \twoheadrightarrow X_{r-2}
\]
given by multiplication.
Let $r'' = r-2$.
For a natural number $r$, let
\[
  \Sigma(r)
  :=
  \text{ the sum of the digits of the $p$-adic expansion of $r$ }.
\]
Let $a$ in $ \{ 3, \ldots, p+1 \} $ such that $r \equiv a \mod (p-1)$.
If the sum of the digits of the $p$-adic expansion of $r - 2$ is equal to $a - 2$, then by \cref{prop:XrAst0IffSrMin} the left-hand side of
\[
  0 \to X_{r''}^* \to X_{r''} \to X_{r''} /X_{r''}^* \to 0
\]
vanishes.
In particular, if $a=3$, then the right-hand side is $X_{r''}/X_{r''}^* = V_1$.
Therefore,
\[
  X_{r''} \otimes V_2
  =
  V_1 \otimes V_2
  =
  V_1 \otimes \rD \oplus V_3
  \twoheadrightarrow
  X_{r-2}.
\]
That is, there is an epimorphism with only two Jordan-Hölder factors onto $X_{r-2}$.
Therefore, necessarily $X_{r-2} = X_{r-1}$.

This equality happens in other cases as well:
For $r = p+2, \ldots, p+(p-2)$, that is, $r = (p-1) + a$ for $a = 3, \ldots, p-1$, by \cref{prop:Xrm012ModAst}.(iii),
\[
  X_{r-2}/X_{r-2}^* = V_a/V_a^* = V_a
\]
where the equality on the right-hand side holds because $V_a$ is irreducible when $a = 3, ..., p-1$;
thus, $X_{r-2}^*$ has dimension $(p-1)+a - (a+1) = p-2 \leq p$;
in particular, it is irreducible.
We conclude $X_{r-2} = X_{r-1}$, because both have two Jordan-Hölder factors.

For $r = 2p-1$, by \cite[Proposition 3.3.(i)]{BG} already $X_{r-1} = V_{2p-1}$.
Therefore $V_{2p-1} = X_{r-1} \subseteq X_{r-2} \subseteq V_{2p-1}$.

By the next statement, $X_{r-2} = X_{r-1}$ if and only if $r = p^n + r_0$ where $r_0 = 2, ..., p-1$ and $n$ in $\N$.
(The preceding discussion showed this only for $r_0 = 2$ or $n=1$.)

\begin{lemma}
  \label{lem:XrInXr1InXr2}
  Let $p > 2$.
  Let $r$ in $\N$.
  We have $0 \subset X_r \subseteq X_{r-1} \subseteq X_{r-2}$ and
  \begin{itemize}
    \item the inclusion $X_r \subseteq X_{r-1}$ is an equality if and only if $r < p$, and
    \item for $p > 2$, the inclusion $X_{r-1} \subseteq X_{r-2}$ is an equality if and only if $r \leq p$ or $r = p^n + r_0$ where $r_0$ in $ \{2, \ldots, p-1 \}$ and $n > 0$.
  \end{itemize}
\end{lemma}
\begin{proof}
  For $X_r \subseteq X_{r-1}$ and when this inclusion is strict, see \cite[Lemma 4.1]{BG}.
  Note that $X_r = X_{r-1} = V_r$ for $r < p$.

  We have $X_{r-1} \subseteq X_{r-2}$, because
  $4X^{r-1}Y =
  \begin{pmatrix}
    1 &1\\
    0&1
  \end{pmatrix}
  X^{r-2}Y^2  -
  \begin{pmatrix}
    1 & -1\\
    0 & 1
  \end{pmatrix}
  X^{r-2}Y^2$.

  If $r < p$, then $V_r$ is irreducible.
  In particular, $X_{r-2} = X_{r-1}$.

  If $r = p$, then $X_{r-2} = X_{r-1}$ as
  \[
    (p - 1)X^{r-2} Y^2 = \sum_{\lambda \in \mathbb{F}_p^*} k^{p-2} (X + kY)^{r-1} Y.
  \]
  We may hence assume $r > p$.
  We have $X_{r-1} = X_{r-2}$ if and only if there are coefficients $C$, $c_0, \ldots, c_{p-1}$, $d_0, \ldots, d_{p-1}$ and $D$ in $\F_p$ such that
  \[
    X^2Y^{r-2}
    =
    C X^r + \sum c_k(kX + Y)^{r-1} X + \sum d_l (X + lY)^{r-1} Y + D Y^r.
    \tag{$*$}
  \]
  For $T \in \left\{ 0, \ldots, r-1 \right\}$, put
  \[
    C_T = \sum_{k = 1, \ldots, p-1} c_k k^T
    \text{ \quad and \quad }
    D_T = \sum_{l = 1, \ldots, p-1} d_l l^{r-1-T}.
  \]
  Comparing the coefficients on both sides of $(*)$, the above equation is satisfied if and only if
  \begin{itemize}
    \item $c_0 + C_0 + r D_1 = 0$ (by the coefficient of $X Y^{r-1}$),
    \item $d_0 + D_{r-1} + (r-1) C_{r-2} = 0$ (by the coefficient of $X^{r-1} Y$),
    \item $C + C_{r-1} = 0$ (by the coefficient of $X^r$),
    \item $D_0 + D = 0$ (by the coefficient of $Y^r$), and,
    \item by the coefficients of $X^{T + 1} Y^{r-(T + 1)}$ for $T = 1, \ldots, r-3$,
      \begin{equation}
        \binom{r-1}{T} C_T + \binom{r-1}{T + 1} D_{T+1} =
        \begin{cases}
          1, & \text{\quad for } T = 1,\\
          0, & \text{\quad for } T = 2, \ldots, r-3.
        \end{cases}
        \tag{$+$}
      \end{equation}
  \end{itemize}
  Because $\# \F_p^* = p-1$, for $0 < T', T'' < r-2$, if $T' \equiv T'' \mod (p-1)$, then $C_{T'} = C_{T''}$ and $D_{T'} = D_{T''}$.
  In particular, for every $T \equiv 1 \mod (p-1)$,
  \[
    \binom{r-1}{T} C_T + \binom{r-1}{T + 1} D_{T+1}
    =
    \binom{r-1}{T} C_1 + \binom{r-1}{T + 1} D_2.
  \]

  In the following, we will show that there are coefficients $c_1, \ldots, c_{p-1}$, and $d_1, \ldots, d_{p-1}$ in $\F_p$ such that $(+)$ is satisfied if and only if the stated conditions on $r$ are satisfied.
  That is, we show that if the stated conditions on $r$ are not satisfied, then $( + )$ cannot be satisfied, but if they are satisfied, then there are such coefficients.

  Because both matrices $(k^T)_{k, T = 1, \ldots, p-1}$ and $(l^{r-2-T})_{l, T = 1, \ldots, p-1}$ of the systems of $\F_p$-linear equations
  \[
    C_T = \sum_{k = 1, \ldots, p-1} c_k k^T
    \text{ and }
    D_{T+1} = \sum_{l = 1, \ldots, p-1} d_l l^{r-2-T}
    \quad \text{ for } T = 1, 2, \ldots, p-1
  \]
  are, up to permutations of columns, given by
  \[
    (k^T)_{k, T = 1, \ldots, p-1},
  \]
  and thus, up to a sign change, have Vandermonde determinant
  \[
    \prod_{k' < k'' \in {1, \ldots, p-1}} (k'' - k') \neq 0.
  \]
  we can freely choose $c_1$, \ldots, $c_{p-1}$ respectively $d_1$, \ldots, $d_{p-1}$ in $\mathbb{F}_p$ such that $C_1, C_2, \ldots, C_{p-1}$ respectively $D_1, D_2, \ldots, D_{p-1}$ satisfy Equations $(+)$ if and only if we can freely choose $C_1, C_2, \ldots, C_{p-1}$ and $D_1, D_2, \ldots, D_{p-1}$ in $\mathbb{F}_p$ that satisfy Equations $(+)$.

  Expand $r-1 = r_0 + r_1 p + r_2 p^2 + \dotsb$ with $r_0, r_1, \ldots \in \left\{ 0, \ldots, p-1 \right\}$.
  \begin{enumerate}[{Case} 1.]
    \item
      $r_0 = 0$.

      Then by Lucas' Theorem modulo $p$,
      \[
        \binom{r-1}{1} = r_0 = 0 \text{ \quad and \quad } \binom{r-1}{2} = \binom{r_0}{2} = 0.
      \]
      This equation contradicts that of $( + )$ for $T = 1$!
      Therefore $X_{r-2} \supset X_{r-1}$.
    \item
      $r_0 > 0$.

      \begin{enumerate}[{label*=\arabic*.}]
        \item
          There is a digit $r_j > 1$.
          Let $j$ be the minimal index of all digits with that property.

          For $T = p^j, p^j + p^j-1$ with $j \geq 1$, by Lucas' Theorem modulo $p$,
          \begin{alignat}{3}
            \binom{r-1}{p^j} = &  \binom{r_j}{1} & \text{\quad and \quad } & \binom{r-1}{p^j + 1} & = & \binom{r_j}{1} \binom{r_0}{1}\\
            \binom{r-1}{p^j + p^j-1} = &  \binom{r_j}{1} \binom{r_{j-1}}{p-1} \dotsm \binom{r_0}{p-1} & \text{\quad and \quad } & \binom{r-1}{2p^j} & = & \binom{r_j}{2}.
          \end{alignat}
          Because $p^j, p^j + p^j-1 \equiv 1 \mod (p-1)$,
          \begin{alignat}{4}
            r_j & C_1 +{} & r_j r_0 & D_1 & ={} & 0\\
            r_j \binom{r_{j-1}}{p-1} \dotsm \binom{r_0}{p-1} & C_1 +{} & \binom{r_j}{2} & D_1 & ={} & 0.
          \end{alignat}
          The determinant of the matrix $M$ of this system of equations is
          \[
            \norm{M}
            =
            r_j \cdot
            \begin{vmatrix}
              1 & r_0\\
              r_j \binom{r_{j-1}}{p-1} \dotsm \binom{r_0}{p-1} & \binom{r_j}{2}
            \end{vmatrix}
            =
            r_j \cdot
            \left[ \binom{r_j}{2} - r_j r_0 \binom{r_{j-1}}{p-1} \dotsm \binom{r_0}{p-1} \right].
          \]
          \begin{enumerate}[{label*=\arabic*.}]
            \item
              $j > 1$.

              By minimality of $j$, we have $\binom{r_{j-1}}{p-1} = 0$.
              Thence $\norm{M} = r_j \binom{r_j}{2} \not= 0$, that is,
              $C_1 = D_1 = 0$.
              This equation contradicts that of $( + )$ for $T = 1$!
              Therefore $X_{r-2} \supset X_{r-1}$.
            \item
              $j = 1$.

          \begin{enumerate}[{label*=\arabic*.}]
            \item
              $r_0 < p-1$.

              We have
              \[
                \norm{M} = r_1 \left[\binom{r_1}{2} - r_1 r_0 \binom{r_0}{p-1}\right].
              \]
              We obtain $\norm{M} = r_1 \binom{r_1}{2} \neq 0$ because $r_0 < p-1$.
              That is, $C_1 = D_1 = 0$.
              This equation contradicts that of $( + )$ for $T = 1$!
              Therefore $X_{r-2} \supset X_{r-1}$.
            \item
              $r_0 = p-1, r_1 < p-1$.

              We have
              \[
                \norm{M}
                = r_1 \left[\binom{r_1}{2} - r_1 (p-1)\right].
              \]
              We obtain $\norm{M} \equiv r_1^2 \frac{r_1 + 1}{2} \not\equiv 0$ because $r_1 < p-1$.
              That is, $C_1 = D_1 = 0$.
              This equation contradicts that of $( + )$ for $T = 1$!
              Therefore $X_{r-2} \supset X_{r-1}$.
              \item
                $r_0 = p-1, r_1 = p-1$.

                Let $T = p$.
                Then
                \[
                  \binom{r-1}{T} \equiv p-1 \equiv \binom{r-1}{1}
                  \quad \text{ and }
                  \binom{r-1}{T+1} \equiv 1 \equiv \binom{r-1}{2}.
                \]
                Because $T \equiv 1 \mod (p-1)$, we have $C_T = C_1$ and $D_{T+1} = D_2$;
                thus the equation $\binom{r-1}{T} C_T + \binom{r-1}{T + 1} D_{T+1} = 0$ in $( + )$ for $T = p$ contradicts $\binom{r-1}{1} C_1 + \binom{r-1}{2} D_2 = 1$ in $( + )$ for $T = 1$!
                Therefore $X_{r-2} \supset X_{r-1}$.
            \end{enumerate}
          \end{enumerate}
        \item
          All $r_1, r_2, \ldots \leq 1$.
          That is, $r-1$ is of the form $r-1 = r_0 + p^{n_1} + \dotsb + p^{n_m}$ for $0 < n_1 < \cdots < n_m$ in $\N$.

          For $T = 1$, we have
          \[
            \binom{r-1}{T} C_1 + \binom{r-1}{T + 1} D_2
            =
            r_0 C_1 + \binom{r_0}{2} D_2
            =
            1.
          \]
          \begin{enumerate}[{label*=\arabic*.}]
            \item
              We have $r_0 = p-1$.
              By Lucas' Theorem,
              \begin{itemize}
                \item
                  for $T = p^{n_1}$, we have, because $T \equiv 1 \mod (p-1)$,
                  \[
                    \binom{r-1}{T} C_1 + \binom{r-1}{T + 1} D_2
                    =
                    r_{n_1} C_1 + r_{n_1} r_0 D_2
                    =
                    C_1 + r_0 D_2 = 0;
                  \]
                \item
                  for $T = p^{n_1} + r_0$, then $T + 1 = 2p$ if $n_1 = 1$, and $T+1 = p^{n_1} + p$ if $n_1 > 1$.
                  Thus, if $n_1 = 1$ we have $\binom{r_1}{2} = 0$ because $r_1 \leq 1$, and if $n_1 > 1$, we have $\binom{r_1}{1} = 0$ because $r_1 = 0$.
                  Therefore, because $T \equiv 1 \mod (p-1)$,
                  \[
                    \binom{r-1}{T} C_1 + \binom{r-1}{T + 1} D_2
                    =
                    r_{n_1} \binom{r_0}{r_0} C_1
                    =
                    C_1
                    =
                    0,
                  \]
              \end{itemize}
              Therefore $C_1 = 0$, thus $D_2 = 0$.
              Thus
              \[
                r_0 C_1 + \binom{r_0}{2} D_2
                =
                1
              \]
              is impossible to satisfy.
            \item
              We have $r_0 < p-1$.
              \begin{enumerate}[{label*=\arabic*.}]
                \item
                  We have $m > 1$.
                  By Lucas' Theorem,
                  \begin{itemize}
                    \item
                      for $T = p^{n_1}$, we have, because $T \equiv 1 \mod (p-1)$,
                      \[
                        \binom{r-1}{T} C_1 + \binom{r-1}{T + 1} D_2
                        =
                        r_{n_1} C_1 + r_{n_1} r_0 D_2
                        =
                        C_1 + r_0 D_2 = 0;
                      \]
                    \item
                      for $T = p^{n_2} + p^{n_1} - 1$, we have $\binom{r - 1}{T} = 0$ because $\binom{r_0}{p-1} = 0$.
                      Therefore, because $T \equiv 1 \mod (p-1)$,
                      \[
                        \binom{r-1}{T} C_1 + \binom{r-1}{T + 1} D_2
                        =
                        r_{n_2} r_{n_1} D_2
                        =
                        D_2
                        =
                        0.
                      \]
                  \end{itemize}
                  Therefore $D_2 = 0$, thus $C_1 = 0$.
                  Thus
                  \[
                    r_0 C_1 + \binom{r_0}{2} D_2
                    =
                    1
                  \]
                  is impossible to satisfy.
                \item
                  We have $m = 1$.
                  In this case, $r$ satisfies the stated conditions for $X_{r-1} = X_{r-2}$, and we show, equivalently, that $(+)$ can be solved.
                  We have:
                  \begin{itemize}
                    \item the only $T$ in $ \left\{ 0, \ldots, r-2 \right\} $ such that $T \equiv 1 \mod (p-1)$ and $\binom{r-1}{T} \nequiv  0 \mod p$ are $T = p^0, p^{n_1}$,
                    \item the only $T$ in $ \left\{ 0, \ldots, r-2 \right\} $ such that $T \equiv 1 \mod (p-1)$ and $\binom{r-1}{T + 1} \nequiv  0 \mod p$ are $T = p^0, p^{n_1}$ for $r_0 > 1$, and, $T = p^{n_1}$ for $r_0 = 1$.
                  \end{itemize}
                  Therefore, to solve $( + )$, it suffices to choose $C_1, \ldots, C_{p-1}$ and $D_1, \ldots, D_{p-1}$ such that they resolve $(+)$ for $T = p^0$ and $p^{n_1}$;
                  that is, by Lucas' Theorem, such that for $T=1$,
                  \[
                    r_0 C_1 + \binom{r_0}{2} D_2 = 1
                  \]
                  and
                  \[
                    \binom{r-1}{p^{n_1}} C_{p^{n_1}} + \binom{r-1}{p^{n_1}+1} D_{p^{n_1}+1}
                    =
                    C_1 + r_0 D_2
                    =
                    0.
                  \]
                  That is, such that
                  \[
                    C_1 = - r_0 D_2
                    \quad \text{ and } \quad
                    D_2 = \frac{1}{\binom{r_0}{2} - r_0^2}
                    \tag{*}
                  \]
                  (where the denominator is nonzero because $r_0 \neq p-1$) and where
\[
                    C_2, \ldots, C_{p-1}
                    \quad \text{ and } \quad
                    D_2, \ldots, D_{p-1}
                  \]
                  are unrestricted.
                  We can therefore choose $c_1$, \ldots, $c_{p-1}$ respectively $d_1$, \ldots, $d_{p-1}$ such that $C_1$ respectively $D_1$ satisfy Equation $(*)$.

                  \qedhere
              \end{enumerate}
          \end{enumerate}
      \end{enumerate}
  \end{enumerate}
\end{proof}

  \subsection{Tensor Product Epimorphism}

  \begin{lemma}[Extension of {\cite[Lemma 3.6]{BG}}]
  \label{lem:Xr2V2xXrpp}
    Let $r \geq 2$.
    Put $r'' = r-2$.
    The map
    \begin{align}
      \phi \from X_{r''} \otimes V_2 & \to X_{r-2}\\
      f \otimes g & \mapsto f \cdot g
    \end{align}
    is an epimorphism of $\F_p[M]$-modules.
  \end{lemma}
  \begin{proof}
    By \cite[(5.1)]{G}, the map $\phi_{r'',2} \from V_{r''} \ox V_2 \twoheadrightarrow V_r$ defined by $u \ox v \mapsto uv$ is $M$-linear.
    Let $\phi$ be its restriction to the $M$-submodule $X_{r''} \ox V_2$.
    The $M$-submodule $X_{r''} \ox V_2$ is generated by $X^{r''} \ox X^2$, $X^{r''} \ox Y^2$ and $X^{r''} \ox XY$, which map to $X^r$, $X^{r-2}Y^2$ and $X^{r-1}Y$.
    Therefore the image of $\phi$ is included in $X_{r-2} \subseteq V_r$.
    Because $X^{r-2}Y^2$ generates $X_{r-2}$, surjectivity follows.
  \end{proof}

  \begin{cor}
    \label{cor:Xr2Iso}
    We have $\dim X_{r-2} \leq 3p + 3$.
    If $\dim X_{r-2} = 3p + 3$, then the epimorphism $\phi \from X_{r''} \otimes V_2 \twoheadrightarrow X_{r-2}$ is an isomorphism.
  \end{cor}
  \begin{proof}
    Because $\dim X_{r''} \leq p+1$ and $\dim V_2 = 3$, the left-hand side of the epimorphism $\phi \from X_{r''} \otimes V_2 \twoheadrightarrow X_{r-2}$ in \cref{lem:Xr2V2xXrpp} has dimension $\leq 3(p+1) = 3p+3$.
    Therefore its kernel is $0$.
  \end{proof}

  \begin{lemma}[Extension of {\cite[Lemma 3]{GG}}]
    \label{lem:Xr2Generators}
    Let $p > 2$ and $r \geq 2$.
    The $\F_p[M]$-module $X_{r-2}$ is generated by
    \[
      \{ X^r, Y^r, X^{r-1}Y,  X^2(j X + Y)^{r-2}, Y^2(X + kY)^{r-2}, XY(lX + Y )^{r-2} : j,k,l \in \mathbb{F}_p \}.
    \]
  \end{lemma}
  \begin{proof}
    We have $X_{r-2} = \langle X^{r-2}Y^2 \rangle$.
    We compute
    \begin{align}
      \begin{pmatrix}
        a & b\\
        c & d
      \end{pmatrix}
      X^{r-2}Y^2 & = (aX + cY)^{r-2}(bX + dY)^2\\
                  & = b^2 X^2(aX + cY)^{r-2} + d^2 Y^2(aX + cY)^{r-2} +2 bd XY(aX + cY)^{r-2}.
    \end{align}
    If $a =0$, then the right-hand side is in the span of $X^2Y^{r-2}, Y^r, XY^{r-1}$.
    If $c = 0$, then the right-hand side is in the span of $X^r, X^{r-2}Y^2, X^{r-1}Y$.
    If $ac \ne 0$, then the right-hand side is in the span of
    \[
      \{ X^r, Y^r, X^{r-1}Y,  X^2(j X + Y)^{r-2}, Y^2(X + kY)^{r-2}, XY(lX + Y )^{r-1}\}
    \]
    where $j,k,l \in \mathbb{F}_p$.
    We conclude as in \cite[Lemma 3]{GG}.
  \end{proof}

  \begin{cor}[Extension of {\cite[Lemma 3.5]{BG}}]
    \label{cor:dimXr2dimXr1}
    If $\dim X_{r-2} = 3p + 3$, then $\dim X_{r-1} = 2p + 2$ is maximal and $\dim X_r = \dim X_{r'} = \dim X_{r''} = p+1$ are maximal.
  \end{cor}
  \begin{proof}
    If $\dim X_{r-2} = 3p + 3$, then
    the left-hand side of the epimorphism $\phi \from X_{r''} \otimes V_2 \twoheadrightarrow X_{r-2}$ from \cref{lem:Xr2V2xXrpp} has dimension at least $3(p+1)$.
    Therefore, as $\dim V_2 = 3$, we have $\dim X_{r''} \geq p+1$ and thus $\dim X_{r''} = p+1$.

    That $\dim X_{r-1} = 2p + 2$ (that is, is maximal) is seen as in the proof of \cite[Lemma 3.5]{BG}.
    Therefore $\dim X_r = p + 1$ (that is, is maximal) by \cite[Lemma 3.5]{BG}.

    If $\dim X_{r-1} = 2p + 2$, then by the epimorphism $X_{r'} \ox V_2 \twoheadrightarrow X_{r-1}$, given by $f \otimes g \mapsto f \cdot g$, also $\dim X_{r'} = p + 1$ is maximal.
  \end{proof}

  \subsection{Singular Quotient of $X_r$, $X_{r-1}$ and $X_{r-2}$}
  \label{sec:singular-quotient-of-x-r-x-r-1-and-x-r-2}

  We generalize \cite[(4.5)]{G} by computing the quotients of $X_r$, $X_{r-1}$ and $X_{r-2}$ by its largest singular module:
  We denote by
  \[
    N
    =
    \{ \text{ all } m \text{ in } M \text{ such that } \det m = 0 \},
  \]
  all singular matrices and, for every module $V$ with an action of $M$, its largest singular submodule by
  \[
    V^*
    =
    \{ \text{ all } v \text{ in } V \text{ such that } n \cdot v = 0 \text{ for all } n \text{ in } N \}.
  \]

  \begin{prop}[Extension of {\cite[(4.5)]{G}}]
    \label{prop:Xrm012ModAst}
    Let $r > 0$.
    \begin{enumerate}
      \item
        For the unique $a$ in $\left\{ 1, \dots, p-1 \right\}$ such that $r \equiv a \mod (p-1)$,
        \[
          X_r / X_r^*
          =
          X_a/X_a^*
          =
          V_a.
        \]
      \item
        For the unique $a$ in $\left\{ 2, \dots, p \right\}$ such that $r \equiv a \mod (p-1)$,
        \[
          X_{r-1} / X_{r-1}^*
          =
          X_{a-1}/X_{a-1}^*
          =
          V_a/V_a^*
          =
          \begin{cases}
            V_a ,      & \text{\quad for $a = 2, \ldots,p-1$ }\\
            V_a / V_a^*, & \text{\quad for $a = p$ and $r \geq p$ }.
          \end{cases}
        \]
      \item
        For the unique $a$ in $\left\{ 3, \dots, p + 1 \right\}$ such that $r \equiv a \mod (p-1)$,
        \[
          X_{r-2} / X_{r-2}^*
          =
          X_{a-2}/X_{a-2}^*
          =
          V_a/V_a^*
          =
          \begin{cases}
            V_a ,        & \text{\quad for $a = 3, \ldots,p-1$ }\\
            V_a / V_a^*,   & \text{\quad for $a = p,p + 1$ and $r \geq p$ }.
          \end{cases}
        \]
    \end{enumerate}
  \end{prop}
  \begin{proof}
    \hfill
    \begin{enumerate}
      \item
        To prove $X_r / X_r^* = X_a / X_a^*$, we adapt the proof of \cite[(4.5)]{G} so that it readily generalizes to $X_{r-1}$:
        Let $U_r$ (denoted $X$ in \textit{op. cit.}) be the vector space of dimension $p + 1$ with basis vectors $x_0$, $x_1$, \ldots, $x_p$.
        Let $\rho_r \colon U_r \to X_r$ be given by
        \[
          x_0 \mapsto x^r
          \quad \text{ and } \quad
          x_i \mapsto (ix + y)^r.
        \]
        In particular,
        \[
          \rho_r x_i
          =
          (\rho_1 x_i)^r.
        \]
        For every nonzero $v$ in $X_1 = V_1$, there is a unique $\gamma$ in $\F_p$ and a unique $i$ in $ \{ 0, 1, \ldots, p \} $ such that $v = \gamma \rho_1(x_i)$.
        In particular, for every $v = m \cdot \rho_1(x_i)$ for $i = 0$, $1$, \ldots, $p$.
        Let $M$ act on $U_r$ by
        \[
          m \cdot x_i
          =
          \begin{cases}
            0,	& \text{ \quad if } m \cdot \rho_1(x_i) = 0 \\
            \gamma^r x_j,	& \text{ \quad if } m \cdot \rho_1(x_i) = \gamma \rho_1( x_j).
          \end{cases}
        \]
        With this action of $M$, the proof of \cite[(4.5)]{G} shows $\rho_r$ is $M$-linear.
        Also, $\# \F_p^* = p-1$, the $\F_p[M]$-modules $U_r$ and $U_a$ are isomorphic.
        We claim
        \[
          \rho_a^{-1}(X_a^*)
          \cong
          \rho_r^{-1}(X_r^*),
        \]
        that is:
        For every $n$ in $N$ and $x$ in $U_a = U_r$, we have $n \cdot \rho_a (x) = 0$ if and only if $n \cdot \rho_r (x) = 0$.

        To see this, note that the image of $n$ on $V_1$ is at most one-dimensional, $\dim (n V_1) \leq 1$, that is, there is $v_n$ in $V_1$ such that for every $v$ in $V_1$ there is $\gamma_v$ in $\F_p$ such that $n \cdot v = \gamma_v v_n$.
        Therefore, by definition of the $M$-linear homomorphism $\rho_r$, for every $i = 0,1, \ldots, p$ there is $\gamma_i$ in $\F_p$ such that
        \[
          n \cdot \rho_r(x_i)
          =
          \gamma_i^r v_n^r.
        \]
        Writing $x = \sum_{i} b_i x_i$, therefore
        \[
          n \cdot \rho_r(x)
          =
          \left[ \sum b_i \gamma_i^r \right] v_n^r.
        \]
        Similarly,
        \begin{align}
          n \cdot \rho_a(x)
        & =
        \left[ \sum b_i \gamma_i^a \right] v_n^a
        \end{align}
        Because $r \equiv a \mod (p-1)$ and $\# \F_p^* = p-1$,
        \[
          \sum b_i \gamma_i^a
          =
          \sum b_i \gamma_i^r.
        \]
        Therefore,
        \[
          n \cdot \rho_r(x) = 0
          \quad \text{ if and only if } \quad
          n \cdot \rho_a(x) = 0,
        \]
        that is,
        \[
          \rho_r^{-1}(X_r^*)
          \cong
          \rho_a^{-1}(X_a^*).
        \]
        Therefore
        \[
          X_r / X_r^*
          \opiso
          U_r / \rho_r^{-1}(X_r^*)
          \cong
          U_a / \rho_a^{-1}(X_a^*)
          \iso
          X_a / X_a^*.
        \]
        (As observed in the proof of \cite[(4.5)]{G}, indeed $X_a^* = 0$ because $a < p$ and $V_a$ is irreducible.)
        \label{en:Glov45r}
      \item
        To prove $X_{r-1} / X_{r-1}^* = X_{a-1} / X_{a-1}^*$, we adapt the above proof:
        Put $r' = r-1$.
        \begin{itemize}
          \item
            Let $U_{r-1} = U_{r'} \otimes V_1$ be the $\F_p[M]$-module given by the tensor product of the $\F_p[M]$-modules $U_{r'}$ and $V_1$:
            If $x_0$, $x_1$, \ldots, $x_p$ is a basis of $U_{r'}$ and $v'$ and $v''$ one of $V_1$, then the basis vectors of $U_{r-1}$ are $x_0 \otimes v'$, \ldots, $x_p \otimes v'$ and $x_0 \otimes v''$, \ldots, $x_p \otimes v''$.
            (NB: We follow the abuse of notation that distinguishes $X_{r-1}$ from $X_{r'}$ for $r' = r - 1$.)
          \item
            let $\rho_{r-1} \from U_{r-1} \to X_{r-1}$ be the composition
            \[
              U_{r-1}
              =
              U_{r'} \otimes V_1
              \overset{\rho_{r'} \otimes \id}{\longrightarrow}
              X_{r'} \otimes V_1
              \to
              X_{r-1}
            \]
            where the right-hand side homomorphism sends $f \otimes g$ to $f \cdot g$.
        \end{itemize}
        Because the $\F_p[M]$-modules $U_{r'}$ and $U_{a'}$ are isomorphic, so are $U_{r-1}$ and $U_{a-1}$.
        We claim
        \[
          \rho_{a-1}^{-1}(X_{a-1}^*)
          \cong
          \rho_{r-1}^{-1}(X_{r-1}^*),
        \]
        that is:
        For every $n$ in $N$ and $x$ in $U_{a-1} = U_{r-1}$, we have $n \cdot \rho_{a-1} (x) = 0$ if and only if $n \cdot \rho_{r-1} (x) = 0$.
        Because the image of $n$ on $V_1$ is at most one-dimensional, $\dim (n V_1) \leq 1$, there is $v_n$ in $V_1$ such that
        \begin{itemize}
          \item for every $i = 0,1, \ldots, p$ there is $\gamma_i$ in $\F_p$ such that
            \[
              n \cdot x_i
              =
              \gamma_i^{r'} v_n^{r'}, \quad \text{ and }
            \]
          \item there are $\gamma'$ and $\gamma''$ in $\F_p$ such that $n \cdot v' = \gamma' v_n$ and $n \cdot v'' = \gamma'' v_n$.
        \end{itemize}
        Writing $x = \sum_{i} b'_i x_i \otimes v' + \sum_{i} b''_i x_i \otimes v''$, therefore
        \begin{align}
          n \cdot \rho_{r-1}(x)
        & =
        \left[ \gamma' \sum b'_i \gamma_i^{r'} \right] v_n^{r'} \cdot v_n
        +
        \left[ \gamma'' \sum b''_i \gamma_i^{r'} \right] v_n^{r'} \cdot v_n\\
        & =
        \left[ \sum_i (\gamma' b'_i  + \gamma'' b''_i) \gamma_i^{r'} \right] v_n^{r}
        \end{align}
        Similarly,
        \begin{align}
          n \cdot \rho_{a-1}(x)
        & =
        \left[ \sum_i (\gamma' b'_i  + \gamma'' b''_i) \gamma_i^{a'} \right] v_n^a.
        \end{align}
        Because $r' \equiv a' \mod (p-1)$ and $\# \F_p^* = p-1$,
        \[
          \sum_i (\gamma' b'_i  + \gamma'' b''_i) \gamma_i^{r'}
          =
          \sum_i (\gamma' b'_i  + \gamma'' b''_i) \gamma_i^{a'}
        \]
        Therefore,
        \[
          n \cdot \rho_{r-1}(x) = 0
          \quad \text{ if and only if } \quad
          n \cdot \rho_{a-1}(x) = 0,
        \]
        that is,
        \[
          \rho_{r-1}^{-1}(X_{r-1}^*)
          \cong
          \rho_{a-1}^{-1}(X_{a-1}^*).
        \]
        Therefore
        \[
          X_{r-1} / X_{r-1}^*
          \opiso
          U_{r-1} / \rho_{r-1}^{-1}(X_{r-1}^*)
          \cong
          U_{a-1} / \rho_{a-1}^{-1}(X_{a-1}^*)
          \iso
          X_{a-1} / X_{a-1}^*.
        \]
        \label{en:Glov45r-1}
      \item
        To prove $X_{r-2} / X_{r-2}^* = X_{a-2} / X_{a-2}^*$, we adapt the above proof:
        Put $r'' = r-2$.
        \begin{itemize}
          \item
            Let $U_{r-2} = U_{r''} \otimes V_2$ be the $\F_p[M]$-module given by the tensor product of the $\F_p[M]$-modules $U_{r''}$ and $V_2$:
            If $x_0$, $x_1$, \ldots, $x_p$ is a basis of $U_{r'}$ and $v_0$, $v_1$ and $v_2$ one of $V_2$, then the basis vectors of $U_{r-2}$ are $x_0 \otimes v_1$, \ldots, $x_p \otimes v_1$, $x_0 \otimes v_1$, \ldots, $x_p \otimes v_1$ and $x_0 \otimes v_2$, \ldots, $x_p \otimes v_2$.
            (NB: We follow the abuse of notation that distinguishes $X_{r-2}$ from $X_{r''}$ for $r' = r - 2$.)
          \item
            let $\rho_{r-2} \from U_{r-2} \to X_{r-2}$ be the composition
            \[
              U_{r-2}
              =
              U_{r''} \otimes V_2
              \overset{\rho_{r''} \otimes \id}{\longrightarrow}
              X_{r''} \otimes V_2
              \to
              X_{r-2}
            \]
            where the right-hand side homohomomorphism sends $f \otimes g$ to $f \cdot g$.
        \end{itemize}
        Because the $\F_p[M]$-modules $U_{r''}$ and $U_{a''}$ are isomorphic, so are $U_{r-2}$ and $U_{a-2}$.

        Let $n$ in $N$ and $x$ in $U_{a-2} = U_{r-2}$.
        It suffices to prove that $n \cdot \rho_{a-2} (x) = 0$ if and only if $n \cdot \rho_{r-2} (x) = 0$, and we will prove this as above:
        Because the image of $n$ on $V_1$ is at most one-dimensional, $\dim (n V_1) \leq 1$, there is $v_n$ in $V_1$ such that
        \begin{itemize}
          \item  by definition of the $M$-action and $\rho_r$ on $U_r$, for every $i = 0,1, \ldots, p$ there is $\gamma_i$ in $\F_p$ such that
            \[
              n \cdot \rho_{r''} (x_i)
              =
              \gamma_i^{r''} v_n^{r''}, \quad \text{ and }
            \]
          \item by definition of the $M$-action on $V_2$ with basis $v_0 = x^2$, $v_1 = xy$ and $v_2 = y^2$, there are $\Gamma_0$, $\Gamma_2$ and $\Gamma_1'$, $\Gamma_1''$ in $\F_p$ such that
            \[
              n \cdot v_0
              =
              \Gamma_0^2 v_n^2,
              \quad
              n \cdot v_1
              =
              \Gamma_1' \Gamma_1'' v_n^2,
              \quad \text{ and } \quad
              n \cdot v_2
              =
              \Gamma_2^2 v_n^2.
            \]
        \end{itemize}
        Writing $x = \sum_{i = 0, 1, \ldots, p, j = 0,1,2} b_{i,j} x_i \otimes v_j$, therefore
        \begin{align}
          & n \cdot \rho_{r-2}(x)
          \\
          = &
          \quad \left[ \sum b_{i,0} \gamma_i^{r''} \Gamma_0^2 \right] v_n^{r''} \cdot v_n^2
          +
          \left[ \sum b_{i,1} \gamma_i^{r''} \Gamma_1' \Gamma_1'' \right] v_n^{r''} \cdot v_n^2\\
            &
          +
          \left[ \sum b_{i,2} \gamma_i^{r''} \Gamma_2^2 \right]
          v_n^{r''} \cdot v_n^2\\
          = &
          \left[ \sum_i \gamma_i^{r''}
          ( \Gamma_0^2 b_{i,0}
          +
          \Gamma_1'\Gamma_1'' b_{i,1}
          +
          \Gamma_2^2 b_{i,2}) \right]
          v_n^r
        \end{align}
        Similarly,
        \[
          n \cdot \rho_{a-2}(x)
          =
          \left[ \sum_i \gamma_i^{a''}
          (\Gamma_0^2 b_{i,0}
          +
          \Gamma_1'\Gamma_1'' b_{i,1}
          +
          \Gamma_2^2 b_{i,2}) \right]
          v_n^a.
        \]
        Because $r'' \equiv a'' \mod (p-1)$ and $\# \F_p^* = p-1$, the result follows as above.
        \qedhere
    \end{enumerate}
  \end{proof}

  \begin{lem}[Jordan-Hölder series of $X_r$]
    \label{lem:JHXr}
    There is a short exact sequence
    \[
      0 \to X_r^* \to X_r \to X_r/X_r^* \to 0.
    \]
    Let $r \geq p$.
    For $a$ in $\left\{ 1, \ldots, p-1 \right\}$ such that $r \equiv a \mod (p-1)$,
    \begin{itemize}
      \item we have $X_r/X_r^* = V_a$, and
      \item $\dim X_r = p + 1$ if and only if $X_r^* \neq 0$;
        if so, then $X_r^* = V_{p-a-1} \otimes \rD^a$.
    \end{itemize}
  \end{lem}
  \begin{proof}
    We have $\dim X_r \leq p + 1$ and $X_r / X_r^* = X_a/X_a^* = V_a$ by \cref{prop:Xrm012ModAst}.(i).
    By \cite[Lemma 4.6]{BG}, either $X_r^* = V_{p-a-1} \otimes \rD^a$ (if and only if $\dim X_r = p+1$) or $X_r^* = 0$ (if and only if $\dim X_r < p+1$).
  \end{proof}

  \begin{lemma}[Extension of {\cite[Lemma 4.7]{BG}}]
    \label{lem:analog47}
    Let $p \geq 3$ and $r \geq p$.
    Let $a$ in $\{ 1, \ldots, p-1 \}$ such that $r \equiv a \mod (p-1)$.
    \begin{enumerate}
      \item If $a = 1$, then $X_r^{*} = X_r^{**}$ if and only if $p \mid  r$, and $X_r^{**} = X_r^{***}$.
      \item If $a = 2$, then $X_r^{*} = X_r^{**}$, and $X_r^{**} = X_r^{***}$ if and only if $r \equiv 0,1 \mod p$.
      \item If $a \geq  3$, then $X_r^* = X_r^{**} = X_r^{***}$.
    \end{enumerate}
  \end{lemma}
  \begin{proof}
    Regarding the equality between $X_r^*$ and $X_r^{**}$:
    If $a = 1$, then by \cite[Lemma 3.1]{BG}, we have $X_r^{*} = X_r^{**}$ if and only if $p | r$.
    If $a \geq 2$, then $X_r^* = X_r^{**}$ by \cite[Lemma 4.7]{BG}.

    Regarding the equality between $X_r^{**}$ and $X_r^{***}$:
    If $X_r^{**} / X_r^{***} \ne 0$, then $X_r^{**} /X_r^{***} = V_{p-a-1} \ox \rD^a$ by \cref{lem:JHXr}.
    By \cref{lem:VrAstQuotients}.(iii), we find that $V_{p-a-1} \ox \rD^a$ is a $\Gamma$-submodule of $V_r^{**}/V_r^{***}$ if and only if $a = 2$.
    (Beware of the shift from $a$ to $a+p-1$ for $a = 1, \ldots ,4$!)
    Therefore, if $a \neq 2$, then $X_r^{**} / X_r^{***} = 0$.

    For $a = 2$, recall the polynomial in the proof of \cite[Lemma 3.1.(i)]{BG}:
    \[
      F(X,Y)
      =
      \sum_{j = 1, \ldots, r-1} \binom{r}{j} \sum_{k \in \F_p} k^{r-j} X^{r-j} Y^j
      \equiv
      \sum_{\substack{j = 1, \ldots, r-1\\ j \equiv 2 \mod (p-1)}} - \binom{r}{j} X^{r-j} Y^j
      \mod p.
    \]
    It is in $X_r^{**}$ by \cref{lem:VrAstCriteria} and \cref{lem:binomialsumaeq2}.
    If $r \not\equiv 0,1 \mod p$, then $\binom{r}{2} = r(r-1)/2 \not\equiv 0$;
    therefore, by the same token, $F(X,Y)$ is not in $X_r^{***}$.
    Thus $X_r^{**}/X_r^{***} \neq 0$.

    If $r \equiv 0 \mod p$, then we follow the proof of \cite[Lemma 3.1.(ii)]{BG}:
    Write $r = p^n u$ for $n \geq 1$ and $p \nmid u$.
    Let $\iota \colon X_u \to X_r$ be the isomorphism
    \[
      f(X,Y) \mapsto f(X^{p^n},Y^{p^n}) = f(X,Y)^{p^n}
    \]
    that restricts to
    \[
      X_u^* \iso X_r^*.
    \]
    Therefore $X_r^* = \iota(X_u^*) = X_r^{* \ldots *}$ with $p^n$-times $*$, that is, in $\theta | f$ in $X_r$ if and only if $\theta^{p^n} | f$.
    In particular, $X_r^* = X_r^{**} = X_r^{***}$.

    If $r \equiv 1 \mod p$, then $X_{r-1}^{**} = 0$ by \cite[Proof of Proposition 5.4]{BG}.
    In particular $X_r^{**} = X_r^{***}$.
  \end{proof}

  \subsection{Jordan-Hölder series of $X_{r-2}$}
  \label{sec:jordan-holder-series-of-x-r-and-x-r-2}

  To compute the Jordan-Hölder series of $Q := V_r / (V_r^{***} + X_{r-2})$, it would help to know that of $X_{r-2}$.
  However, to this end, the exact Jordan-Hölder series of $X_{r-2}$ will turn out dispensable, but that of $X_{r''} \otimes V_2 \twoheadrightarrow X_{r-2}$ sufficient.
  Therefore, the following \cref{prop:JHXrppxV2} will serve as fulcrum of all subsequent computations of the Jordan-Hölder factors of $Q$:

  \begin{prop}
  \label{prop:JHXrppxV2}
    Let $r \geq p+1$.
    Let $r \equiv a \mod (p-1)$ for $a$ in $\left\{ 3, \ldots, p + 1 \right\}$.
    Put $r'' = r-2$.
    We have the following short exact sequences (where, by convention, $V_i = 0$ for $i < 0$):
    \begin{enumerate}
      \item
        If $X_{r''}^* \neq 0$,
        \begin{itemize}
          \item
            For $a = 3$,
            \begin{align}
              0 & \to (V_{2p-1} \ox \rD) \op (V_{p-4} \ox \rD^3)\\
                & \to X_{r''} \ox V_2\\
                & \to (V_1 \ox \rD) \op V_3 \to 0
            \end{align}
            where $V_{2p-1}$ has Jordan-Hölder series $V_{p-2} \otimes \rD$, $V_1$ and $V_{p-2} \otimes \rD$.
          \item
            For $a$ in $\{ 4, \ldots, p-1 \}$,
            \begin{align}
              0 & \to (V_{p-a + 3} \ox \rD^{a-2}) \op (V_{p-a + 1} \ox \rD^{a-1}) \op (V_{p-a-1} \ox \rD^a)\\
                & \to X_{r''} \ox V_2\\
                & \to (V_{a-4} \ox \rD^2) \op (V_{a-2} \otimes \rD) \op V_a \to 0.
            \end{align}
          \item
            For $a = p$,
            \begin{align}
              0 & \to (V_3 \otimes \rD^{p-2}) \op (V_1 \otimes \rD^{p-1})\\
                & \to X_{r''} \ox V_2\\
                & \to (V_{p-4} \otimes \rD^2) \op V_{2p-1} \to 0
            \end{align}
            where $V_{2p-1}$ has Jordan-Hölder series $V_{p-2} \otimes \rD$, $V_1$ and $V_{p-2} \otimes \rD$.
          \item
            For $a = p + 1$,
            \begin{equation}
              0 \to V_2 \otimes \rD^{p-1} \to X_{r''} \ox V_2 \to V_{3p-1} \to 0
            \end{equation}
            where $V_{3p-1} = (V_{p-1} \ox \rD) \op U$ and $U$ has successive semisimple Jordan-Hölder factors $V_{p-3} \ox \rD^2$, $(V_0 \otimes \rD) \op V_2$ and $V_{p-3} \ox \rD^2$.
        \end{itemize}
      \item
        If $X_{r''}^* = 0$, then all summands on the left-hand sides vanish.
    \end{enumerate}
  \end{prop}
  \begin{proof}
    If $X_{r''}^* \neq 0$, then by \cref{lem:JHXr}, for the unique $a'' \in \left\{ 1, \ldots, p-1 \right\}$ such that $r'' = r-2 \equiv a'' \mod (p-1)$, (that is, $a'' = a-2$ for the unique $a \in \left\{ 3, \ldots, p + 1 \right\}$ such that $r \equiv a \mod (p-1)$),
    \[
      0 \to V_{p-a''-1} \ox \rD^{a''} \to X_{r''} \to V_{a''} \to 0. \tag{$*$}
    \]
    By flatness of the $\F_p[M]$-module $V_2$,
    \[
      0 \to (V_{p-a''-1} \ox \rD^{a''}) \ox V_2 \to X_{r''} \ox V_2 \to V_{a''} \ox V_2 \to 0
    \]
    We regard the left-hand side of the short exact sequence, that is, $(V_{p-a''-1} \ox \rD^{a''}) \ox V_2$:
    \begin{itemize}
      \item
        if $a'' = 1$, then by \cref{cor:JHVxV2},
        \[
          V_2 \ox V_{p-a''-1} = V_{2p-1} \op V_{p-4} \ox \rD^2;
        \]
      \item
        if $a'' = 2, \ldots, p-3$, then by \cref{lem:JHTensor}.\ref{en:JHTensor0p-1},
        \begin{align}
          V_2 \ox V_{p-a''-1} & = [V_1 \ox V_{p-a''}] \op V_{p-a''-3} \ox \rD^2\\
                              & = [(V_{p-a''-1} \otimes \rD) \op V_{p-a'' + 1}] \op V_{p-a''-3} \ox \rD^2;
        \end{align}
      \item
        if $a'' = p-2$, that is, $p-a''-1 = 1$, then $V_2 \ox V_1 = (V_1 \otimes \rD) \op V_3$ by \cref{lem:JHTensor}.\ref{en:JHTensor0p-1};
      \item
        if $a'' = p-1$, that is, $p-a''-1 = 0$, then $V_2 \ox V_0 = V_2$.
    \end{itemize}

    We regard the right-hand side of the short exact sequence, that is, $ V_{a''} \ox V_2$:
    \begin{itemize}

      \item
        if $a'' = 1$, then $V_1 \ox V_2 = (V_1 \otimes \rD) \op V_3$ by \cref{lem:JHTensor}.\ref{en:JHTensor0p-1}.
      \item
        if $a'' = 2, \ldots,p-3$, then by \cref{lem:JHTensor}.\ref{en:JHTensor0p-1} (where we recall $V_{-1} = 0$),
        \begin{align}
          V_2 \ox V_{a''} & = [V_1 \ox V_{a'' + 1}] \op V_{a''-2} \ox \rD^2\\
                          & = [(V_{a''} \otimes \rD) \op V_{a'' + 2}] \op V_{a''-2} \ox \rD^2.
        \end{align}

      \item
        if $a'' = p-2$, then, like for $a'' = 1$ on the left-hand side of the short exact sequence,
        \[
          V_2 \ox V_{p-2}
          = (V_1 \ox V_{p-1}) \op V_{p-4} \ox \rD^2
          = (V_{2p-1}) \op V_{p-4} \ox \rD^2,
        \]
        where $V_{2p-1}$ has by \cref{lem:JHTensor}.\ref{en:JHTensorp2p-1} (for $k = 1$) Jordan-Hölder series $V_{p-2} \otimes \rD$, $V_1$ and $V_{p-2} \otimes \rD$;
      \item
        if $a'' = p-1$, then by \cref{lem:JHTensor}.\ref{en:JHTensorp2p-1},
        \[
          V_2 \ox V_{a''} = V_{3p-1} = (V_{p-1} \ox \rD) \op U,
        \]
        where $U$ has successive semisimple Jordan-Hölder factors $V_{p-3} \ox \rD^2$, $(V_0 \otimes \rD) \op V_2$ and $V_{p-3} \ox \rD^2$.
        \qedhere
    \end{itemize}

    If instead $X_{r''}^* = 0$, then the left-hand side of $(*)$ vanishes, and accordingly that of the Jordan-Hölder series of $X_{r''} \ox V_2$.
  \end{proof}

  Let us collect what we can infer about the Jordan-Hölder factors of $X_{r-2}$ by \cref{lem:JHTensor} from looking at the short exact sequence
  \[
    0 \to
    X_{r''}^* \otimes V_2 \to
    X_{r''} \otimes V_2 \to
    X_{r''} / X_{r''}^* \otimes V_2 \to
    0.
  \]
  \begin{itemize}
    \item
      The left-hand side has minimal dimension $3$ for $a'' = p-1$, the right-hand side has minimal dimension $2 \cdot 3 = 6$ for $a'' = 1$.
    \item  Regarding the number of Jordan-Hölder factors,
      \begin{itemize}
        \item
          the left-hand side has $2$ Jordan-Hölder factors for $a'' = p-2$,
        \item
          the left-hand side has the minimal number of Jordan-Hölder factors $1$ for $a'' = p-1$,
        \item
          whereas the right-hand side has minimal number of Jordan-Hölder factors $2$ for $a'' = 1$, and
        \item
          in the generic case $a'' \in \{ 2, \ldots, p-3 \}$, both sides have $3$ Jordan-Hölder factors.
      \end{itemize}
    \item
      Under the conditions of \cref{lem:XrInXr1InXr2}, there are at least $3$ Jordan-Hölder factors in $X_{r-2}$.
      Because $X_{r''} \otimes V_2$ has by \cref{prop:JHXrppxV2} only $6$ Jordan-Hölder factors, $X_{r-2}$ has by the epimorphism $X_{r''} \otimes V_2 \twoheadrightarrow X_{r-2}$ between $3$ and $6$ Jordan-Hölder factors.
  \end{itemize}

  \subsection{Sum of the Digits}

  For a natural number $r$, let
  \[
    \Sigma(r)
    :=
    \text{ the sum of the digits in the $p$-adic expansion of $r$ }.
  \]
  Since $p \equiv 1 \mod (p-1)$, we have $\Sigma(r) \equiv r \mod (p-1)$.
  Thus, if $a$ in $ \{ 1, \ldots, p-1 \}$ such that $r \equiv a \mod (p-1)$, then $\Sigma(r) = a$ is smallest possible.
  In other words, $\Sigma(r) = a$ holds if and only if $\Sigma(r) < p$.
  If $\Sigma(r) < p$, we say $\Sigma(r)$ is \emph{minimal}, otherwise $\Sigma(r)$ is \emph{non-minimal}.

  In the forthcoming subsections we will compute the Jordan-Hölder series of $X_{r-2}$ depending on minimality of $\Sigma(r)$, $\Sigma(r')$ and $\Sigma(r'')$:
  Assuming $X_{r-2} \neq X_{r-1}$, this \cref{sec:Xr-2} will show that, for given $r$, the Jordan-Hölder factors of the kernel of the multiplication map $\phi \colon X_{r''} \otimes V_2 \to X_{r-2}$ of \cref{lem:Xr2V2xXrpp} are to be found among those of $X_{r''}^* \otimes V_1$, $X_{r'}^* \otimes V_1$ and $X_r^*$ with minimal $\Sigma(r'')$, $\Sigma(r')$ or $\Sigma(r)$ (with no contribution to this kernel by the tensor products with non-minimal $\Sigma(r'')$, $\Sigma(r')$ or $\Sigma(r)$).

  The following \cref{lem:digits-minimality} shows that, with few exceptions for $a = 1, 2$, the minimality of $\Sigma(r'')$ implies that of $\Sigma(r')$;
  likewise, the minimality of $\Sigma(r')$ implies that of $\Sigma(r)$.

  \begin{lem}
    \label{lem:digits-minimality}
    Let $a$ in $ \{ 1, \ldots, p-1 \} $ such that $r \equiv a \mod (p-1)$.
    Put $r' = r-1$ and $r'' = r-2$.
    \begin{itemize}
      \item
        For $a$ in $ \{ 3, \ldots, p-1 \} $,
        \begin{itemize}
          \item if $\Sigma(r'')$ is minimal, then $\Sigma(r')$ and $\Sigma(r)$ are minimal;
          \item if $\Sigma(r')$ is minimal, then $\Sigma(r)$ is minimal.
        \end{itemize}
      \item
        For $a = 2$, we have $\Sigma(r')$ is minimal if and only if $r' = p^n$;
        moreover
        \begin{itemize}
          \item If $\Sigma(r'')$ is minimal, then $\Sigma(r')$ is minimal only if $r' = p$ and $\Sigma(r)$ is minimal only if $r = 2$ or $r = p^n + p$ for some $n \geq 0$;
          \item If $\Sigma(r')$ is minimal, then $\Sigma(r)$ is minimal.
        \end{itemize}
      \item
        for $a = 1$, we have $\Sigma(r)$ is minimal if and only if $r = p^n$;
        moreover
        \begin{itemize}
          \item If $\Sigma(r'')$ is minimal (and $r > p$), then $\Sigma(r')$ is minimal but $\Sigma(r)$ is only minimal if $n = 1$;
          \item If $\Sigma(r')$ is minimal, then $\Sigma(r)$ is not minimal.
        \end{itemize}
        For every $a$, if $\Sigma(r'')$ and $\Sigma(r')$ are non-minimal, then $\Sigma(r)$ can be either minimal or non-minimal.
    \end{itemize}
  \end{lem}
  \begin{proof}
    We use the definition of minimality of $\Sigma(r'')$ and that $\Sigma(r') = \Sigma(r'') + 1$ (respectively $\Sigma(r) = \Sigma(r'') + 2$) if $p \nmid r'$ (respectively $p \nmid r$):
    \begin{enumerate}
      \item For $a$ in $ \{ 3, \ldots, p-1 \} $:
        \begin{enumerate}
          \item
            Because $r'' \equiv a - 2$  and $a - 2 \leq p-3$, we have $\Sigma(r'') < p$ if and only if $\Sigma(r'') \leq p-3$.
            Therefore, if $\Sigma(r'') < p$, then both $\Sigma(r') = \Sigma(r'') + 1$ and $\Sigma(r) = \Sigma(r'') + 2 < p$.
          \item Because $r' \equiv a - 1$ and $a - 1 \leq p-2$, if $\Sigma(r') \leq a-1 \leq p-2 < p$, then $\Sigma(r) \leq p-1 < p$.
        \end{enumerate}
      \item For $a = 2$:
        \begin{enumerate}
          \item
            We have $\Sigma(r'')$ is minimal if and only if $\Sigma(r'') = 0$, $p-1$ or $p + p-2$.
            Therefore, if $r' \neq p$, then $\Sigma(r') = \Sigma(r'') + 1 = p$ and if $r \neq p^n + p$ for some $n \geq 1$, then $\Sigma(r) = p+1$.
          \item
            We have $\Sigma(r')$ is minimal if and only if $\Sigma(r') = 1$.
            Therefore $r = p^n + 1$ for some $n \ge  0$ and $\Sigma(r)$ is minimal.
        \end{enumerate}
      \item For $a = 1$:
        \begin{enumerate}
          \item
            We have $\Sigma(r'')$ is minimal if and only if $\Sigma(r'') = p-2$.
            If $r = p$, then $\Sigma(r'') = p-2$.
            Otherwise, because $r > p$, in particular $r'' = r_0 + p R$ with $r_0 < p-2$ and some $R$ in $\mathbb{N}$.
            Therefore $\Sigma(r') = \Sigma(r'') + 1 < p$ is minimal but $\Sigma(r) = \Sigma(r'') + 2 = p$ is non-minimal.
          \item
            If $\Sigma(r') = p-1$ is minimal, then $\Sigma(r) = p$ is not-minimal.
            \qedhere
        \end{enumerate}
    \end{enumerate}
  \end{proof}

  As we will repeatedly cite \cite{BG}, here's how our minimality conditions on $\Sigma(r)$ and $\Sigma(r')$ relate to theirs in Section 4 (for $a = 2, \ldots, p-1$) on $u' = u-1$ where $r = u p^n$ such that $p \nmid u$.
  Then $\Sigma(r)$ is minimal, if and only if $\Sigma(u)$ is minimal, if and only if $\Sigma(u')$ is minimal because $p \nmid u$.
  Putting $r' = r-1$,
  \begin{itemize}
    \item If $\Sigma(u')$ is minimal, then $\Sigma(r')$ is minimal if and only $n = 0$, because $\Sigma(r') = \Sigma(u') - 1 + d$ where $d = 1$ if $n = 0$, that is, $p \nmid r$, and $d > p-1$ if $n > 0$, that is, $p \mid r$.
    \item If $\Sigma(u')$ is not minimal, then $\Sigma(r')$ is not minimal, because $\Sigma(r') = \Sigma(u') - 1 + d$ where $d = 1$ if and only if $p \nmid r$, that is, $n = 0$, and $d > p-1$ if and only if $p \mid r$, that is, $n > 0$.
  \end{itemize}

  The following \cref{prop:XrAst0IffSrMin} states (and proves more directly) results contained in \cite[Sections 3 and 4]{BG}, in particular \cite[Lemma 3.10, Proposition 3.11, Lemma 4.5 and Lemma 4.6]{BG}.

  \begin{prop}
    \label{prop:XrAst0IffSrMin}
    Let $p \geq 3$ and $r \geq p$.
    We have $X_r^{*} = 0$ if and only if $\Sigma(r)$ is minimal.
  \end{prop}
  \begin{proof}
    If $\Sigma(r)$ is minimal, that is, $\Sigma(r) = a$, and
    \begin{itemize}
      \item
        if $a = 1$, that is, $r = p^n$, then $X_1 \iso X_r$ by the $\F[M]$-homomorphism $X \mapsto X^{p^n}$, in particular $X_r^* = 0$ (\cite[Proposition 3.12]{BG});
      \item
        if $a$ in $ \{ 2, \ldots, p-1 \} $, then $\dim X_r < p+1$ by the proof of \cite[Lemma 4.5]{BG} (for $r' = r-1$);
        thus $X_r^* = 0$ by \cref{prop:Xrm012ModAst}.
    \end{itemize}

    Let $\Sigma(r)$ be non-minimal, that is, $\Sigma(r) \geq p$.
    We have $X_r^* = 0$ if and only if $\dim X_r < p + 1$ if and only if the standard generating set of $X_r$ is linearly dependent:
    That is, there is $b_0$, \ldots, $b_{p-1}$ and $b_p$ in $\F_p$, not all zero, such that
    \[
      b_0 Y^r + \sum_{k = 1, \ldots, p-1} b_k (kX + Y)^r + b_p X^r
      =
      0.
      \tag{$*$}
    \]
    We show that if $\Sigma(r) \geq p$, then $(*)$ implies $b_0, \ldots, b_{p-1}, b_p$ to vanish.
    It suffices to show that $b_1$, \ldots, $b_{p-1}$ vanish.
    Because $\# \F_p^* = p-1$,
    \[
      \sum_{k = 1, \ldots, p-1} b_k (kX + Y)^r
      =
      \sum_{k = 1, \ldots, p-1} b_k \sum_{i = 1, \ldots, p-1} k^i \sum_{j \equiv i \mod (p-1)} \binom{r}{j} X^j Y^{r-j}.
      \tag{$**$}
    \]
    For $i = 1, \ldots, p-1$, let
    \[
      B_i = \sum_{k = 1, \ldots, p-1} b_k k^i.
    \]
    By the nonzero Vandermode determinant of $(k^i)_{i,j = 1, \ldots, p-1}$, if $B_1 = \dotsb = B_{p-1} = 0$, then $b_1 = \dotsb = b_{p-1} = 0$.
    Thus, it suffices to show $B_1 = \dotsb = B_{p-1} = 0$.
    Comparing the coefficients of $X^{t} Y^{r-t}$, by $(*)$ and $(**)$, for every $t$ such that $t \equiv i$,
    \[
      B_i \binom{r}{t} = 0.
      \tag{$***$}.
    \]
    Let $t$ in $ \{ 1, \ldots, p-1 \} $.
    Write $r = r_0 + r_1 p + \dotsb$.
    Since $\Sigma(r) = r_0 + r_1 + \dotsb \geq p$, we can write $t = t_0 + t_1 + \dotsb$ with $0 \leq t_j \leq r_j$ for $j = 0, 1, \ldots$.
    Put $t' = t_0 + t_1 p + \dotsb$.
    Then $t' \equiv t \mod (p-1)$, and, by Lucas' Theorem, $\binom{r}{t'} \neq 0$.
    By $(***)$
    \[
      0
      =
      \binom{r}{t'} B_{t'}
      =
      \binom{r}{t'} B_t;
    \]
    that is, $B_t = 0$.
    We conclude that $B_1, \ldots, B_{p-1}$, (and therefore $b_1$, \ldots, $b_{p-1}$) vanish.
  \end{proof}

  \subsection{Sum of the Digits of $r-2$ is \emph{non}-minimal}
  \label{sec:sigma-non-minimal}

  Let $a$ in $\{ 3, \ldots, p + 1 \}$ such that $r \equiv a \mod (p-1)$.
  Let $r'' = r - 2$.
  We assume in this \cref{sec:sigma-non-minimal} that $\Sigma(r'')$ is non-minimal, that is, $\Sigma(r'') \geq p$ and will show that $X_{r-2} / X_{r-1}$ has two Jordan-Hölder factors.

  By \cref{lem:XrInXr1InXr2}, we have $X_{r-2} = X_{r-1}$ if and only if $r = p^n + r_0$ with $r_0$ in $ \{ 2, \ldots, p-1 \} $.
  That is, $r'' = p^n + r_0''$ with $0 \leq r_0 \leq p-3$;
  in particular, $\Sigma(r'')$ is minimal.
  By the same token, $X_{r-1} = X_r$ if and only if $r < p$.

  We conclude that if $r \geq p$ and $\Sigma(r'')$ non-minimal, then
  \[
    0 \subseteq X_r^* \subset X_r \subset X_{r-1} \subset X_{r-2}
  \]
  where
  \begin{itemize}
    \item the two inclusions to the right of $X_r$ are proper by \cref{lem:XrInXr1InXr2},
    \item we have $X_r/X_r^* = V_a$, in particular a proper inclusion $X_r^* \subset X_r$ by \cref{prop:Xrm012ModAst} (which in this case is \cite[(4.5)]{G}), and
    \item we have $X_r^* = 0$ if and only if $\Sigma(r)$ is minimal by \cref{prop:XrAst0IffSrMin}.
  \end{itemize}

  By \cref{lem:JHXr} and \cref{prop:XrAst0IffSrMin} the Jordan-Hölder series of $X_r$ is known.
  Therefore, by \cite[Proposition 3.13 and 4.9]{BG}:
  \begin{itemize}
    \item Let $r \equiv a \mod (p-1)$ for $1 \leq a \leq p-1$.
      \begin{itemize}
        \item Either $\Sigma(r)$ is non-minimal, then the Jordan-Hölder series
          \[
            0 \to
            V_{p-a-1} \otimes \rD^a \to
            X_r \to
            V_a \to
            0,
            \tag{$*$}
          \]
          (which is dual to that of $V_r / V_r^*$, that is, inverts the directions of the arrows of
          \[
            0 \to
            V_a \to
            V_r / V_r^{*} \to
            V_{p-a-1} \otimes \rD^a \to
            0,)
          \]
        \item or it is minimal, in which case the right-hand side of the short exact sequence $(*)$ around $X_r \cong V_a$ vanishes.
      \end{itemize}
    \item Let $r \equiv a \mod (p-1)$ for $2 \leq a \leq p$.
      \begin{itemize}
        \item Either $\Sigma(r')$ is non-minimal, then the Jordan-Hölder series is
          \[
            0 \to
            V_{p-a+1} \otimes \rD^{a-1} \to
            X_{r-1} / X_r \to
            V_{a-2} \otimes \rD \to
            0,
            \tag{$**$}
          \]
          (which is dual to that of $V_r^* / V_r^{**}$ for $a = 2, 3, \ldots, p$)
        \item or it is minimal, in which case
          \begin{itemize}
            \item either $r < p$ and $X_{r-1} / X_r = 0$,
            \item or, otherwise, the right-hand side of the short exact sequence $(**)$ around $X_{r-1} / X_r$ vanishes.
          \end{itemize}
      \end{itemize}
  \end{itemize}

  Regarding $\Sigma(r'')$, let $r \equiv a \mod (p-1)$.
  \begin{itemize}
    \item Either $\Sigma(r'') > p$, then
      \begin{itemize}
        \item We show in \cref{sec:SrSrpSrpp-non-minimal}, \cref{sec:Srp-min-Srpp-non-min} and \cref{sec:Sr-min-Srp-non-min-Srpp-non-min} that the Jordan-Hölder series for $a = 4, \ldots, p-1$ is
      \[
        0 \to
        V_{p-a+3} \otimes \rD^{a-2} \to
        X_{r-2} / X_{r-1} \to
        V_{a-4} \otimes \rD^2 \to
        0,
        \tag{$***$}
      \]
        which is dual to that of $V_r^* / V_r^{**}$ for $a = 4, 5, \ldots, p+1, p+2$ and $r \geq p$.
        If $\Sigma(r'), \Sigma(r) > p$, then this also holds for $a = p, p+1$.
        If $\Sigma(r') < p$, then this also holds for $a = p$
      \end{itemize}
    \item or it is minimal, and for $a = 3, \ldots, p+1$
      \begin{itemize}
        \item either $r = p^n + r_0$ with $r_0 \in \{ 2, \ldots, p-1 \}$, then we proved in \cref{lem:XrInXr1InXr2} that $X_{r-2} / X_{r-1} = 0$,
        \item or, otherwise, we will prove in \cref{sec:sigma-minimal} that the right-hand side of the short exact sequence $(***)$ around $X_{r-2} / X_{r-1}$ vanishes.
      \end{itemize}
  \end{itemize}

  Independently of whether one of $\Sigma(r')$ or $\Sigma(r)$ is minimal or not, if $\Sigma(r'')$ is non-minimal, then, except when $r \equiv 3 \mod (p-1)$, a specific fourth Jordan-Hölder factor appears in $X_{r-2}$:

  \begin{lem}
    \label{lem:Xr2Vpa3}
    Let $a$ in $ \{ 4, \ldots, p+1 \} $ such that $r \equiv a \mod (p-1)$.
    If \, $\Sigma(r'') \geq p$ and $r \geq 3p+2$, then $V_{p-a+3} \otimes \rD^{a-2}$ is a Jordan-Hölder factor of $X_{r-2}$.
    
  \end{lem}
  \begin{proof}
    Let $a$ in $ \{ 1, \ldots, p-1 \} $ such that $r \equiv a \mod (p-1)$.
    Because $\Sigma(r'')$ is non-minimal, by \cref{lem:JHXr} and \cref{prop:XrAst0IffSrMin},
    \[
      (V_{p-a+1} \otimes \rD^{a-2}) \otimes V_2 \iso
      X_{r''}^* \otimes V_2 \twoheadrightarrow
      X_{r-2}.
      \tag{$*$}
    \]
    
    For $n = 0, \ldots, p-3$ in $\N$, let us construct an $\F_p[M]$-linear map
    $
      V_{n+2}
      \to
      V_n \otimes V_2.
    $
    Given $f$ in $V_{n+2}$, let $f_{xx}$, $f_{xy}$ and $f_{yy}$ in $V_n$ denote its partial derivatives of second order.
    By the proof of \cite[(5.2)]{G}, the $\F_p$-linear map
    \begin{align}
      \phi_n \colon V_{n+1} & \to V_n \otimes V_1\\
      f & \mapsto f_x \otimes x + f_y \otimes y
    \end{align}
    is $M$-linear, and so is its iteration $(\phi_n \otimes \id) \circ \phi_{n+1}$, given by
    \begin{align}
      V_{n+2} & \to V_n \otimes (V_1 \otimes V_1)\\
      f & \mapsto
      f_{xx} \otimes x \otimes x +
      f_{xy} \otimes x \otimes y +
      f_{yx} \otimes y \otimes x +
      f_{yy} \otimes y \otimes y.
    \end{align}
    By composing with $\id \otimes \pi$ where $\pi$ is the $\F_p[M]$-linear homomorphism
    $
      V_1 \otimes V_1 \to V_2
    $
    given by $f \otimes g \mapsto f \cdot g$, we obtain that the $\F_p$-linear map
    \begin{align}
      V_{n+2}
      & \to
      V_n \otimes V_2\\
      f
      & \mapsto
      f_{xx} \otimes x^2 +
      f_{xy} \otimes 2 xy +
      f_{yy} \otimes y^2,
    \end{align}
    is $M$-linear.
    In particular, we obtain for $a > 3$ an $\F_p[M]$-linear map
    \[
      V_{p-a+3} \otimes \rD^{a-2}
      \to
      (V_{p-a+1} \otimes \rD^{a-2}) \otimes V_2
      \iso
      X_{r''}^* \otimes V_2
    \]
    whose left arrow sends
    \[
      X^{p-a+3}
      \mapsto
      (p-a+3)(p-a+2) X^{p-a+1} \otimes X^2.
    \]
    If $a > 3$, then $\binom{p-a+3}{2} \nequiv 0 \mod p$, that is, the right-hand side is nonzero.
    Thus, under the right arrow, the map $(*)$,
    \[
      X^{p-a+1} \otimes X^2
      \mapsto
      \psi(X^{p-a+1}) \cdot X^2
      \neq
      0.
    \]
    Therefore, $V_{p-a+3} \otimes \rD^{a-2}$ is a nonzero Jordan-Hölder factor of $X_{r-2}$.
  \end{proof}

    \subsubsection{Sum of the Digits of $r-1$ is minimal}
    \label{sec:Srp-min-Srpp-non-min}

    Because $\Sigma(r')$ is minimal, by \cite[Proposition 3.13 and 4.9]{BG} we have $\dim X_{r-1} < 2p+2$, therefore, by \cref{cor:dimXr2dimXr1}, we have $\dim X_{r-2} < 3p+3$;
    that is, $X_{r-2}$ has at most five Jordan-Hölder factors.

    Let $r \geq 2p+1$ and $\Sigma(r') < p$, that is, the sum of the digits of $r-1$ is minimal.
    Let $a$ in $ \{ 3, \ldots, p+1 \} $ such that $r \equiv a \mod (p-1)$.
    Recall the Jordan-Hölder series of $X_{r-1}$:
    \begin{itemize}
      \item
        If $a = 2, \ldots, p-1$, then by \cite[Proposition 4.9.(i)]{BG},
        $$
          \label{eq:Xr1aSrpMin}
          X_{r-1} = V_{a-2} \otimes \rD \oplus V_a.
        $$
      \item
        Otherwise, if $a = p$, then by \cite[Proposition 3.13.(i)]{BG},
        $
          \label{eq:Xr1aEpp1SrpMin}
          X_{r-1} = V_{2p-1}
        $
        where we recall that $V_{2p-1}$ has successive semisimple Jordan-Hölder factors $V_{p-2} \otimes \rD$, $V_1$ and $V_{p-2} \otimes \rD$ as stated in \cref{cor:JHVxV2}.
    \end{itemize}

    \begin{prop}
      \label{prop:Xr2SrpMin}
      Let $a$ in $ \{ 4, \ldots, p \} $ such that $r \equiv a \mod (p-1)$.
      Let $\Sigma(r'') \geq p$ and $\Sigma(r') < p$.
      If $r \geq 3p+2$, then
      \[
        0 \to
        V_{p-a+3} \otimes \rD^{a-2} \to
        X_{r-2} / X_{r-1} \to
        V_{a-4} \otimes \rD^2 \to
        0.
      \]
    \end{prop}
    \begin{proof}
      By \cref{lem:Xr2Vpa3},
      \[
        X_{r-2}
        \hookleftarrow
        V_{p-a+3} \otimes \rD^{a-2}.
      \]
      Expand $r = r_0 + r_1 p + \dotsb$ $p$-adically.
      Because $\Sigma(r') = a-1$ in $\{ 3, \ldots, p-1 \}$ (and $r \geq p$), we have $r_0 \leq a-1$.
      Therefore $r \equiv r_0 \neq a \mod p$.
      If $r_0 = a-1$ in $\{ 3, \ldots, p-1 \}$, then $r = r_0 + p^n$;
      in particular, $\Sigma(r'')$ would be minimal.
      Therefore $r_0 \neq a-1 \mod p$.

      Thus we can apply \cref{lem:Xr2AstAstModXr2AstAstAstUnequalModP} (for $a > 4$) respectively \cref{lem:Xr2AstAstModXr2AstAstAstUnequalModPaEqual4} (for $a = 4$), yielding by \cref{lem:VrAstQuotients}.\ref{en:VrAstAstModVrAstAstAst},
      \[
        X_{r-2}^{**} / X_{r-2}^{***}
        \hookleftarrow
        V_{a-4} \otimes \rD^2.
      \]
      By \cref{lem:Xr2V2xXrpp}, the Jordan-Hölder series of $X_{r-2}$ is included in that of \cref{prop:JHXrppxV2}.
      
      We conclude by \cref{cor:dimXr2dimXr1} and \eqref{eq:Xr1aSrpMin} that the Jordan-Hölder series of $X_{r-2} / X_{r-1}$ is
      \[
        0 \to
        V_{p-a+3} \otimes \rD^{a-2} \to
        X_{r-2} / X_{r-1} \to
        V_{a-4} \otimes \rD^2 \to
        0.
        \qedhere
      \]
    \end{proof}

    \cref{prop:Xr2SrpMin} with \cref{prop:JHXrppxV2} and (the Jordan-Hölder series of $X_{r-1}$) yield the Jordan-Hölder series of $X_{r-2}$.

    As the Jordan-Hölder series of $Q = V_r / (X_{r-2} + V_r^{***})$ (and thus our main theorem) does not depend on whether $\Sigma(r'')$, $\Sigma(r')$ or $\Sigma(r)$ are minimal or not, we dispense with the cases $a = 2,3$ at this point.

    \subsubsection{Sum of the Digits of $r-1$ is \emph{non}-minimal but that of $r$ is minimal}
    \label{sec:Sr-min-Srp-non-min-Srpp-non-min}

    Because $\Sigma(r)$ is minimal, by \cref{prop:XrAst0IffSrMin} we have $\dim X_r < p$, therefore, by \cref{cor:dimXr2dimXr1}, we have $\dim X_{r-2} < 3p+3$;
    that is, $X_{r-2}$ has at most five Jordan-Hölder factors.
    We will show that all occur.

    Let $a$ in $ \{ 3, \ldots, p+1 \} $ such that $r \equiv a \mod (p-1)$.
    Let $r \geq 2p+1$ and $\Sigma(r) < p$, that is, the sum of the digits of $r$ is minimal.
    Recall the Jordan-Hölder series of $X_{r-1}$:
    \begin{enumerate}
      \item For $a = 3, \ldots, p-1$ and $a = p+1$, by \cite[Proposition 4.9.(ii)]{BG},
        \[
          \label{eq:Xr1aSrMin}
          0 \to V_{p-a+1} \otimes \rD^{a-1} \to X_{r-1} \to V_{a-2} \otimes \rD \oplus V_a \to 0.
        \]
      \item For $a = p$, we have $r = p^n$ for $n > 1$ and by \cite[Proposition 3.13.(iii)]{BG},
        \[
          \label{eq:Xr1aE1SrMin}
          0 \to V_1 \otimes \rD^{p-1} \to X_{r-1} \to W \to 0
        \]
        where $W = V_{2p-1} / V_{2p-1}^*$, that is, $0 \to V_{p-2} \otimes \rD \to  W \to V_1 \to 0$.
    \end{enumerate}

    \begin{prop}
      \label{prop:Xr2SrpNonMinSrMin}
      Let $r \geq 3p+2$ and $\Sigma(r) < p$, $\Sigma(r') \geq p$, $\Sigma(r'') \geq p$.
      Let $r \equiv a \mod (p-1)$.
      If $a$ in $ \{ 4, \ldots, p-1 \}$, then
      \[
        0 \to
        V_{p-a+3} \otimes \rD^{a-2} \to
        X_{r-2} / X_{r-1} \to
        V_{a-4} \otimes \rD^2 \to
        0.
      \]
    \end{prop}
    \begin{proof}
      By \cref{lem:Xr2Vpa3},
      \[
        X_{r-2}
        \hookleftarrow
        V_{p-a+3} \otimes \rD^{a-2}.
      \]
      Expand $r = r_0 + r_1 p + \dotsb$ $p$-adically.

      If $\Sigma(r-1)$ is \emph{non}-minimal but $\Sigma(r)$ is minimal, then $r \equiv 0 \mod p$.
      In particular, for $a = \{ 4, \ldots, p-1 \}$, we have $r \not\equiv a, a-1 \mod p$.

      Thus we can apply \cref{lem:Xr2AstAstModXr2AstAstAstUnequalModP} (for $a > 4$) respectively \cref{lem:Xr2AstAstModXr2AstAstAstUnequalModPaEqual4} (for $a = 4$), yielding by \cref{lem:VrAstQuotients}.\ref{en:VrAstAstModVrAstAstAst},
      \[
        X_{r-2}^{**} / X_{r-2}^{***}
        \hookleftarrow
        V_{a-4} \otimes \rD^2.
      \]
      By \cref{lem:Xr2V2xXrpp}, the Jordan-Hölder series of $X_{r-2}$ is included in that of \cref{prop:JHXrppxV2}.
      Because $\Sigma(r)$ is minimal, by \cref{prop:XrAst0IffSrMin} we have $\dim X_r < p$, therefore, by \cref{cor:dimXr2dimXr1}, we have $\dim X_{r-2} < 3p+3$;
      that is, $X_{r-2}$ has at most five Jordan-Hölder factors;
      whereas $X_{r-1}$ has three Jordan-Hölder factors by \eqref{eq:Xr1aSrMin}.

      Since neither $V_{a-4} \otimes D^2$ nor $V_{p-a-3} \otimes D^{a-2}$ are Jordan-Hölder factors of $X_{r-1}$, we can conclude by \cref{cor:dimXr2dimXr1} that the Jordan-Hölder series of $X_{r-2} / X_{r-1}$ is
      \[
        0 \to
        V_{p-a+3} \otimes \rD^{a-2} \to
        X_{r-2} / X_{r-1} \to
        V_{a-4} \otimes \rD^2 \to
        0.
        \qedhere
      \]
    \end{proof}

    \cref{prop:Xr2SrpNonMinSrMin} with \cref{prop:JHXrppxV2} and (the Jordan-Hölder series of $X_{r-1}$) yield the Jordan-Hölder series of $X_{r-2}$.

    \begin{lem}[Extension of {\cite[Lemma 3.10]{BG}}]
      \label{lem:Xr2SrMinrEqPn}
      If $r = p^n$ for some $n > 1$, then $\dim X_{r-2} = 2p + 4$.
    \end{lem}
    \begin{proof}
      By \cref{lem:Xr2Generators},
      \[
        \{ X^2 (kX + Y)^{r-2}, XY (lX + Y)^{r-2}, Y^2 (X + mY)^{r-2}, X^r, Y^r, X^{r-1}Y, XY^{r-1} : k,l,m \in \F_p \}
      \]
      is a set of generators of $X_{r-2}$.
      Because
      \[
        (X+kY)^2
        =
        X^2 + 2k XY + k^2 Y^2,
      \]
      and therefore
      \[
        (X+kY)^r
        =
        X^2 (X+kY)^{r-2} + 2k XY (X+kY)^{r-2} + k^2 Y^2 (X+kY)^{r-2} ,
      \]
      the span over $\F_p$ of the sets
      \[
        \{ X^2 (kX + Y)^{r-2}, XY (lX + Y)^{r-2}, Y^2 (X + mY)^{r-2}, X^r, Y^r, X^{r-1}Y : k,l,m \in \F_p \}
      \]
      and
      \[
        \{ X^2 (kX + Y)^{r-2}, XY (lX + Y)^{r-2}, (X+mY)^r,  X^{r-2}Y^2, Y^r, X^{r-1}Y, X^r : k,l,m \in \F_p \}
      \]
      are equal.
      Because $r = p^n$, we have $(X+mY)^r = X^r + m^r Y^r$, and therefore the span of
      \[
        \{ (X+mY)^r : m \in \F_p \}
      \]
      equals that of $X^r$ and $Y^r$.
      Therefore the span over $\F_p$ of
      \[
        \{ X^2 (kX + Y)^{r-2}, XY (lX + Y)^{r-2}, Y^2 (X + mY)^{r-2}, X^r, Y^r, X^{r-1}Y : k,l,m \in \F_p \}
      \]
      equals that of
      \[
        \{ X^2 (kX + Y)^{r-2}, XY (lX + Y)^{r-2}, X^{r-2} Y^2, X^r, Y^r, X^{r-1}Y : k,l \in \F_p \}.
      \]
      We show that the elements of the latter set are linearly independent, that is, if
      \[
        A X^r + B Y^r + C X^{r-1}Y + D X^{r-2}Y^2 + \sum_{k \in \F_p} e_k X^2 (kX + Y)^{r-2} + \sum_{l \in \F_p} f_l XY (lX + Y)^{r-2}
        =
        0,
        \tag{$*$}
      \]
      then the coefficients $A,B,C,D$ and $e_k, f_l$ for $f,l$ in $\F_p$ all vanish.
      Let $t$ in $ \{ 1, \ldots, r \} $.
      Comparing the coefficients of $X^{t+2} Y^{r-2-t}$ on both sides of $(*)$ gives
      \[
        \binom{r-2}{t} \sum_{k = 1, \ldots, p-1} e_k k^t
        +
        \binom{r-2}{t+1} \sum_{l = 1, \ldots, p-1} f_l l^{t+1}
        =
        0.
        \tag{$**$}
      \]
      Let
      \[
        E_t
        :=
        \sum_{k = 1, \ldots, p-1} e_k k^t
        \text{ \quad and \quad }
        F_{t+1}
        :=
        \sum_{l = 1, \ldots, p-1} f_l l^{t+1}.
      \]
      Because $\# \F_p^* = p-1$, the sums $E_t$ and $F_{t+1}$ only depend on $t \mod (p-1)$.
      Because the Vandermonde determinant is nonzero, if $E_1, \ldots, E_{p-1} = 0$ then $e_1, \ldots, e_{p-1} = 0$;
      likewise if $F_1, \ldots, F_{p-1} = 0$ then $f_1, \ldots, f_{p-1} = 0$.
      It therefore suffices to show that $E_1, \ldots, E_{p-1} = 0$ and $F_1, \ldots, F_{p-1} = 0$.

      Write
      \[
        r-2
        =
        p^n-2
        =
        r_{n-1} p^{n-1} + \dotsb + r_1 p + r_0
        =
        (p-1)p^{n-1} + \dotsb + (p-1) p + p-2.
      \]
      For $t = 1, \ldots, p-2$, put $t' = t + p - 1$.
      Then $t' \leq r$ and $t' \equiv t \mod (p-1)$.
      By $(**)$,
      \begin{align}
        \binom{r-2}{t} E_t  + \binom{r-2}{t+1} F_{t+1}  = 0 \\
        \binom{r-2}{t'} E_t + \binom{r-2}{t'+1} F_{t+1} = 0.
      \end{align}
      The determinant of this linear equation system is
      \[
        \binom{r_0}{t+1} \binom{r_1}{1} \binom{r_0}{t-1}
        -
        \binom{r_0}{t} \binom{r_1}{1} \binom{r_0}{t}
        \equiv
        - \binom{r_1}{1} \binom{r_0}{t-1} \binom{r_0}{t} \frac{r_0+1}{t(t+1)}
        \not\equiv
        0
        \mod p
      \]
      because $0 < r_0 + 1, r_1 \leq p-1$.
      Therefore $E_t,F_t = 0$.

      For $t = p-1$, put $t' = r_0 + p$.
      Then $t' \leq r-2$ and $t' \equiv t \mod p$.
      We compute
      \[
        \binom{r-2}{t'+1}
        \equiv
        \binom{r_1}{1} \binom{r_0}{p-1}
        \equiv
        0
        \mod p
        \text{ \quad and \quad }
        \binom{r-2}{t'}
        \equiv
        \binom{r_1}{1} \binom{r_0}{p-2}
        \not\equiv
        0
        \mod p.
      \]
      Therefore $(**)$ gives $E_t = 0$.
      Similarly, choosing $t' = p(p-1)$ in $(**)$ yields $F_t = 0$.
    \end{proof}

    \begin{prop}
      Let $\Sigma(r'') \geq p$, $\Sigma(r') \geq p$ and $\Sigma(r) < p$.
      If $r \equiv p \mod (p-1)$ and $r \geq 3p+2$, then $r = p^n$ for $n > 1$ and the Jordan-Hölder series of $X_{r-2} / X_{r-1}$ is
      \[
        0 \to
        V_3 \otimes \rD^{p-2} \to
        X_{r-2} / X_{r-1} \to
        V_{p-4} \otimes \rD^2 \to
        0.
      \]
    \end{prop}
    \begin{proof}
      By \cref{prop:JHXrppxV2} for $a = p$, we have
      \begin{align}
        0 & \to (V_3 \otimes \rD^{p-2}) \op (V_1 \otimes \rD^{p-1})\\
          & \to X_{r''} \ox V_2\\
          & \to (V_{p-4} \otimes \rD^2) \op V_{2p-1} \to 0
          \tag{$*$}
      \end{align}
      where $V_{2p-1}$ has Jordan-Hölder series $V_{p-2} \otimes \rD$, $V_1$ and $V_{p-2} \otimes \rD$.
      By \cref{lem:Xr2SrMinrEqPn} and \cite[Proposition 3.13.(iii)]{BG}, we have $\dim X_{r-1} = p + 3$ and $\dim X_{r-2} = 2p + 4$.
      By comparing \cref{eq:Xr1aE1SrMin} with $(*)$, the Jordan-Hölder factors $V_3 \otimes \rD^{p-2}$ and $V_{p-4} \otimes \rD^2$ must appear in the Jordan-Hölder series of $X_{r-2}$.
    \end{proof}

    \subsubsection{Sum of the Digits of $r-1$ and $r$ are \emph{non}-minimal}
    \label{sec:SrSrpSrpp-non-minimal}

    We show that if $\Sigma(r'')$, $\Sigma(r')$ and $\Sigma(r)$ are all non-minimal, then $X_{r-2}$ is maximal, that is, $\dim X_{r-2} = 3p+3$.

    We recall that $\Sigma(r)$ is non-minimal if and only if, for $a$ in $ \{ 1, \ldots, p-1 \}$ such that $r \equiv a \mod (p-1)$, we have $\Sigma(r) > a$, that is, if and only if $\Sigma(r) \geq p$.
    Therefore, in analogy to \cite[Lemma 4.3]{BG}, we conclude that $\Sigma(r'')$, $\Sigma(r')$ and $\Sigma(r)$ are non-minimal if and only if
    \begin{enumerate}
      \item either $p \nmid r',r$ and $\Sigma(r'')$ non-minimal,
      \item or $r = p^n u$ for $n \geq  1$, and $\Sigma(u)$ non-minimal,
      \item or $r' = p^n u'$ for $n \geq  1$ and $\Sigma(u')$ non-minimal.
    \end{enumerate}
    We will prove successively that $\dim X_{r-2} = 3p+3$ is maximal in each one of these possibilities:

    \begin{lemma}[Analogue of {\cite[Lemma 4.2]{BG}}]
    \label{lem:Xr2FullDimGeneralPNotDivRpR}
      Let $p > 3$ and let $r \geq 3p+2$.
      If $\Sigma(r'') \geq p$ and $p \nmid r', r$, then $\dim X_{r-2} = 3p+3$.
    \end{lemma}
    \begin{proof}
      We need to show that the spanning set
      \[
        \{ X^r, Y^r, X^{r-1}Y,  X^2(j X + Y)^{r-2}, Y^2(X + kY)^{r-2}, XY(lX + Y )^{r-2} : j,k,l \in \F_p\}
      \]
      is linearly independent;
      that is, if there are constants $A, B, C$ and $d_j, e_k, f_l \in \mathbb{F}_p$ for $j, k, l = 0, 1,..., p-1$ satisfying
      \begin{align}
        0 = & \quad A X^r+B Y^r +C X^{r-1}Y\\
            & + \sum_j d_j Y^2(X + jY)^{r-2} + \sum_k e_k XY(kX + Y )^{r-2} + \sum_l f_l X^2(l X + Y)^{r-2} \tag{$*$}
      \end{align}
      then $A, B, C = 0$ and  $d_j, e_k, f_l =0$ for $j, k, l = 0, 1,..., p-1$.

      Let us assume $(*)$.
      Put
      \[
        D_i := \sum d_j j^i,\; E_i := \sum e_k k^{r-3-i},\; F_i := \sum f_l l^{r-4-i} \quad \text{for $i = 0, \ldots, r-4$}
      \]
      Because $\# \F_p^* = p-1$, we have $D_{i'} \equiv D_{i''}$ for all $i' \equiv i'' \mod (p-1)$ for $i', i'' > 0$.
      If $D_1$, \ldots, $D_{p-1} = 0$, then $d_1, \ldots, d_{p-1} = 0$ (and therefore $d_0 = 0$), because the system of linear equations of $D_1$, \ldots, $D_{p-1} = 0$ has full rank (by its nonzero Vandermonde determinant).
      Likewise if $E_1$, \ldots, $E_{p-1} = 0$, then $e_1, \ldots, e_{p-1} = 0$ and if $F_1$, \ldots, $F_{p-1} =0$, then $f_1, \ldots, f_{p-1} = 0$.
      To show that all coefficients $A$, $B$, $C$ e $d_j$, $e_k$ and $f_l$ for $j,k,l = 0$, \ldots, $p-1$ vanish, it therefore suffices to show $D_1, \ldots, D_{p-1} = 0$ and $E_1, \ldots, E_{p-1} = 0$.

      By comparing the coefficient of $X^{r-2-t}Y^{t+2}$ on both sides of $(*)$ for $t$ in  $\{ 1$, \ldots, $r-5 \}$,
      \[
        0
        =
        \binom{r-2}{t} D_{t} + \binom{r-2}{t+1} E_{t} +  \binom{r-2}{t+2} F_t.
        \label{eq:linear-equations}
      \]
      We will show that \cref{eq:linear-equations} forces $D_{t'}$, $E_{t''}$ and $F_{t'''}$ to vanish for $t'$ and $t''$ in full sets of representatives of $\{ 1, \ldots, p-1 \}$.
      That is, for every $t$ in $ \left\{ 1, \ldots, p-1 \right\} $ there is $t'$, $t''$ and $t'''$ with $t' \equiv t$, $t'' \equiv t$ and $t''' \equiv t \mod (p-1)$ such that $D_{t'}$, $E_{t''}$ and $F_{t'''}$ vanish.

      Expand $r-2 = r_0 + r_1 p + r_2 p^2 + \dotsb$ with $r_0, r_1, \ldots \in \left\{ 0, \ldots, p-1 \right\}$.
      Let $i$ be the smallest index such that $r_i \neq 0$.
      Fixate $t$ in $ \left\{ 1, \ldots, p-1 \right\} $.

      \begin{enumerate}[{Case} 1.]
        \item
          Suppose $t \in \left\{ 1, \ldots, r_i-1 \right\}$.

          If $r_0 = 0$, then $i > 0$.
          By Lucas' Theorem,
          \begin{itemize}
            \item
              for $t':= t p^i$, we have $\binom{r-2}{t'} \not\equiv 0$ and $\binom{r-2}{t' + 1}$, $\binom{r-2}{t' + 2} \equiv 0 \mod p$,
              thus \cref{eq:linear-equations} yields $D_{t'} = 0$;
            \item
              for $t'' := (t+1) p^i - 1$, we have $\binom{r-2}{t''+1} \not\equiv 0$ and $\binom{r-2}{t''+2}$, $\binom{r-2}{t''} \equiv 0 \mod p$, thus \cref{eq:linear-equations} yields $E_{t''} = 0$.
            \item
              The following choice of $t'$ satisfies $t' \equiv t \mod (p-1)$ and $\binom{r-2}{t'+2} \not\equiv 0$, so that  $F_{t'} = 0$ by \cref{eq:linear-equations} as we already know $D_{t'} = E_{t'} = 0$:
              \begin{itemize}
                \item If $t < p-2$, put $t' := (t+2) p^i - 2$.
                \item Otherwise, if $t = p-2$ (thus $r_i = p-1$),
                  then put $t' = p^i - 2$.
              \end{itemize}
          \end{itemize}
          Because $t', t'' \equiv t \mod p$, we have $D_t = D_{t'} = 0$ and $E_t = E_{t''} = 0$.
          We can therefore assume that $r_0 > 1$;
          in particular, $i = 0$.

          In the following, we choose $t', t'' \equiv t \mod (p-1)$ such that $(+)$ yields modulo $p$ the system of equations:
          \begin{alignat}{4}
            \binom{r-2}{t}   D_t & +{} & \binom{r-2}{t+1} & E_t +{} & \binom{r-2}{t+2} &   F_t   & {}\equiv 0 \\
            \binom{r-2}{t'}  D_t & +{} & \binom{r-2}{t'+1} & E_t +{} & \binom{r-2}{t'+2} &  F_t  & {}\equiv 0 \\
            \binom{r-2}{t''} D_t & +{} & \binom{r-2}{t''+1} & E_t +{} & \binom{r-2}{t''+2} & F_t & {}\equiv 0
          \end{alignat}
          We show $D_t = E_t = F_t = 0$ by proving that the determinant of the matrix $M$ attached to this system of equations is nonzero, that is,
          \[
            \norm{M}
            =
            \begin{vmatrix}
              \binom{r-2}{t}   & \binom{r-2}{t+1}   &  \binom{r-2}{t+2} \\
              \binom{r-2}{t'}  & \binom{r-2}{t'+1}  &  \binom{r-2}{t'+2} \\
              \binom{r-2}{t''} & \binom{r-2}{t''+1} &  \binom{r-2}{t''+2}
            \end{vmatrix}
            \not\equiv 0 \mod p.
          \]
          \begin{enumerate}[{label*=\arabic*.}]
            \item
              There is an index $i > 0$ such that $r_i > 1$.
              Put $t' := t+p^{i} - 1$ and $t'' := t + 2p^{i} - 2$
              \begin{enumerate}[{label*=\arabic*.}]
                \item
                  \label{en:Xr2FullDim-111}
                  Suppose $t \in \left\{ 2, \ldots, r_0-2 \right\}$.
                  By Lucas' Theorem, we have
                  \begin{itemize}
                    \item $\binom{r-2}{t'} = \binom{r_0}{t-1}\binom{r_{i}}{1}$, $\binom{r-2}{t'+1} = \binom{r_0}{t}\binom{r_{i}}{1}$ and $\binom{r-2}{t'+2} = \binom{r_0}{t+1}\binom{r_{i}}{1}$, as well as
                    \item $\binom{r-2}{t''} = \binom{r_0}{t-2}\binom{r_{i}}{2}$, $\binom{r-2}{t''+1} = \binom{r_0}{t-1}\binom{r_{i}}{2}$ and $\binom{r-2}{t''+2} = \binom{r_0}{t}\binom{r_{i}}{2}$.
                  \end{itemize}
                  Thus,
                  \begin{align}
                    \norm{M}
                        & \equiv
                        \begin{vmatrix}
                          \binom{r_0}{t}   & \binom{r_0}{t+1}   &  \binom{r_0}{t+2} \\
                          \binom{r_0}{t-1}\binom{r_{i}}{1}  & \binom{r_0}{t} \binom{r_{i}}{1} &  \binom{r_0}{t+1}\binom{r_{i}}{1} \\
                          \binom{r_0}{t-2} \binom{r_{i}}{2}& \binom{r_0}{t-1}\binom{r_{i}}{2} &  \binom{r_0}{t}\binom{r_{i}}{2}
                        \end{vmatrix}\\
                        & =
                        r_{i} \binom{r_{i}}{2} \cdot
                        \begin{vmatrix}
                          \binom{r_0}{t}   & \binom{r_0}{t+1}   &  \binom{r_0}{t+2} \\
                          \binom{r_0}{t-1}  & \binom{r_0}{t}  &  \binom{r_0}{t+1} \\
                          \binom{r_0}{t-2} & \binom{r_0}{t-1} &  \binom{r_0}{t}
                        \end{vmatrix}\\
                        \mod p.
                  \end{align}
                  By \cite[(2.17)]{Krattenthaler:DeterminantFormulas} (for $a=t$ and $a+b = r_0$ in the notation of \textit{loc. cit.}),
                  \[
                    \begin{vmatrix}
                      \binom{r_0}{t}   & \binom{r_0}{t+1}   &  \binom{r_0}{t+2} \\
                      \binom{r_0}{t-1}  & \binom{r_0}{t}  &  \binom{r_0}{t+1} \\
                      \binom{r_0}{t-2} & \binom{r_0}{t-1} &  \binom{r_0}{t}
                    \end{vmatrix}
                    \equiv
                    \prod_{i=1,2,3} \prod_{j=1, \ldots, t} \prod_{k = 1, \ldots, r_0-t} \frac{i+j+k-1}{i+j+k-2}
                    \mod p
                  \]
                  For this product to be nonzero, every factor has to be nonzero.
                  Because $j+k \leq r_0$, we have $i+j+k-1$ in $ \{ 2, \ldots, r_0+2 \}$.
                  This set does not contain $0$ in $\F_p$ if and only if $r_0 < p -2$.
                  Because $p \nmid r', r$, we have $r_0 < p-2$, and conclude $\norm{M} \neq 0$ in $\F_p$.
                  That is, $D_t = E_t = F_t = 0$.
                \item
                  \label{en:Xr2FullDim-112}
                  Suppose $t = 1$.
                  Then $t'' = 2p^i - 1 = p^i + p^i - 1 = p^i + (p-1)(1 + p + \cdots + p^{i-1})$ (and $t' = p^i$).
                  Because $r_0 < p -2$, by Lucas' Theorem, $\binom{r-2}{t''} \equiv 0 \mod p$.
                  Therefore
                  \begin{align}
                    \norm{M}
                        & \equiv
                        r_{i} \binom{r_{i}}{2} \cdot
                        \begin{vmatrix}
                          \binom{r_0}{1}   & \binom{r_0}{2}   &  \binom{r_0}{3} \\
                          \binom{r_0}{0}  & \binom{r_0}{1}  &  \binom{r_0}{2} \\
                          0 & \binom{r_0}{0} &  \binom{r_0}{1}
                        \end{vmatrix}
                        =
                        r_{i} \binom{r_{i}}{2} \frac{r_0 (r_0 + 1)(r_0 + 2)}{6}
                        \neq 0 \mod p,
                  \end{align}
                  because $r_0 < p-2$.
                  This determinant is well-defined because by assumption $p > r_0 +2 \geq 3$.
                \item
                  \label{en:Xr2FullDim-113}
                  Suppose $t = r_0 - 1$.
                  Because $p \not\mid r, r-1$, we have $r_0 < p-2$.
                  Therefore $t+ 1 = r_0 + 1 < p-1$.
                  Thus $\binom{r-2}{t + 2} \equiv \binom{r_0}{r_0 + 1} \equiv 0 \mod p$ by Lucas' Theorem.
                  Therefore, similarly to the case $t = 1$,
                  \begin{align}
                    \norm{M}
                        & \equiv
                        r_{i} \binom{r_{i}}{2} \cdot
                        \begin{vmatrix}
                          \binom{r_0}{t}   & \binom{r_0}{t+1}   &  0\\
                          \binom{r_0}{t-1}  & \binom{r_0}{t}  &  \binom{r_0}{t+1} \\
                          \binom{r_0}{t-2} & \binom{r_0}{t-1} &  \binom{r_0}{t}
                        \end{vmatrix}
                        \neq 0 \mod p.
                  \end{align}
              \end{enumerate}
            \item
              All $r_1$, $r_2$, \ldots $\leq 1$.
              Because $\Sigma(r'') \geq p$ and $r_0 < p-1$, there are $0 < i' < i''$ such that $r_{i'}$ and $r_{i''} = 1$.
              Put $t' := t + p^{i'} - 1$ and $t'' := t + p^{i''} + p^{i'} - 2$.
              \begin{enumerate}[{label*=\arabic*.}]
                \item
                  Suppose $t \in \left\{ 2, \ldots, r_0-2 \right\}$.
                  Then, similar to \ref{en:Xr2FullDim-111},
                  \[
                    \norm{M}
                    \equiv
                    r_{i'}^2 r_{i''}
                    \begin{vmatrix}
                      \binom{r_0}{t}   & \binom{r_0}{t+1}   &  \binom{r_0}{t+2} \\
                      \binom{r_0}{t-1}  & \binom{r_0}{t}  &  \binom{r_0}{t+1} \\
                      \binom{r_0}{t-2} & \binom{r_0}{t-1} &  \binom{r_0}{t}
                    \end{vmatrix}
                    \neq
                    0
                    \mod p.
                  \]
                \item
                  Suppose $t = 1$.
                  Because $\Sigma(r'') \geq p$ and $r_0 < p-1$, there are $0 < i' < i''$ such that $r_{i'}$ and $r_{i''} = 1$.
                  Put $t' := t + p^{i'} - 1$ and $t'' := t + p^{i''} + p^{i'} - 2$.
                  Then $t'' = p^{i''} + p^{i'} - 1 =  p^{i''}+ (p-1)(1 + p + \cdots + p^{i'-1})$.
                  Then, similar to \ref{en:Xr2FullDim-112},
                  \begin{align}
                    \norm{M}
                            & \equiv
                            r_{i''} r_{i'}^2 \cdot
                            \begin{vmatrix}
                              \binom{r_0}{1}   & \binom{r_0}{2}   &  \binom{r_0}{3} \\
                              \binom{r_0}{0}  & \binom{r_0}{1}  &  \binom{r_0}{2} \\
                              0 & \binom{r_0}{0} &  \binom{r_0}{1}
                            \end{vmatrix}
                            \neq 0 \mod p.
                  \end{align}
                \item
                  Suppose $t = r_0-1$.
                  Then $\binom{r-2}{t+2} \equiv 0 \mod p$.
                  Putting $t' := t + p^{i'} - 1$ and $t'' := t + p^{i''} + p^{i'} - 2$, similar to \ref{en:Xr2FullDim-113},
                  \begin{align}
                    \norm{M}
                            & \equiv
                            r_{i''} r_{i'} \cdot
                            \begin{vmatrix}
                              \binom{r_0}{t}   & \binom{r_0}{t+1}   &  0\\
                              \binom{r_0}{t-1}  & \binom{r_0}{t}  &  \binom{r_0}{t+1} \\
                              \binom{r_0}{t-2} & \binom{r_0}{t-1} &  \binom{r_0}{t}
                            \end{vmatrix}
                            \neq 0 \mod p.
                  \end{align}
              \end{enumerate}
          \end{enumerate}
        \item
          Suppose $t \in \left\{ r_i, \ldots, p-1 \right\}$.
          \begin{itemize}
            \item
              By assumption $\Sigma(r'') = r_i + \dotsb + r_m \geq p$, so we can write $t = r_i + s_{i+1} + \dotsb + s_m$ with $s_j$ in $ \{ 0, \ldots, r_j \} $ for $j = i+1, \ldots, m$.
              Put $t' = r_i + s_{i+1} p + \dotsb + s_m p^m$.
              Then $t' \equiv t \mod (p-1)$ and $\binom{r-2}{t'} \not\equiv 0 \mod p$ by Lucas' Theorem.
              If
              \begin{itemize}
                \item
                  either $i = 0$, then, because $p \not\mid r-1,r$, we have $r_0 < p-2$.
                  Therefore $\binom{r-2}{t'+1}, \binom{r-2}{t'+2} \equiv 0 \mod p$ by Lucas' Theorem.
                \item
                  or $i > 0$, then $r_0 = 0$.
                  Therefore $\binom{r-2}{t'+1}, \binom{r-2}{t'+2} \equiv 0 \mod p$ by Lucas' Theorem.
              \end{itemize}
              By \cref{eq:linear-equations}, in either case $D_t = D_{t'} = 0$.
            \item
              To show $E_t =0$, we choose $t'$ with $t' \equiv t \mod (p-1)$ as follows:
              \begin{itemize}
                \item
                  If $i = 0$, then let $r_0' = r_0 - 1$.
                  Because by assumption $\Sigma(r'') = r_0 + \dotsb + r_m \geq p$ and $t \leq p-1$, we can write $t = r_0' + s'_1 + \dotsb$ with $s'_j$ in $ \{ 0, \ldots, r_j \} $ for $j = 1, 2, \ldots$.
                  Put $t' = r_0' + s'_1 p + \dotsb$.
                  Then $t' \equiv t \mod (p-1)$.

                  Because $i = 0$ and $p \not\mid r-1,r$, we have $r_0 < p-2$.
                  Therefore $\binom{r-2}{t'+1} \nequiv 0$ and $\binom{r-2}{t'+2} \equiv 0 \mod p$ by Lucas' Theorem.
                \item
                  If $i > 0$, then let $r_i' = r_i - 1$.
                  Because by assumption $\Sigma(r'') = r_i + \dotsb + r_m \geq p$ and $t \leq p-1$, we can write $t = r_i' + s'_{i+1} + \cdots$ with $s'_j$ in $ \{ 0, \ldots, r_j \} $ for $j = 1, 2, \ldots$.
                  Put $t' = (p-1) + \dotsb + (p-1) p^{i-1} + r_i' p^i + s'_{i+1} p^{i+1} + \dotsb$.
                  Then $t' \equiv t \mod (p-1)$.
                  Because $t'+ 1 = r_i + s'_{i+1} p^{i+1} + \dotsb$, by Lucas' Theorem $\binom{r-2}{t'+1} \nequiv 0 \mod p$.

                  Since $i > 0$, in particular $r_0 = 0$, that is, $t'+ 2 = 1 + r_i p^i + s'_{i+1} p^{i+1} + \dotsb$.
                  By Lucas' Theorem, $\binom{r-2}{t'+2} \equiv  0 \mod p$ .
              \end{itemize}
              Since $D_t = 0$, we conclude by \cref{eq:linear-equations}, that in either case $E_t = 0$.
            \item
              To show $F_t =0$, we choose $t'$ with $t' \equiv t \mod (p-1)$ as follows:
              \begin{itemize}
                \item If $\Sigma(r'') = p$ and $t = p-1$, then, since $\Sigma(r'') \geq p$, we can write $t+2 = s'_0 + s'_1 + \cdots + s'_m$ with $s'_j$ in $\{ 0, \ldots, r_j \}$ for $j = 0, 1, 2, \ldots$ and $s'_0 < 2$.
                \item
                  Otherwise we can write $T = t+2 = s'_0 + s'_1 + \cdots + s'_m$ with $s'_j$ in $\{ 0, \ldots, r_j \}$ for $j = 0, 1, 2, \ldots$.
              \end{itemize}
              Put $T' = s_0' + s'_1 p + \dotsb$ and $t' = T'-2$.
              Then $t' \equiv t \mod (p-1)$ and $\binom{r-2}{t'+2} \nequiv 0$ by Lucas' Theorem.
              Since $D_t, E_t = 0$, we conclude by \cref{eq:linear-equations}, that in either case $F_t = 0$.
              \qedhere
          \end{itemize}
      \end{enumerate}
    \end{proof}

    \begin{lem}
      \label{lem:Xr2FullDimGeneralPDivR}
      Let $p > 3$ and write $r = p^n u$ for $n \geq  1$ such that $p \nmid u$.
      If $\Sigma(u)$ is non-minimal, then $\dim X_{r-2} = 3p + 3$.
    \end{lem}
    \begin{proof}
      For every $x$ in $\N$ put
      \[
        r(x)
        :=
        x p^n - 2
        =
        x p^n - p^n + p^n - 2
        =
        p^n (x-1) + (p-1)[p^{n-1} + \dotsb + p] + (p-2).
      \]
      We notice that $r(x) \equiv x-2 \mod (p-1)$.
      Expand $p$-adically $u = u_0 + u_1 p + u_2 p^2 + \dotsb$ with $u_0$, $u_1$, $u_2$, \ldots in $\left\{ 0, \ldots, p-1 \right\}$ and $u_0 > 0$.
      Then
      \[
        r-2
        =
        r(u)
        =
        [(u_0-1) + u_1 p + u_2 p^2 + \dotsb] p^n + (p-1) (p^{n-1} + \dotsb + p) + (p-2).
      \]
      Using the notation of \cref{lem:Xr2FullDimGeneralPNotDivRpR}, we will show that \cref{eq:linear-equations} forces $D_{t'}$ and $E_{t''}$ or $F_{t''}$ to vanish for $t'$ and $t''$ in full sets of representatives of $\{ 1, \ldots, p-1 \}$.
      That is, for every $t$ in $ \left\{ 0, \ldots, p-2 \right\} $ there is $t'$ and $t''$ with $t' \equiv t$ and $t'' \equiv t \mod (p-1)$ such that $D_{t'}$ and $E_{t''}$ vanish.

      \begin{enumerate}[{Case} 1.]
        \item Suppose $t \in \left\{ 0, \ldots, u_0-3 \right\}$.
          Let $i$ be the smallest index $> 0$ such that $u_i > 0$ (which exists because $u_0 \leq p-1$ and $\Sigma(u) \geq p$).
          Put $t' = r(t+2)$ and $t'' = r(t+1 + p^i)$
          Then $t'$ and $t'' \equiv t \mod (p-1)$.
          By Lucas' Theorem,
          \begin{itemize}
            \item we have $\binom{r-2}{t'} \equiv \binom{u_0-1}{t+1} \neq 0$ and $\binom{r-2}{t''} \equiv u_i \binom{u_0-1}{t} \neq 0$,
            \item we have $\binom{r-2}{t'+2} \equiv \binom{u_0-1}{t+2} \neq 0$ and $\binom{r-2}{t''+2} \equiv u_i \binom{u_0-1}{t+1} \neq 0$, and
            \item we have $\binom{r-2}{t'+1} \equiv 0$ and $\binom{r-2}{t''+1} \equiv 0$.
          \end{itemize}
          Therefore $(+)$ yields modulo $p$ the system of equations:
          \begin{alignat}{4}
            \binom{r-2}{t'}  D_t & +{} & \binom{r-2}{t'+2}  & F_t & {}\equiv 0 \\
            \binom{r-2}{t''} D_t & +{} & \binom{r-2}{t''+2} & F_t & {}\equiv 0
          \end{alignat}
          To see that $D_t = F_t = 0$, we will prove that the determinant of the matrix $M$ attached to this system of equations is nonzero, that is,
          \[
            \norm{M}
            \equiv
            \begin{vmatrix}
              \binom{r-2}{t'}  &  \binom{r-2}{t'+2} \\
              \binom{r-2}{t''} &  \binom{r-2}{t''+2}
            \end{vmatrix}
            \neq 0 \mod p.
          \]
          Putting $u'_0  = u_0-1$, by \cite[(2.17)]{Krattenthaler:DeterminantFormulas},
          \begin{align}
            \begin{vmatrix}
              \binom{r-2}{t'}  &  \binom{r-2}{t'+2} \\
              \binom{r-2}{t''} &  \binom{r-2}{t''+2}
            \end{vmatrix}
            & \equiv
            u_i
            \begin{vmatrix}
              \binom{u_0'}{t+1}  &  \binom{u_0'}{t+2} \\
              \binom{u_0'}{t} &  \binom{u_0'}{t+1}
            \end{vmatrix}\\
            & =
            u_i
            \prod_{i=1,2} \prod_{j=1, \ldots, t+1} \prod_{k = 1, \ldots, u_0' - (t+1)} \frac{i+j+k-1}{i+j+k-2}
            \mod p
          \end{align}
          For this product to be nonzero, every factor has to be nonzero.
          Because $j+k \leq u'_0$, we have $i+j+k-1$ in $ \{ 2, \ldots, u_0'+1 \}$.
          This set does not contain $0$ in $\F_p$ if and only if $u_0' < p-1$.
          Because $u_0 \leq p-1$, we have $u'_0 = u_0-1 < p-1$, and conclude $\norm{M} \neq 0 \mod p$.
          That is, $D_t = F_t = 0$.

          To see that $E_t = 0$, put $t' = r(t+1)+1$.
          Then $\binom{r-2}{t'+1} \not\equiv 0$.
        \item
          Suppose either $u_0 = 1$ or, otherwise, $t \in \left\{ u_0-2, \ldots, p-2 \right\}$.
          \begin{itemize}
            \item
              To show $D_t = 0$, we choose $t'$ with $t' \equiv t \mod (p-1)$ as follows:
              Because by assumption $\Sigma(u) = u_0 + u_1 + \dotsb + u_m \geq p$ and $t \leq p-2$, we can write $t+2 = u_0 + s_1 + \dotsb + s_m$ with $s_j$ in $ \{ 0, \ldots, u_j \} $ for $j = 1, \ldots, m$.
              Put $t' = r(u_0 + s_1 p + \dotsb + s_m p^m)$.
              Then $t' \equiv t \mod (p-1)$.
              We have $\binom{r-2}{t'} \nequiv 0 \mod p$ and $\binom{r-2}{t'+1}, \binom{r-2}{t'+2} \equiv 0 \mod p$ by Lucas' Theorem.
              By \cref{eq:linear-equations}, we conclude $D_t = D_{t'} = 0$.
            \item
              To show $E_t$ or $F_t = 0$, we choose $t'$ with $t' \equiv t \mod (p-1)$ as follows:
              \begin{enumerate}[{label*=\arabic*.}]
                \item
                  We have $t \leq p-3$:
                  Because by assumption $\Sigma(u) = u_0 + \dotsb + u_m \geq p$ and $t \leq p-3$, we can write $t+3 = u_0 + s'_1 + \dotsb$ with $s'_j$ in $ \{ 0, \ldots, u_j \} $ for $j = 1, \ldots, m$.
                  Put $t' = r(u_0 + s'_1 p + \dotsb + s'_m p^m) - 1$.
                  Then $t' \equiv t \mod (p-1)$.
                  By Lucas' Theorem, $\binom{r-2}{t'+1} \nequiv 0$ and $\binom{r-2}{t'+2} \equiv 0 \mod p$.
                \item
                  We have $t = p-2$:
                  \begin{enumerate}[{label*=\arabic*.}]
                    \item
                      If $n = 1$ and $u_0 > 1$ or $n > 1$, then $\binom{r-2}{t+2} \nequiv 0 \mod p$ by Lucas' Theorem.
                      In addition, $\binom{r-2}{t+1} \equiv 0$.
                    \item
                      If $n = 1$ and $u_0 = 1$, then let $i$ be the smallest index $> 0$ such that $u_i > 0$ (which exists because $\Sigma(u) \geq p$).
                      Let $t' = r(p^i)$.
                      Then $t' \equiv p-2 = t \mod (p-1)$.
                      We have $\binom{r-2}{t'+1} \equiv 0$ and $\binom{r-2}{t'+2} \nequiv 0 \mod p$ by Lucas' Theorem.
                  \end{enumerate}
              \end{enumerate}
              Because $D_t = 0$, we conclude by \cref{eq:linear-equations} that $F_t = 0$.
              \qedhere
          \end{itemize}
      \end{enumerate}
    \end{proof}

    \begin{lem}
      \label{lem:Xr2FullDimGeneralPDivRp}
      Let $p > 3$ and write $r-1 = p^n u$ for $n \geq  1$ such that $p \nmid u$.
      If $\Sigma(u)$ is non-minimal, then $\dim X_{r-2} = 3p + 3$.
    \end{lem}
    \begin{proof}
      For every $x$ in $\N$ such that $p \nmid x$, define
      \[
        r(x)
        :=
        x p^n - 1
        =
        (x p^n - p^n) + p^n - 1
        =
        (p^n (x-1)) + (p-1)[p^{n-1} + \dotsb + p + 1]
      \]
      We notice that $r(x) \equiv x-1 \mod (p-1)$.
      Expand $p$-adically $u = u_0 + u_1 p + u_2 p^2 + \dotsb$ with $u_0$, $u_1$, $u_2$, \ldots in $\left\{ 0, \ldots, p-1 \right\}$ and $u_0 > 0$.
      Then
      \[
        r-2
        =
        r(u)
        =
        [(u_0-1) + u_1 p + u_2 p^2 + \dotsb] p^n + (p-1) (p^{n-1} + \dotsb + p + 1).
      \]
      Using the notation of \cref{lem:Xr2FullDimGeneralPNotDivRpR}, we will show that \cref{eq:linear-equations} forces $D_{t'}$ and $E_{t''}$ to vanish for $t'$ and $t''$ in full sets of representatives of $\{ 1, \ldots, p-1 \}$.
      That is, for every $t$ in $ \left\{ 0, \ldots, p-2 \right\} $ there is $t'$ and $t''$ with $t' \equiv t$ and $t'' \equiv t \mod (p-1)$ such that $D_{t'}$ and $E_{t''}$ vanish.

      \begin{enumerate}[{Case} 1.]
        \item Suppose $t \in \left\{ 0, \ldots, u_0-2 \right\}$.

          As in \cref{lem:Xr2FullDimGeneralPNotDivRpR}, we choose $t'$, $t''$ and $t''' \equiv t \mod (p-1)$ such that \cref{eq:linear-equations} yields modulo $p$ the system of equations
          \begin{alignat}{4}
            \binom{r-2}{t'}  D_t & +{} & \binom{r-2}{t'+1} & E_t +{} & \binom{r-2}{t'+2} &  F_t  & {}\equiv 0 \\
            \binom{r-2}{t''} D_t & +{} & \binom{r-2}{t''+1} & E_t +{} & \binom{r-2}{t''+2} & F_t & {}\equiv 0\\
            \binom{r-2}{t'''}   D_t & +{} & \binom{r-2}{t'''+1} & E_t +{} & \binom{r-2}{t'''+2} &   F_t   & {}\equiv 0
          \end{alignat}
          and prove that the determinant of the matrix $M$ attached to this system of equations is nonzero, that is,
          \[
            \norm{M}
            =
            \begin{vmatrix}
              \binom{r-2}{t'}  & \binom{r-2}{t'+1}  &  \binom{r-2}{t'+2} \\
              \binom{r-2}{t''} & \binom{r-2}{t''+1} &  \binom{r-2}{t''+2} \\
              \binom{r-2}{t'''}   & \binom{r-2}{t'''+1}   &  \binom{r-2}{t'''+2}
            \end{vmatrix}
            \not\equiv 0 \mod p.
          \]
          Put $t' = p^n t$, $t'' = r(p^i + t + 1) - 1$, $t''' = r(t+1)$ for the smallest $i > 0$ such that $u_i > 0$ (which exists because $u_0 \leq p-1$ and $\Sigma(u) \geq p$).
          Then $t'$, $t''$ and $t''' \equiv t \mod (p-1)$.
          By Lucas' Theorem, with $u' = u_0 - 1$,
          \begin{itemize}
            \item
              we have $\binom{r-2}{t'} \equiv \binom{u'}{t}$, $\binom{r-2}{t'+1} \equiv (p-1) \binom{u'}{t}$ and $\binom{r-2}{t'+2} \equiv \binom{p-1}{2} \binom{u'}{t}$,
            \item
              we have $\binom{r-2}{t''} \equiv u_i (p-1) \binom{u'}{t}$, $\binom{r-2}{t''+1} \equiv u_i \binom{u'}{t}$ and $\binom{r-2}{t''+2} \equiv u_i \binom{u'}{t+1}$,
            \item
              we have $\binom{r-2}{t'''} \equiv \binom{u'}{t}$, $\binom{r-2}{t'''+1} \equiv \binom{u'}{t+1}$ and $\binom{r-2}{t'''+2} \equiv (p-1) \binom{u'}{t+1}$,
          \end{itemize}
          Therefore,
          \begin{align}
            \norm{M}
              & \equiv
              u_i
              \begin{vmatrix}
                \binom{u'}{t}   & (p-1) \binom{u'}{t}   &  \binom{p-1}{2} \binom{u'}{t} \\
                (p-1) \binom{u'}{t} & \binom{u'}{t} &  \binom{u'}{t}\\
                \binom{u'}{t}  & \binom{u'}{t}  &  (p-1) \binom{u'}{t}
              \end{vmatrix}\\
              & =
              u_i
              \begin{vmatrix}
                \binom{u'}{t}   & - \binom{u'}{t}   &  \binom{u'}{t} \\
                - \binom{u'}{t} & \binom{u'}{t} &  \binom{u'}{t}\\
                \binom{u'}{t}  & \binom{u'}{t}  &  - \binom{u'}{t}
              \end{vmatrix}\\
              & =
              u_i
              \begin{vmatrix}
                0  & 0   &  \binom{u'+1}{t+1} \\
                - \binom{u'}{t} & \binom{u'}{t} &  \binom{u'}{t}\\
                \binom{u'}{t}  & \binom{u'}{t}  &  - \binom{u'}{t}
              \end{vmatrix}\\
              & =
              u_i
              \binom{u'+1}{t+1} \binom{u'}{t}
              \left[ -\binom{u'}{t} - \binom{u'}{t} \right]
              =
              - u_i
              \binom{u'}{t} \binom{u'+1}{t+1}^2
              \mod p.
          \end{align}
          Because $t < u' < p-1$, we have $\norm{M} \neq 0$.
        \item
          Suppose $t \in \left\{ u_0-1, \ldots, p-2 \right\}$.
          \begin{itemize}
            \item
              To show $D_t = 0$, we choose $t'$ with $t' \equiv t \mod (p-1)$ as follows:
              Because by assumption $\Sigma(u) = u_0 + u_1 + \dotsb \geq p$ and $u_0 \leq t+1 \leq p-1 \leq p$, we can write $t+1 = u_0 + s_1 + \dotsb$ with $s_j$ in $ \{ 0, \ldots, u_j \} $ for $j = 1, 2, \ldots$.
              Put $t' = r(u_0 + s_1 p + \dotsb)$.
              Then $t' \equiv t \mod (p-1)$.
              By Lucas' Theorem, $\binom{r-2}{t'} \nequiv 0$ but $\binom{r-2}{t'+1}$ and $\binom{r-2}{t'+2} \equiv 0 \mod p$.
              By \cref{eq:linear-equations}, we conclude $D_t \equiv D_{t'} = 0 \mod p$.
            \item
              To show $E_t$ or $F_t = 0$, we choose $t'$ with $t' \equiv t \mod (p-1)$ as follows:
              Because by assumption $\Sigma(u) = u_0 + _1 + \dotsb \geq p$ and $u_0 \leq t+2 \leq p$, we can write $t+2 = u_0 + s_1 + \dotsb + s_m$ with $s_j$ in $ \{ 0, \ldots, u_j \} $ for $j = 1, 2, \ldots$.
              Put $t' = r(u_0 + s'_1 p + \dotsb) - 1$.
              Then $t' \equiv t \mod (p-1)$.
                by Lucas' Theorem, $\binom{r-2}{t'}$ and $\binom{r-2}{t'+1} \nequiv 0$, but $\binom{r-2}{t'+2} \equiv 0 \mod p$.
              Because $D_t \equiv 0 \mod p$, we conclude by \cref{eq:linear-equations} that $E_t \equiv E_{t'} = 0 \mod p$.
              \qedhere
          \end{itemize}
      \end{enumerate}
    \end{proof}

    \begin{cor}
      Let $p > 3$.
      If $r \geq 3p+2$ and $\Sigma(r'')$, $\Sigma(r')$ and $\Sigma(r)$ are non-minimal, then $X_{r-2} \iso X_{r''} \otimes V_2$ and its Jordan-Hölder series is that of \cref{prop:JHXrppxV2}.
    \end{cor}
    \begin{proof}
      If $\Sigma(r'')$, $\Sigma(r')$ and $\Sigma(r)$ are non-minimal, then by the preceding \cref{lem:Xr2FullDimGeneralPNotDivRpR}, \cref{lem:Xr2FullDimGeneralPDivR} and \cref{lem:Xr2FullDimGeneralPDivRp}, the dimension of $X_{r-2}$ is equal to that of $X_{r''} \otimes V_2$, hence the natural epimorphism
      $
        X_{r''} \otimes V_2
        \twoheadrightarrow
        X_{r-2}
      $
      is an isomorphism.
    \end{proof}

    \subsection{Sum of the Digits of $r-2$ is minimal}
    \label{sec:sigma-minimal}

    Let $a$ in $\{ 3, \ldots, p + 1 \}$ such that $r \equiv a \mod (p-1)$.
    Let $r'' = r - 2$.
    We assume in this \cref{sec:sigma-minimal} that $\Sigma(r'')$ is minimal, that is, $\Sigma(r'') < p$, or, equivalently, $\Sigma(r'') = a-2$.

    If $r$ satisfies the conditions of \cref{lem:XrInXr1InXr2}, that is, $r \leq p$ or $r = p^n + r_0$ where $r_0$ in $ \{2, \ldots, p-1 \}$ and $n > 0$, then the inclusion $X_{r-1} \subseteq X_{r-2}$ is an equality.
    Therefore, the Jordan-Hölder series of $X_{r-2} = X_{r-1}$ is known
    \begin{itemize}
      \item for $a = p$ by \cite[Proposition 3.13]{BG}, and
      \item for $a = 2, \ldots, p-1$ by \cite[Proposition 4.9]{BG}.
    \end{itemize}

    Otherwise, $X_{r-2}$ has at least three distinct Jordan-Hölder factors by \cref{lem:XrInXr1InXr2}:
    By \cref{prop:XrAst0IffSrMin} and \cref{lem:JHXr},
    \[
      X_{r''} = V_{a-2}
      \quad \text{ and }
      X_{r''}^* = 0.
    \]
    By \cref{lem:Xr2V2xXrpp}, there is thus an $\F_p[M]$-linear surjection
    \begin{equation}
      \label{eq:Va2V2ToXr2}
      \phi : V_{a-2} \otimes V_2 \twoheadrightarrow X_{r-2}
    \end{equation}

      \subsubsection{$r \equiv 3 \mod (p-1)$}

      \begin{prop}
        Let $r \geq p$.
        If $r \equiv 3 \mod (p-1)$ and $\Sigma(r'') < p$, then
        \[
          V_1 \otimes V_2
          \cong
          V_1 \otimes \rD \oplus V_3 \iso X_{r-2}.
        \]
      \end{prop}
      \begin{proof}
        For $a=3$ the right-hand side of \cref{eq:Va2V2ToXr2} is $V_1 \otimes V_2$.
        By \cite[Prop. 4.9]{BG}
        \[
          V_1 \otimes V_2
          =
          V_1 \otimes \rD \oplus V_3
          \twoheadrightarrow
          X_{r-2}.
        \]
        That is, there is an epimorphism with only two Jordan-Hölder factors
        onto $X_{r-2}$.
        Because $r \geq p$, by \cref{lem:XrInXr1InXr2}.(i) we have $0 \neq X_r \neq X_{r-1}$, therefore $X_{r-2}$ has at least two Jordan-Hölder factors;
        therefore this epimorphism must be an isomorphism.
      \end{proof}

      Alternatively, if $r \equiv 3 \mod (p-1)$ and $\Sigma(r'')$ is minimal, that is, $\Sigma(r'') = 1$, then $r = p^n + 2$.
      In particular, $r$ satisfies the conditions of \cref{lem:XrInXr1InXr2}, and the inclusion $X_{r-1} \subseteq X_{r-2}$ is an equality.
      By \cref{lem:JHTensor}.\ref{en:JHTensor0p-1},
      \[
        V_1 \otimes V_2 = V_1 \otimes \rD \oplus V_3 \iso X_{r-1}.
      \]

      \subsubsection{$r \equiv 4, \ldots, p-1 \mod (p-1)$}

      Let $a$ in $\{ 4, \ldots, p-1 \}$ such that $r \equiv a \mod (p-1)$.
      By \cref{lem:digits-minimality}, if $\Sigma(r'')$ is minimal, then $\Sigma(r')$ and $\Sigma(r)$ are minimal, too.

      \begin{prop}
        \label{prop:dimXr-2min}
        Let $p > 2$.
        Let $a$ in $\left\{ 3, \ldots, p-1 \right\}$ such that $r-2 \equiv a-2 \mod (p-1)$ and $r \geq p$.
        Let $\Sigma(r'') < p$.
        \begin{enumerate}
          \item If $r = p^n + r_0$ where $r_0 = a-1$ and $n > 0$, then
            \[
              X_{r-2}
              =
              V_{a-2} \otimes \rD \oplus V_a,
            \]
          \item otherwise,
            \[
            X_{r-2}
            \cong
            V_a \oplus (V_{a-2} \otimes \rD) \oplus (V_{a-4} \otimes \rD^2).
          \]
      \end{enumerate}

    \end{prop}
    \begin{proof}
      If $r \equiv 3 \mod (p-1)$ and $\Sigma(r'')$ is minimal, that is, $\Sigma(r'') = 1$, then $r = p^n + 2$.
      In particular, $r$ satisfies the conditions of \cref{lem:XrInXr1InXr2}, and the inclusion $X_{r-1} \subseteq X_{r-2}$ is an equality, and by \cite[Lemma 4.5]{BG},
      \[
        V_1 \otimes V_2 = V_1 \otimes \rD \oplus V_3 \iso X_{r-1}.
      \]
      Let $a$ in $ \left\{ 4, \ldots, p-1 \right\}$ such that $r-2 \equiv a-2 \mod (p-1)$.
      If $r = p^n + r_0$ where $r_0$ in $\{2, \ldots, p-1 \}$ and $n > 0$, then the inclusion $X_{r-1} \subseteq X_{r-2}$ is an equality and, by \cite[Proposition 4.9.(i)]{BG}
      \[
        X_{r-2}
        =
        X_{r-1}
        =
        V_{a-2} \otimes \rD \oplus V_a.
      \]
      Otherwise, \cref{eq:Va2V2ToXr2} becomes by \cref{prop:JHXrppxV2},
      \[
        V_{a-2} \otimes V_2
        =
        V_{a} \oplus (V_{a-2} \otimes \rD) \oplus (V_{a - 4} \otimes \rD^2) \twoheadrightarrow X_{r-2}.
      \]
      By \cref{lem:XrInXr1InXr2} the right-hand side has at least three Jordan-Hölder factors.
      Because the map is surjective, these are exhausted by those of the left-hand side.
      Thus the surjection is a bijection.
    \end{proof}

    \subsubsection{$r \equiv p \mod (p-1)$}

    If $a = p$, then $\Sigma(r'')$ is minimal if and only if $\Sigma(r'') = p-2$.
    Therefore, as observed in \cref{lem:digits-minimality}, indeed $\Sigma(r') = p-1$ is minimal, but $\Sigma(r) = p$ is non-minimal!

    \begin{prop}
      Let $r \geq p$ and $r \equiv p \mod (p-1)$.
      Let $\Sigma(r'') < p$.
      \begin{enumerate}
        \item If $r = p^n + (p-1)$, then
          \[
            X_{r-2} \iso V_{2p-1},
          \]
        \item otherwise,
          \[
            X_{r-2} \cong V_{p-4} \otimes \rD^2 \oplus V_{2p-1}.
          \]
      \end{enumerate}
    \end{prop}
    \begin{proof}
      Because $\Sigma(r') = p-1 < p$ is minimal, by \cite[Prop. 3.3.(i)]{BG}
      \[
        X_{r-1} \iso V_{2p-1}.
      \]
      If $r = p^n + (p-1)$, then the inclusion $X_{r-1} \subseteq X_{r-2}$ is an equality.

      Otherwise, by \cref{prop:XrAst0IffSrMin}, we have $X_{r''}^* = 0$.
      Therefore \cref{eq:Va2V2ToXr2} becomes by \cref{prop:JHXrppxV2},
      \[
        V_{p-4} \op V_{2p-1} \twoheadrightarrow X_{r-2}
      \]
      where $V_{2p-1}$ has successive semisimple Jordan-Hölder factors $V_{p-2} \otimes \rD$, $V_1$ and $V_{p-2} \otimes \rD$.
      Because $\Sigma(r) = p$ is non-minimal, $X_r^* \neq 0$ by \cref{prop:XrAst0IffSrMin}.
      Therefore, by \cref{lem:XrInXr1InXr2} (as we assume $r$ not to be of the form $r = p^n + r_0$ for some $n \geq 1$ and $r_0 < p$), there are proper inclusions
      \[
        0 \subset X_r^* \subset X_r \subset X_{r-1} \subset X_{r-2}.
      \]
      In particular, $X_{r-2}$ has at least $4$ Jordan-Hölder factors.
      Therefore, all $4$ Jordan-Hölder factors of the left-hand side must appear on the right-hand side of the epimorphism $V_{p-4} \otimes \rD^2 \oplus V_{2p-1} \twoheadrightarrow X_{r-2}$;
      therefore, it must be an isomorphism.
    \end{proof}

    \subsubsection{$r \equiv p + 1 \mod (p-1)$}

    If $a = p+1$, then $\Sigma(r'')$ is minimal if and only if $\Sigma(r'') = p-1$.
    Therefore, as observed in \cref{lem:digits-minimality}:
    $\Sigma(r') = p$ is not minimal.
    If $r = p^n + p$ for some $n > 1$, then $\Sigma(r) = 2$ is minimal.
    Otherwise, if $r \neq p^n + p$ for any $n > 1$, then $\Sigma(r) = p+1$ is not minimal.

    \begin{obs*}
      We have $\dim X_{r-2} / X_{r-1} \leq p+1$.
      To see this, let
      \[
        X_{r'} \otimes V_1 \twoheadrightarrow X_{r-1}
        \quad \text{ and } \quad
        X_{r''} \otimes V_2 \twoheadrightarrow X_{r-2}
      \]
      be the natural $\F[M]$-linear epimorphisms given by multiplication.
      Let $X_{r-1} \to X_{r-2}$ be the inclusion and
      \[
        X_{r'} \otimes V_1 \to X_{r''} \otimes V_2
      \]
      the $\F[M]$-linear monomorphism given by $X^{r'} \otimes Y \mapsto X^{r''} \otimes XY$.
      The diagram
      \[
      \begin{CD}
        X_{r''} \otimes V_2 @ > >> X_{r-2} \\
        @AAA@AAA\\
        X_{r'} \otimes V_1 @ > >> X_{r-1}
      \end{CD}
    \]
    commutes as, by $\F[M]$-linearity, it suffices to check that $X^{r'} \otimes Y \mapsto X^{r''} \cdot XY$ either way.
    Therefore the image of $X_{r'} \otimes V_1$ taking the left upper route (that is, under the mapping $X_{r'} \otimes V_1 \to X_{r''} \otimes V_2 \to X_{r-2}$) is included in $X_{r-1}$ inside $X_{r-2}$.
    Therefore the surjection
    \[
      X_{r''} \otimes V_2 / X_{r'} \otimes V_1
      \twoheadrightarrow
      X_{r-2} / X_{r-1}
    \]
    is well defined.
    Since the left-hand side has dimension $\leq p+1$, so the right-hand side as well.
    \end{obs*}

    \begin{prop}[Extension of {\cite[Proposition 3.3]{BG}}]
    \label{prop:4JHIsoXr-2}
      Let $r \geq p$ and $r \equiv p+1 \mod (p-1)$.
      If $r = p^n + p$ for some $n$, then $X_{r-2}$ has four Jordan-Hölder factors of $V_{3p-1}$, missing one of $V_{p-3} \otimes \rD^2$.
    \end{prop}
    \begin{proof}
      By \cref{prop:XrAst0IffSrMin}, we have $X_{r''}^* = 0$.
      Therefore \cref{eq:Va2V2ToXr2} becomes by \cref{prop:JHXrppxV2},
      \[
        V_{3p-1} \twoheadrightarrow X_{r-2}
      \]
      We recall that by \cref{cor:JHVxV2}, the successive semisimple Jordan-Hölder factors of the $\F_p[M]$-module $V_{3p-1}$ are $V_{3p-1} = U_2 \op (U_0 \otimes \rD)$ where
      \begin{itemize}
        \item we have $U_0 = V_{p-1}$, and
        \item the $\F_p[M]$-module $U_2$ has successive semisimple Jordan-Hölder factors $V_{p-3} \ox \rD^2$, $(V_0 \otimes \rD) \op V_2$ and $V_{p-3} \ox \rD^2$.
      \end{itemize}
      In particular, $V_{3p-1}$ has $5$ Jordan-Hölder factors.

      By \cite[Proposition 4.9.(ii)]{BG},
      \[
        0 \to
        V_{p-1} \otimes \rD
        \to
        X_{r-1}
        \to
        V_0 \otimes \rD \oplus V_2
        \to 0
      \]
      In particular, $X_{r-1}$ has $3$ Jordan-Hölder factors.

      Because $r \equiv p+1 \mod (p-1)$, impossibly $r = p^n + r_0$ for $1 < r_0 < p$.
      Hence, by \cref{lem:XrInXr1InXr2},
      \[
        X_{r-1} \subset X_{r-2}.
      \]
      Therefore $X_{r-2}$ has at least $4$ (and at most $5$) Jordan-Hölder factors.
      Since $\dim X_{r-2} / X_{r-1} \leq p+1$ by the preceding observation, only one of the $V_{p-3} \otimes \rD^2$ Jordan-Hölder factors can be in $X_{r-2} / X_{r-1}$.
      We conclude that $X_{r-2}$ has exactly $4$ Jordan-Hölder factors.
    \end{proof}

    \begin{prop}[Extension of {\cite[Proposition 3.3]{BG}}]
    \label{prop:V3p-1IsoXr-2}
      Let $r \geq p$ and $r \equiv p+1 \mod (p-1)$.
      If $r \neq p^n + p$ for any $n$ and $\Sigma(r'') < p$, then
      \[
        X_{r-2} \cong V_{3p-1}.
      \]
    \end{prop}
    \begin{proof}
      By \cref{prop:XrAst0IffSrMin}, we have $X_{r''}^* = 0$.
      Therefore \cref{eq:Va2V2ToXr2} becomes by \cref{prop:JHXrppxV2},
      \[
        V_{3p-1} \twoheadrightarrow X_{r-2}
      \]
      We recall that by \cref{lem:JHTensor}.\ref{en:JHTensorp2p-1}, the successive semisimple Jordan-Hölder factors of the $\F_p[M]$-module $V_{3p-1}$ are $V_{3p-1} = U_2 \op (U_0 \otimes \rD)$ where
      \begin{itemize}
        \item we have $U_0 = V_{p-1}$, and
        \item the $\F_p[M]$-module $U_2$ has successive semisimple Jordan-Hölder factors $V_{p-3} \ox \rD^2$, $(V_0 \otimes \rD) \op V_2$ and $V_{p-3} \ox \rD^2$.
      \end{itemize}
      In particular, $V_{3p-1}$ has $5$ Jordan-Hölder factors.

      Write $r = p^n u$ where $p$ does not divide $u$.
      We have $\Sigma(u) = \Sigma(r) \equiv 2 \mod (p-1)$.
      If $\Sigma(u - 1) < p$, that is, $\Sigma(u - 1) = 1$, then either $r = p^n + p$ for some $n$ or $\Sigma(r'') \geq p$, in contradiction to our assumptions.
      Therefore, we may apply \cite[Proposition 4.9.(iii)]{BG}, yielding
      \[
        0 \to
        V_{p-3} \otimes \rD^2 \oplus V_{p-1} \otimes \rD
        \to
        X_{r-1}
        \to
        V_0 \otimes \rD \oplus V_2
        \to 0
      \]
      In particular, $X_{r-1}$ has $4$ Jordan-Hölder factors.

      Because $r \equiv p+1 \mod (p-1)$, impossibly $r = p^n + r_0$ for $1 < r_0 < p$.
      Hence, by \cref{lem:XrInXr1InXr2},
      \[
        X_{r-1} \subset X_{r-2}.
      \]
      Therefore $X_{r-2}$ has at least $5$ Jordan-Hölder factors.
      Hence, all $5$ Jordan-Hölder factors of the left-hand side must appear on the right-hand side of the epimorphism $V_{3p-1} \twoheadrightarrow X_{r-2}$ and thus it is an isomorphism.
    \end{proof}


\section{Vanishing conditions on the singular quotients of $X_{r-2}$}
\label{sec:singular-quotients-x-r-2}

In this section, we study the singular quotients of $X_{r-2}$, that is, whether  $X_{r-2}^{*}/X_{r-2}^{**}$, $X_{r-2}^{**}/X_{r-2}^{***}$ and $X_{r-2}^{*}/X_{r-2}^{***}$ are zero or not by applying \cref{lem:VrAstCriteria} and \cref{lem:binomialsumaeq2}.
In correspondence with \cref{lem:VrAstQuotients}, we will choose $a$ such that $r \equiv a \mod (p-1)$ for $X_{r-2}^*/X_{r-2}^{**}$ in the range $\{ 3, \ldots, p+1 \}$, whereas for $X_{r-2}^{**}/X_{r-2}^{***}$ in $\{ 5, \ldots, p+3 \}$.

\begin{lemma}
  \label{lem:Xr2r4p-1AstModAstAstEqualModP}
  Let $a \in \{ 4, \ldots, p \}$.
  If $r > p$ and $r \equiv a \mod (p-1)$ and $r \equiv a \mod p$, then
  \[
    0  =
    \begin{cases}
      X_{r-2}^*/X_{r-2}^{**}, & \text{\quad if $a = 4$}\\
      X_{r-2}^*/X_{r-2}^{***}, & \text{\quad if $5 \leq a \leq p$}.
    \end{cases}
  \]
\end{lemma}
\begin{proof}
  The proof is similar to the proof of \cite[Lemma 6.2]{BG}:
  Consider $\sum_{k \in \mathbb{F}_p} k^{p-2} (kX + Y)^r \in X_r$.
  Working $\mod p$:
  \begin{align}
    \sum_{k\in \mathbb{F}_p} k^{p-3} (kX + Y)^r
    \equiv & - \sum_{\substack{0 < j \leq r -2 \\ j \equiv a-2 \mod (p-1)}} \binom{r}{j}X^{r-j}Y^j\\
    yequiv - \binom{r}{2} X^2Y^{r-2} - G(X, Y)
  \end{align}
  where we claim that
  \[
    G(X,Y)
    \equiv \sum_{\substack{0 < j < r -2 \\ j \equiv a-2 \mod (p-1)}} \binom{r}{j}X^{r-j}Y^j
    \in
    \begin{cases}
      V_r^{**},	  & \text{ \quad for } a = 4,\\
      V_r^{***},	& \text{ \quad for }5 \leq a \leq p.
    \end{cases}
  \]
  Proof of our claim:
  Let $c_j$ denote the coefficients of $G$.
  If $a \geq 5$, we find $c_j = 0$ for $j = 0,1,2$ and $j = r-2,r-1,r$.
  If $a=4$, then $c_j = 0$ for $j = 0,1$ and $j = r-2,r-1,r$, but $c_2 \ne 0$.
  By \cref{lem:binomialsumaeq2} for $i = 2$ we have $\sum c_j, \sum j c_j, \sum j(j-1) c_j \equiv 0 \mod p$ to obtain $G(X,Y) \in V_r^{***}$ for $a \geq 5$ and $G(X,Y) \in V_r^{**}$ for $a = 4$:
  Therefore $\binom{r}{2} X^2Y^{r-2}$ is in $X_r + V_r^{***}$ for $a \geq 5$ and in $X_r + V_r^{**}$ for $a = 4$.
  Since the cases $a = p + 1$ and $a = p+2$ are excluded, neither $r \equiv a \equiv 1 \mod p$ nor $r \equiv a \equiv 2 \mod p$, and we conclude $X_{r-2} \subseteq X_r + V_r^{***}$ for $a \geq 5$ and $X_{r-2} \subseteq X_r + V_r^{**}$ for $a = 4$.

  By \cref{lem:analog47}, we have $X^*_r = X_r^{***}$.
  Now by following the argument at the end of the proof of \cite[Lemma 6.2]{BG}, we conclude $X_{r-2}^* \subseteq X_{r-2}^{***}$ for $a \geq 5$ and $X_{r-2}^* \subseteq X_{r-2}^{**}$ for $a = 4$.
\end{proof}

\subsection{$X_{r-2}^{*}/X_{r-2}^{**}$}

  \begin{lemma}
    \label{lem:Xr2AstModXr2AstAstUnequalModP}
    
    Let $a = 4, \ldots, p$ and $r \equiv a \mod (p-1)$.
    If $r \geq 2p + 1$ and $r \not\equiv a \mod p$, then
    \[
      X_{r-2}^*/X_{r-2}^{**} = V_{a-2} \otimes \rD.
    \]
  \end{lemma}
  \begin{proof}
    Consider the polynomial
    \begin{align}
      F(X,Y) & = (a-2)X^{r-1}Y + \sum_{k\in \mathbb{F}_p} k^{p + 2-a}(kX +Y)^{r-2}X^2 \in X_{r-2} \\
             & \equiv (a-r)X^{r-1}Y - \sum_{\substack{0 < j < r-3 \\ j \equiv a-3 \mod (p-1)} }\binom {r-2}{j} X^{j + 2}Y^{r-2-j} \mod p.
    \end{align}
    By \cref{lem:VrAstCriteria} we see $F(X,Y) \in V_r^*$ but the coefficient $c_1$ of $X^{r-1}Y$ in $F(X,Y)$ is $ a-r  \not\equiv 0\mod p$ by the hypothesis, so $F(X,Y) \not \in V_r^{**}$.
    Thus $      X_{r-2}^*/X_{r-2}^{**} \ne 0$.

    Since $\phi (X_{r''}^* \otimes V_2 ) \subseteq X_{r-2}^{**}$, the Jordan-Hölder factors are in (the non-singular part of) the right-hand side of the short exact sequence of \cref{prop:JHXrppxV2}.
    Using \cref{lem:VrAstQuotients}.(ii), the only possible Jordan-Hölder factor is $X_{r-2}^*/X_{r-2}^{**} = V_{a-2} \otimes \rD$.
  \end{proof}

  Since $V_r^*/V_r^{**}$ splits if and only if $a=p+1$, this is the only value of $a$ for which $X_{r-2}^{*}/X_{r-2}^{**}$ can be different from $V_r^*/V_r^{**}$, $V_{a-2} \otimes \rD$ (which is its socle in the non-split case) or $0$ (and indeed it is if $r \equiv a \mod p$):

  \begin{lemma}
    \label{lem:Xr2AstModXr2AstAstUnequalModPaEqualP1}
    If $r \geq 2p + 1$ and $r \equiv p+1 \mod (p-1)$, then
    \[
      X_{r-2}^{*}/X_{r-2}^{**} =
      X_{r-1}^{*}/X_{r-1}^{**} =
      \begin{cases}
        V_r^*/V_{r}^{**}, & \text{\quad if $r \not \equiv 0, 1 \mod p$}\\
        V_{p-1} \ox D, & \text{\quad if $r \equiv 0 \mod p$}\\
        V_{0} \ox D, & \text{\quad if $r \equiv 1 \mod p$}.
      \end{cases}
    \]
  \end{lemma}
  \begin{proof}
    Consider
    \[
      F(X,Y) := XY^{r-1} - X^{r-1}Y \in X_{r-1} \subseteq X_{r-2}.
    \]
    By \cref{lem:VrAstCriteria}, we have $F(X,Y) \in V_r^{*}$ but $F(X,Y) \notin V_r^{**}$ as the coefficient $c_1$ of $X^{r-1}Y$ is not zero.
    Thus, $X_{r-2}^* / X_{r-2}^{**} \not= 0$.
    Since the polynomial $F(X,Y)  \in X_{r-1}$ and $V_r^*/V_r^{**}$ splits for $a=p+1$, we can determine the Jordan-Hölder series of $ X_{r-1}^{*}/X_{r-1}^{**}$ by checking if the image of the polynomial $F(X,Y)$ maps to zero or not.
    This has been studied already in Section 5 of \cite{BG}, yielding that $X_{r-2}^{*}/X_{r-2}^{**}$ contains the quotient $X_{r-1}^{*}/X_{r-1}^{**}$.
    In fact, by \cite[Lemma 4.32(i)]{GhateVangala} one gets equality so we can use the results of \cite{BG} to get the structure of the quotient.
  \end{proof}

  By \cref{lem:analog47}, for $a = 3$ and $p \nmid r-2$, we have $X_{r''}^* \neq X_{r''}^{**}$, so not necessarily $\phi (X_{r''}^* \otimes V_2 ) \subseteq X_{r-2}^{**}$.
  (We observe in particular that $r \equiv 3 \mod (p-1)$ and $r \not\equiv 2 \mod p$ imply $\Sigma(r'') \geq p$ (otherwise $\Sigma(r'') = 1$, that is, $r'' = p^n$ for some $n$), thus $X_{r''}^* \neq 0$.)
  Indeed, there is no inclusion:

  \begin{lemma}
    \label{lem:aEq3VrAst}
    If $r \geq 2p + 1$ and  $r \equiv 3 \mod (p-1)$, then
    \[
      X_{r-2}^*/X_{r-2}^{**}
      =
      \begin{cases}
        V_r^*/V_{r}^{**}, & \text{\quad if $r \not \equiv 2 \mod p$}\\
        V_1 \ox D,        & \text{\quad if $r \equiv 2 \mod p$}.
      \end{cases}
    \]
  \end{lemma}
  \begin{proof}
    \hfill
    \begin{itemize}
      \item Let $r \not \equiv 2 \mod p$.
        Consider
        \[
          F(X,Y)
          :=
          XY^{r-1} - X^{r-2}Y^2  .
        \]
        We see that the $0 \ne F(X,Y) \in X_{r-2}^*/ X_{r-2}^{**}$.
        By the same calculation as in \cite[Theorem 8.6]{BG}, we see that if $r \not \equiv 2 \mod p$, then $F(X,Y)$ generates $V_r^*/V_r^{**}$.
        Hence $ X_{r-2}^*/X_{r-2}^{**}
        =
        	V_r^*/V_{r}^{**}$.

      \item For the case $r \equiv 2 \mod p$, we have by \cref{prop:JHXrppxV2} the short exact sequence:
        \begin{align}
          0 \to (V_{2p-1} \otimes \rD) \op (V_{p-4} \ox \rD^3) \to X_{r''} \ox V_2 \to (V_1 \otimes \rD) \op V_3 \to 0.
        \end{align}
        where $V_{2p-1}$ has $V_{p-2} \ox \rD, V_{p-2} \ox \rD$ and $V_1$ as factors.
        Let $\phi \colon X_{r''} \otimes V_2 \to X_{r-2}$ be the natural mapping.
        Because $p | r''$, by \cref{lem:analog47}, we have $X_{r''}^{*} = X_{r''}^{**}$, so  $\phi (X_{r''}^* \otimes V_2) \subseteq X_{r-2}^{**}$.
        We obtain
        \begin{equation}
          \frac{X_{r''} \otimes V_2}{X_{r''}^* \otimes V_2}
          \twoheadrightarrow
          \frac{X_{r-2}}{X_{r-2}^{**}}
        \end{equation}
        leaving $V_{3}$ and $V_1 \otimes \rD$ as only possible Jordan-Hölder factors of $X_{r-2}/X_{r-2}^*$.
        We have $0 \ne F(X,Y) = X^2Y^{r-2} - X^{r-1}Y \in X_{r-2}^*/X_{r-2}^{**}$.
        Since $V_{3}$ does not appear in \cref{lem:VrAstQuotients}.\ref{en:VrAstModVrAstAst}, we can conclude $X_{r-2}^{*}/X_{r-2}^{**} = V_1 \ox D$.
        \qedhere
    \end{itemize}
  \end{proof}

\subsection{$X_{r-2}^{**}/X_{r-2}^{***}$}

  \begin{lemma}
    \label{lem:Xr2AstAstModXr2AstAstAstUnequalModP}
    Let $r \geq 3p + 2$ and $r \equiv a \mod (p-1)$ for $a = 5, \ldots, p$.
    If $r \not\equiv a, a-1 \mod p$, then
    $
      X_{r-2}^{**}/X_{r-2}^{***} \cong V_{a-4} \otimes \rD^2.
    $
  \end{lemma}
  \begin{proof}
    For $A$, $B$ and $C$ constants in $\F_p$, let $F(X,Y)$ in $X_{r-2}$ be given by:

    \begin{align}
      F(X,Y) =      & A \cdot \left[ (r-2)X^2Y^{r-2} + \sum_{k \in \mathbb{F}_p} k^{p-2} (kX +Y)^{r-2}XY \right] \\
      +    & B \cdot \left[\frac{(r-1)(r-2)}{2} X^2Y^{r-2} + \sum_{k \in \mathbb{F}_p} k^{p+3-a} (X + kY)^{r-1} Y \right]\\
      +    & C \cdot X^2 Y^{r-2}\\
      \equiv & A \cdot \left[ - \sum_{\substack{0 < j < r-3 \\ j \equiv a-3 \mod (p-1)} } \binom{r-2}{j} X^{r-j-1}Y^{j+1} \right] \\
      +      & B \cdot \left[ - \sum_{\substack{0 < j < r-3 \\ j \equiv a-3 \mod (p-1)} } \binom{r-1}{j} X^{r-j-1}Y^{j+1} \right]\\
      +      & C \cdot X^2 Y^{r-2} \mod p.
    \end{align}
    By \cref{lem:binomialsumaeq2} for $i = 1$, we obtain the following system of linear equations for $\sum_j c_j$ and $\sum_j j c_j$ to simultaneously vanish:
    \[
      \sum_j c_j = C + \alpha  A  + \frac{ \alpha \beta}{2}  B  = 0
    \]
    and
    \[
      \sum_j j c_j = (r-2)C +  \alpha (r-1) A   +  \frac{\alpha ( (\beta-2)r +2 ) B}{2}  = 0
    \]
    where $\alpha = r-a$ and $\beta = a+r-3$.
    For $F$ not to be in $V_r^{***}$, we need $C \neq 0$.

    The determinant given by the rightmost two columns is
    \[
      \frac{  \alpha^2 ((\beta-2) r+2)}{2} - \frac{\alpha^2 \beta (r-1)}{2} = \frac{\alpha^2 (\beta-2r+2)}{2}
    \]
    and thus is nonzero if and only if $\alpha = r-a \not\equiv 0 \mod p$ and $2r - 2 - \beta =r - a +1 \not \equiv 0 \mod p$, that is, $r \not\equiv a-1 \mod p$.
    Thus, if $r\not \equiv a, a-1 \mod p$, then we can find $\alpha$ and $\beta$ such that $F$ is in $X_{r-2}^{**}$, but not in $X_{r-2}^{***}$, due to the nonzero coefficient of $X^2 Y^{r-2}$.

    By \cref{lem:analog47}, we have $\phi (X_{r''}^* \otimes V_2 ) \subseteq X_{r-2}^{***}$.
    Therefore the searched-for Jordan-Hölder factors are in (the non-singular part of) the right-hand side of the short exact sequence of \cref{prop:JHXrppxV2}.
    Using \cref{lem:VrAstQuotients}.(iii), the only possible Jordan-Hölder factor is $X_{r-2}^{**}/X_{r-2}^{***} = V_{a-4} \otimes \rD^2$.
  \end{proof}

  We recall that the case $r \equiv a \mod p$ was examined in \cref{lem:Xr2r4p-1AstModAstAstEqualModP}.
  It remains to examine the case $r \equiv a-1 \mod p$.
  We do not show here that $X_{r-2}^{**}/X_{r-2}^{***} \cong 0$, equivalently, that both factors from $V_r^{**}/V_r^{***}$ are in the Jordan-Hölder series of $Q$. However, in \cref{sec:ElimJH} we show that either both factors are in the kernel of $\ind_{KZ}^G Q \twoheadrightarrow \bar\Theta_{k,a_p}$ or only one of them appears as the final factor.
  (In fact, the recent preprint \cite[Lemma 4.15]{GhateVangala} shows $X_{r-2}^{**}/X_{r-2}^{***} \cong 0$.)

  We will now compute $X_{r-2}^{**}/X_{r-2}^{***}$ for the remaining cases $p+1, p+2$ and $p+3$:

  \begin{lemma}
    \label{lem:Xr2AstAstModXr2AstAstAstUnequalModPaEqualp1}
    If $r \geq 3p + 2$ and $r \equiv p + 1 \mod (p-1)$ and $r \not\equiv 0, 1 \mod p$, then
    \[
      X_{r-2}^{**}/X_{r-2}^{***} \neq 0.
    \]
  \end{lemma}
  \begin{proof}
    Consider
    $
    F(X,Y) \in X_{r}
    $
    given by
    \[
      F(X,Y)
      =
      \sum_{k \in \mathbb{F}_p}  (kX + Y)^{r}
      \equiv
      - \sum_{\substack{0 < j < r \\ j \equiv 2 \mod (p-1)}} \binom {r}{j}X^{r-j}Y^{j}
      \mod p.
    \]
    Working $\mod p$:
    \[
      - F(X,Y) \equiv \sum_{\substack{0 < j < r \\ j \equiv 2 \mod (p-1)}} \binom {r}{j}X^{j}Y^{r-j}.
    \]
    Let $c_j$ denote the coefficients of $-F$.
    By \cref{lem:binomialsumaeq2} we see that $\sum_j c_j = \sum_{j} \binom{r}{j} \equiv 0 \mod p$.
    We compute
    \begin{align}
        \sum_j j c_j
        & =
        \sum_{0 < j \equiv 2 < r} j \binom{r}{j}\\
        & =
        r \sum_{0 < j' \equiv 1 < r'} \binom{r'}{j'} \equiv 0 \mod p
    \end{align}
    by \cref{lem:binomialsumaeq2}.
    Therefore, by \cref{lem:VrAstCriteria}, we have $F(X,Y) \in V_r^{**}$, but $F(X,Y) \notin V_r^{***}$ because the coefficient $c_{2}$ of $X^{r-2}Y^2$ is $\binom{r}{2} \not\equiv  0 \mod p$ by hypothesis.
    Thus, $X_{r-2}^{**}/X_{r-2}^{***} \ne 0$.
    (In fact, we have shown that even $X_{r}^{**}/X_{r}^{***} \ne 0$.)
     \end{proof}

\begin{lemma}
    \label{lem:aEq3VrAstAstnot012}
    If $r \geq 3p+2$ and $r \equiv p+2 \mod (p-1)$ and $r \not\equiv 0, 1,2 \mod p$, then
    \[
      X_{r-2}^{**}/X_{r-2}^{***}  = V_r^{**}/V_r^{***}.
    \]
  \end{lemma}
  \begin{proof}
    Consider the polynomial
    \begin{align}
      F(X,Y) & := A_1 X^{r-2}Y^{2}\\
             & -  A_2 \sum_{k \in \mathbb{F}_p} k^{p-2}(kX+Y)^{r} \\
             & - \sum_{k \in \mathbb{F}_p} k^{p-1}(kX+Y)^{r-1} X\\
             & - A_3 \sum_{k \in \mathbb{F}_p} k (kX+Y)^{r-2} X^2
    \end{align}
    in $X_{r-2}$ where $A_1$, $A_2$ and $A_3$ are constants that can be chosen such that
    \begin{align}
      A_1 + 3A_2 & \equiv -1 \mod p,\\
      r A_2 & \equiv -1 \mod p, \text{ and }\\
      2A_1 + 2r A_2 - (r-2)A_3 & \equiv  1 - r \mod p.
    \end{align}
    For this, we observe that we can put $A_2 = -r^{-1}$ as $r \nequiv 0 \mod p$ and that the determinant of the linear equation system in $A_1$ and $A_3$ is $r-2 \nequiv 0 \mod p$.

    We have
    \begin{align}
      F(X,Y) & \equiv A_1 X^{r-2}Y^{2}\\
             & +  A_2  \sum_{\substack{0 < j \leq r-1,\\ j \equiv 2 \mod (p-1)}} \binom{r}{j} X^{r-j}Y^j\\
             & + \sum_{\substack{0 < j \leq r-1,\\ j \equiv 2 \mod (p-1)}} \binom{r-1}{j} X^{r-j}Y^j\\
             & + A_3 \sum_{\substack{0 < j \leq r-1,\\ j \equiv 2 \mod (p-1)}} \binom{r-2}{j} X^{r-j}Y^j \mod p.
    \end{align}
    Denote the coefficient of $X^{r-j}Y^j$ by $c_j$.
    First, we note that $c_0, c_1, c_r$ do not occur.
    The coefficient $c_{r-1}$ vanishes as $A_2 r + 1 \equiv 0 \mod p$.
    By \cref{lem:binomialsumaeq2},
    \[
      \sum_j c_j = 3A_2 + A_1 + 1 \equiv 0 \mod p
    \]
    and
    \[
      \sum j c_j \equiv 2 A_1 + 2A_2 r + (r-1) + (r-2) A_3 \equiv 0 \mod p.
    \]

    Since all indices of nonzero coefficients in $F$ are congruent $\mod (p-1)$, we can apply \cref{lem:VrAstCriteria} and obtain $F(X,Y) \in V_r^{**}$.

    Using \cref{lem:VrAstCriteria} and \cref{lem:binomialsumaeq2},
    \[
      F(X,Y) \equiv
      \binom{r-1}{2}\theta^2 X^{r-3p-1}Y^{p-1} + \left(A_1 + A_2 \binom{r}{2} + \binom{r-1}{2} + A_3 \binom{r-2}{2}\right) \theta^2 X^{r-2p-2}
      \mod V_r^{***},
    \]
    which by \cref{lem:Analog85} maps to a non-zero element in $V_1 \ox \rD$ as $r \not \equiv 0,1,2 \mod p$.
    Hence $X_{r-2}^{**}/X_{r-2}^{***} = V_r^{**}/V_r^{***}$ as the short exact sequence of \cref{lem:VrAstQuotients}.(iii) does not split.
  \end{proof}

  \begin{lemma}
    \label{lem:aEq3VrAstAst}
    If $r \geq 3p+2$ and $r \equiv p+2 \mod (p-1)$ and $r \equiv 0 \mod p$, then
    \[
      X_{r-2}^{**}/X_{r-2}^{***}   \not= 0.
    \]
  \end{lemma}
  \begin{proof}
    Consider
    \[
      F(X,Y) =
      \sum_{k \in \mathbb{F}_p} X(kX+Y)^{r-1} \equiv
      - \sum_{\substack{0 < j < r-1,\\ j \equiv 2 \mod (p-1)}} \binom{r-1}{j} X^{r-j}Y^j
      \mod p.
    \]
    Denote the coefficients of $-F$ by $c_j$.
    First, we note that $c_0$, $c_1$, $c_{r-1}$, $c_r$ do not occur.
    By \cref{lem:binomialsumaeq2} for $i = 0$ we see that $\sum_j c_j = \sum_{j} \binom{r-1}{j} \equiv 0 \mod p$ and, again by \cref{lem:binomialsumaeq2} for $i = 0$,
    \begin{align}
      \sum_j j c_j
      & = \sum_{0 < j \equiv a-1 < r-1} j \binom{r-1}{j}\\
      & = (r-1) \sum_{0 < j' \equiv r-2 < r-2} \binom{r-2}{j'} \equiv 0 \mod p.
    \end{align}
    Therefore, by \cref{lem:VrAstCriteria}, we have $F(X,Y) \in V_r^{**}$, but $F(X,Y) \notin V_r^{***}$ because the coefficient $c_2$ of $X^{r-2}Y^2$ is $\binom{r-1}{2} \not \equiv 0 \mod p$ as $r \equiv 0 \mod p$ by assumption.
  \end{proof}

  \begin{lemma}
    \label{lem:Xr2r3AstAstModAstAstAstEqualModP}
    If $r \geq 3p + 2$ and $r \equiv p+2 \mod (p-1)$ and $r \equiv 2 \mod p$, then
    \[
      X_{r-2}^{**}/X_{r-2}^{***} = 0.
    \]
  \end{lemma}
  \begin{proof}
    By \cref{prop:JHXrppxV2} we have the short exact sequence:
    \begin{align}
      0 \to X_{r''}^* \otimes V_2 = V_{2p-1} \otimes \rD \op (V_{p-4} \ox \rD^3) \to X_{r''} \ox V_2 \to (V_1 \otimes \rD) \op V_3 \to 0
    \end{align}
    where the left-hand side either vanishes or equals $V_{2p-1} \otimes \rD$.
    Let $\phi \colon X_{r''}^* \otimes V_2 \to X_{r-2}$.
    For $r \equiv 2 \mod p$, that is $p | r''$, by \cref{lem:analog47} we have $X_{r''}^{*} = X_{r''}^{**} = X_{r''}^{***}$, so  $\phi (X_{r''}^* \otimes V_2) \subseteq X_{r-2}^{***}$.
    Therefore
    \[
      \frac{X_{r''} \otimes V_2}{X_{r''}^* \otimes V_2}
      \twoheadrightarrow
      \frac{X_{r-2}}{X_{r-2}^{***}}.
    \]
    Since the short exact sequence \cref{lem:VrAstQuotients}.(iii) does not split for $a = 3$, we have $X^{**}_{r-2} / X^{***}_{r-2} \neq 0$ if and only if $V_{p-2} \otimes \rD^2 \hookrightarrow  X^{**}_{r-2} / X^{***}_{r-2}$.
    As $V_{p-2} \otimes \rD^2$ does not appear in the right-hand side of the corresponding short exact sequence in \cref{prop:JHXrppxV2}, we conclude $X_{r-2}^{**}/X_{r-2}^{***} = 0$.
  \end{proof}

  \begin{lemma}
    \label{lem:Xr2AstAstModXr2AstAstAstUnequalModPaEqual4}
    If $r \geq 3p + 2$ and $r \equiv p+3 \mod (p-1)$ and $r \not\equiv 2,3 \mod p$, then
    \[
      V_{p-1} \otimes \rD^2 \hookrightarrow
      X_{r-2}^{**}/X_{r-2}^{***}.
    \]
  \end{lemma}
  \begin{proof}
    Consider
    \[
      F(X,Y) = \sum_{k \in \mathbb{F}_p}  (kX + Y)^{r-2}X^2 \in X_{r-2}.
    \]
    Working $\mod p$:
    \[
      - F(X,Y) \equiv \sum_{\substack{0 < j < r -2\\ j \equiv 2 \mod (p-1)}} \binom {r-2}{j}X^{r-j}Y^{j}.
    \]
    Let $c_j$ denote the coefficients of $F$.
    By \cref{lem:binomialsumaeq2} we see that $\sum c_j \equiv \sum j c_j \equiv 0 \mod p$.
    Therefore, by \cref{lem:VrAstCriteria}, we have $F(X,Y) \in V_r^{**}$, but $F(X,Y) \notin V_r^{***}$ because
    \[
      F(X,Y) \equiv
      \binom{r-2}{2} \theta^2 X^{r-2p-2} \mod V_r^{***}.
    \]
    The coefficient of $X^{r-2}Y^{2}$ is $\binom{r-2}{2} \not\equiv 0 \mod p$ by hypothesis.
    Thus, $X_{r-2}^{**}/X_{r-2}^{***} \ne 0$ (containing $V_{p-1} \otimes \rD^2$ by \cref{lem:Analog85}).
    \end{proof}

  \begin{lemma}
    \label{lem:Xr2r4AstAstModAstAstAstEqualMod2}
    If $r \geq 3p + 2$ and $r \equiv p+3 \mod (p-1)$ and $r \equiv 2 \mod p$, then
    \[
      X_{r-2}^{**}/X_{r-2}^{***} = V_0 \otimes \rD^2 .
    \]
  \end{lemma}
  \begin{proof}
    Let
    \[
      F(X,Y) := \sum_{k \in \F_p} k^{p-3} (kX+Y)^r +  3X^2 Y^{r-2} + 3X^{r-2} Y^2 \in X_{r-2}.
    \]
    Working $\mod p$:
    \begin{align}
      F(X,Y) & \equiv - \sum_{\substack{0 < j \leq r-2\\ j \equiv a-2 \equiv 2 \mod (p-1)}} \binom{r}{j} X^{r-j}Y^{j} + 3X^2 Y^{r-2} + 3X^{r-2} Y^2\\
             & \equiv - \sum_{\substack{0 < j < r-2\\ j \equiv 2 \mod (p-1)}} \binom{r}{j} X^{r-j}Y^{j} - \binom{r}{r-2}X^2 Y^{r-2}+ 3 X^2Y^{r-2} + 3 X^{r-2}Y^2.
    \end{align}
    As $r \equiv 2 \mod p$, we see that $ \binom{r}{r-2} \equiv 1 \mod p$. Thus,

    \[
      F(X,Y) \equiv - \sum_{\substack{0 < j < r-2\\ j \equiv 2 \mod (p-1)}} \binom{r}{j} X^{r-j}Y^{j} + 2 X^2Y^{r-2} + 3 X^{r-2}Y^2.
    \]
    Let $c_j$ denote the coefficients of $F$.

    By \cref{lem:binomialsumaeq2} for $a=4$ and $i = 2$, using $r \equiv 2 \mod p$,
    \[
      \sum c_j \equiv - \frac{(4-2)(4+2-1)}{2} + 2 + 3 \equiv 0 \mod p
    \]
    and
    \begin{align}
      \sum j c_j & \equiv - \sum_{\substack{0 < j < r-2\\ j \equiv 2 \mod (p-1)}} j \binom{r}{j}+2(r-2) + 3 \cdot 2\\
                 & \equiv - r \sum_{\substack{0 < j' < r'-2\\ j' \equiv 1 \mod (p-1)}}  \binom{r'}{j'}+2(r-2) + 3 \cdot 2\\
                 & \equiv - \frac{r((a-1)-(r'-1))(a-1+r-1-1)}{2} +0+6\\
                 & \equiv - \frac{2(3-1)(3+1-1)}{2} + 6 \equiv -6 + 6 \equiv 0 \mod p.
    \end{align}

    Therefore, by \cref{lem:VrAstCriteria}, we have $F(X,Y) \in V_r^{**}$, but $F(X,Y) \notin V_r^{***}$ because the coefficient $c_{r-2}$ of $X^2Y^{r-2}$ is $2 \not\equiv  0 \mod p$.
    Thus, $X_{r-2}^{**}/X_{r-2}^{***} \ne 0$.

    By \cref{lem:analog47} for $r \equiv 2 \mod p$, we have $\phi (X_{r''}^* \otimes V_2 ) \subseteq X_{r-2}^{***}$.
    Therefore the searched-for Jordan-Hölder factors are in (the non-singular part of) the right-hand side of the short exact sequence of \cref{prop:JHXrppxV2}.
    Using \cref{lem:VrAstQuotients}.(iii), the only possible Jordan-Hölder factor is $X_{r-2}^{**}/X_{r-2}^{***} = V_0 \otimes \rD^2$.
  \end{proof}

  We can refine \cref{lem:Xr2r4AstAstModAstAstAstEqualMod2} for the case $r \equiv 4 \mod p$, as follows:

  \begin{lemma}
    \label{lem:Xr2r4AstAstModAstAstAstEqualModP}
    If $r \geq 3p + 2$ and $r \equiv p+3 \mod (p-1)$ and $r \equiv 4 \mod p$, then $X_{r-2}^{**}/X_{r-2}^{***}$ contains the Jordan-Hölder factor  $V_{p-1} \ox \rD^2$.
  \end{lemma}
  \begin{proof}
    The proof is similar to the proof of \cite[Lemma 5.5]{BG}:
    We have the short exact sequence
    \begin{align}
      0 \to X_{r''}^* \otimes V_2 \to X_{r''} \ox V_2 \to X_{r''} / X_{r''}^*  \otimes V_2.
    \end{align}

    Let $F(X,Y) :=  X^{r-2} Y^2 - X^2 Y^{r-2}$ in $X_{r-2}^{**}$.
    We claim that $F(X,Y) \mapsto 0$ in the Jordan-Holder factor $V_0 \ox D$ of $V_r^{**}/V_r^{***}$ under the projection following \cite[Lemma 5.1]{BG}:
    \[
      \frac{X_{r-2}^{**}}{X_{r-2}^{***}} \hookrightarrow
      \frac{V_r^{**}}{V_r^{***}} \cong
      \frac{V_{r-2p-2} }{V_{r-2p-2}^*}\ox \rD^2 \rightarrow
      \frac{V_{2p-2}}{V_{2p-2}^*}\ox \rD^2 \twoheadrightarrow
      V_0 \ox \rD^2.
    \]
    Proof of our Claim:
   As in \cite[Lemma 5.1]{BG}, we have $X^{r-2p-2-i(p+1)}Y^{i(p-1)} \mapsto X^0 Y^0 =: e$ for $i = 1, \ldots,\frac{r-2p-2}{p-1}-1$, while the initial term $X^{r-2p-2}$ and the last term $Y^{r-2p-2}$ of the sum both vanish.
   Under this projection, the coefficient of the basis vector $e$ of $V_0 \otimes \rD^2$ is given by
   \begin{align}
     \sum_{i = 1, ..., \frac{r-2p-2}{p-1} - 1} i+1
     = &
     2 + \cdots + \frac{r-2p-2}{p-1}\\
     = &
     \left(\frac{r-2p-2}{p-1}\right)\left(\frac{r-2p-2}{p-1} + 1\right)/2 - 1\\
     \equiv & (-2)(-1) / 2 - 1 = 0 \mod p
   \end{align}
   because $r \equiv 4 \mod p$.
   That is, as claimed, $f \mapsto 0$ in $V_0 \ox \rD^2$.
   Thus $X_{r-2}^{**}/X_{r-2}^{***}$ contains $V_{p-1} \ox \rD^2$ as a Jordan-Hölder factor.
  \end{proof}


\section{The Jordan-Hölder series of $Q$}
\label{sec:Q}

To study the Jordan-Hölder series of $Q := V_r / (V_r^{***} + X_{r-2})$, we consider the following commutative diagram with exact rows and columns:
\[
\begin{CD}
  @.0@.0@.0@.\\
  @.@VVV@VVV@VVV @.\\
  0@ > >>\dfrac{X_{r-2}^{*}}{X_{r-2}^{***}}@ > >> \dfrac{X_{r-2}}{X_{r-2}^{***}}@ > >>\dfrac{X_{r-2}}{X_{r-2}^{*}}@ > >>0\\
  @.@VVV @VVV@VVV@.\\
  0@ > >>\dfrac{V_r^{*}}{V_r^{***}}@ > >> \dfrac{V_r}{V_r^{***}}@ > >>\dfrac{V_r}{V_r^{*}}@ > >>0\\
  @.@VVV@VVV @VVV @.\\
  0@ > >>\dfrac{V^*_r}{X_{r-2}^* + V_r^{***}} @ > >> Q@ > >> \dfrac{V_r}{X_{r-2} + V_r^*}@ > >>0\\
  @.@VVV@VVV @VVV @.\\
  @. 0@.0@.0@.\\
\end{CD}
\label{eq:CDQ}
\]
By \cref{prop:Xrm012ModAst} and \cref{lem:VrAstQuotients} the two Jordan-Hölder factors of $V_r/V_r^{*}$ and (one or two) Jordan-Hölder factors of $X_{r-2} / X_{r-2}^{*}$ are known, so we can determine the Jordan-Hölder factor on the right-hand side of the bottom line:
\[
  \label{eq:VrXr2VrAst}
  U
  :=
  \dfrac{V_r}{X_{r-2} + V_r^*}
  =
  \begin{cases}
    0,	& \text{ \quad for } a = 1, 2\\
    V_{p-a-1} \otimes \rD^a,	& \text{ \quad for } a = 3, \ldots, p-1
  \end{cases}
\]
where $a$ in $\{ 1, \ldots, p-1 \}$ such that $r \equiv a \mod (p-1)$.
Therefore, we are left with determining the Jordan-Hölder factor of the left-hand side of the bottom line,
\[
  W
  :=
  \dfrac{V_r^*}{X_{r-2}^* + V_r^{***}}.
\]
By \cref{lem:VrAstQuotients} the four Jordan-Hölder factors of $V_r^*/V_r^{***}$ are known, so by looking at the short exact sequence of the left column of Diagram \eqref{eq:CDQ}, we are reduced to determining the Jordan-Hölder factors of
\[
  X_{r-2}^*/X_{r-2}^{***},
\]
that is, of
\[
  X_{r-2}^*/X_{r-2}^{**}
  \quad \text{ and } \quad
  X_{r-2}^{**}/X_{r-2}^{***},
\]
where we computed in \cref{sec:singular-quotients-x-r-2} whether the quotient $X_{r-2}^*/X_{r-2}^{**}$ respectively $X_{r-2}^{**}/X_{r-2}^{***}$ is nonzero or not.

By \cref{sec:Xr-2}, we have the exact sequence:
\[
  \label{eq:XrppAstxV2Xr2xV2}
 \phi( X_{r''}^* \otimes V_2)
  \rightarrow X_{r-2}
  \rightarrow X_{r-2} / \phi( X_{r''}^* \otimes V_2) \to 0.
\]
Let $a$ in $\{3, \ldots, p + 1\}$ such that $r \equiv a \mod (p-1)$.
By \cref{lem:analog47} applied to $r''$,
\begin{itemize}
  \item for $a = 3$ and $p \mid  r-2$,
  \item for $a = 4$ and $r-2 \equiv 0,1 \mod p$, and
  \item for $a = 5, \ldots, p + 1$,
\end{itemize}
we have $X_{r''}^* = X_{r''}^{**} = X_{r''}^{***}$, so $\phi (X_{r''}^* \otimes V_2 ) \subseteq X_{r-2}^{***}$.
Thus, the Jordan-Hölder series of $X_{r-2}^*/X_{r-2}^{***}$ is included in the largest non-singular submodule of the right-hand side $X_{r-2} / \phi( X_{r''}^* \otimes V_2)$ of \eqref{eq:XrppAstxV2Xr2xV2}.

By \cref{prop:Xrm012ModAst}, the Jordan-Hölder factor $V_a$ (and $V_{p-a-1} \otimes \rD^a$ for $a = 1, 2$) in \eqref{eq:XrppAstxV2Xr2xV2} vanishes when we reduce $X_{r-2}$ in \eqref{eq:XrppAstxV2Xr2xV2} to its largest singular subspace $X_{r-2}^*$.
Thus, by \cref{prop:JHXrppxV2} there is a single Jordan-Hölder factor for $a = 3$, two Jordan-Hölder factors in $X_{r-2}^* / \phi (X_{r''}^* \otimes V_2 )$ for $a = 4, \ldots, p$, but three for $a = p+1$.
In particular,
\begin{itemize}
  \item if $a = 3$ and $r \equiv 2 \mod p$, we can prove $X_{r-2}^*/X_{r-2}^{**} \neq 0$ (and $X_{r-2}^{**}/X_{r-2}^{***} = 0$ in \cref{lem:aEq3VrAst} and \cref{lem:Xr2r3AstAstModAstAstAstEqualModP} respectively),
  \item or $a = 4$ and $r \equiv 2 \mod p$ or $a = 5, \ldots, p$ (except the case $r \equiv a, a-1 \mod p$) we can prove both $X_{r-2}^*/X_{r-2}^{**}$ and $X_{r-2}^{**}/X_{r-2}^{***}$ to be non-zero,
\end{itemize}
then we know all Jordan-Hölder factors of $X_{r-2}^* / X_{r-2}^{***}$.
The remaining cases when, the conditions of \cref{lem:analog47} are not satisfied, that is,
\begin{itemize}
  \item $a = 3$ and $r \not\equiv 2 \mod p$, or
  \item $a = 4$ and $r \not\equiv 2,3 \mod p$, or
  \item $a = p+1$, or
  \item  there are more than two Jordan-Hölder factors in $X_{r-2}^* / \phi (X_{r''}^* \otimes V_2 )$,
\end{itemize}
were handled separately in \cref{sec:singular-quotients-x-r-2}.

  \subsection{$a=3$}

  \begin{proposition}
    \label{prop:QrEq3}
    If $r \geq 3p + 2$ and $r \equiv 3 \mod (p-1)$, then the Jordan-Hölder series of $Q$ is:
    \[
      0 \to W \to Q \to U \to 0
    \]
    where  $U = V_{p-4} \otimes \rD^3$ and the Jordan-Hölder factors of $W$ are as follows:
    \begin{enumerate}
    \item None, if $r \not \equiv 0,1,2 \mod p$.
    \item At most the factor $V_1 \ox \rD$, if $r \equiv 0 \mod p$.
    \item If $r \equiv 1 \mod p$, then
      \begin{enumerate}
        \item None, if $X_{r-2}^{**} / X_{r-2}^{***} = V_r^{**} / V_r^{***}$,
        \item $V_{p-2} \otimes \rD^2$, if $X_{r-2}^{**} / X_{r-2}^{***} \neq 0$, or
        \item $V_r^{**} / V_r^{***}$, if $X_{r-2}^{**} / X_{r-2}^{***} = 0$.
      \end{enumerate}
    \item $V_{p-2} \ox D^2, V_{p-2} \ox D^2$ and $V_1 \ox \rD$ if $r \equiv 2 \mod p$.
  \end{enumerate}
  \end{proposition}
  \begin{proof}
    By \eqref{eq:VrXr2VrAst}, we have $U = V_{p-4} \otimes \rD^3$.
    We now use the results of the previous section.
    \begin{enumerate}
      \item By \cref{lem:aEq3VrAst} and by \cref{lem:aEq3VrAstAstnot012} none of the factors in $W$ appear as $X_{r-2}^{*}/X_{r-2}^{***} = V_r^{*}/V_r^{***}$.
      \item By \cref{lem:aEq3VrAst} we see that $X_{r-2}^{*}/X_{r-2}^{**} = V_r^{*}/V_r^{**}$ while by \cref{lem:aEq3VrAstAst} we have that $X_{r-2}^{**}/X_{r-2}^{***} \ne 0$. As $V_{p-2} \ox \rD^2$ is inside $X_{r-2}^{**}/X_{r-2}^{***}$  hence the only possible factor that appears in $W$ is at most $V_1 \ox \rD$. 
      \item If $r \equiv 1 \mod p$, then by \cref{lem:aEq3VrAst}, we know that $X_{r-2}^{*}/X_{r-2}^{**} = V_r^{*} / V_r^{**}$.
      \item If $r \equiv 2 \mod p$, then by \cref{lem:aEq3VrAst}, we know that $X_{r-2}^{*}/X_{r-2}^{**} = V_1 \ox \rD$ while by \cref{lem:Xr2r3AstAstModAstAstAstEqualModP} we know that $X_{r-2}^{**}/X_{r-2}^{***} = 0$ hence both factors of $V_r^{**} / V_r^{***}$ appear in $W$.
        \qedhere
    \end{enumerate}
  \end{proof}

  \subsection{$a=4$}

 \begin{proposition}
    \label{prop:QrEq4}
    If $r \geq 3p + 2$ and $r  \equiv 4 \mod (p-1)$, then the Jordan-Hölder series of $Q$ is:
    \[
      0 \to W \to Q \to U \to 0
    \]
    where $U = V_{p-5} \ox \rD^4$ and:
    \begin{enumerate}
      \item If $r \equiv 1 \mod p$, then $W$ has Jordan-Hölder factors $V_{p-3} \ox \rD^3$ and $V_0 \ox \rD^2$.
      \item If $r \equiv 4 \mod p$, then $W$ has Jordan-Hölder factors $V_{p-3} \ox \rD^3$ and $V_2 \ox \rD$.
      \item If $r \not \equiv 1,2,3,4 \mod p$, then $W$ has the single Jordan-Hölder factor $V_{p-3} \ox \rD^3$.
    \end{enumerate}
  \end{proposition}
  \begin{proof}
    By \eqref{eq:VrXr2VrAst}, we have $U = V_{p-5} \ox \rD^4$.
   \begin{enumerate}
      \item
        If $r \equiv 1 \mod p$, then $X_{r-2}^* / X_{r-2}^{**} = V_{2} \otimes \rD^3$ by \cref{lem:Xr2AstModXr2AstAstUnequalModP}  and $V_{p-1} \otimes \rD^2 \hookrightarrow X_{r-2}^{**} / X_{r-2}^{***}$ by \cref{lem:Xr2AstAstModXr2AstAstAstUnequalModPaEqual4}.
        Hence, $W$ has Jordan-Hölder factors $V_{p-3} \ox \rD^3$and $V_0 \ox \rD^2$
      \item
        If $r \equiv 4 \mod p$, then $X_{r-2}^* / X_{r-2}^{**} = 0$  by \cref{lem:Xr2r4p-1AstModAstAstEqualModP} and $X_{r-2}^{**} / X_{r-2}^{***} $ contains the factor $V_{p-1} \otimes \rD^2$ by \cref{lem:Xr2r4AstAstModAstAstAstEqualModP}. Also, \cite[Lemma 4.20]{GhateVangala} shows $X_{r-2}^{**} / X_{r-2}^{***} $ contains $V_0 \otimes \rD^2$.
        Hence, $W$ has Jordan-Hölder factors $V_{p-3} \ox \rD^3$ and $V_2 \ox \rD$.
      \item
        If $r \not \equiv 1, 2, 3, 4 \mod p$, then $0 \ne X_{r-2}^* / X_{r-2}^{**} = V_2 \ox D$ by \cref{lem:Xr2AstModXr2AstAstUnequalModP}.
        By \cite[Lemma 4.20]{GhateVangala}, we have that  $X_{r-2}^{**} / X_{r-2}^{***} = V_r^{**}/V_r^{***}$. Hence, there is only a single Jordan-Hölder factor $V_{p-3} \ox \rD^3$.
        \qedhere
    \end{enumerate}
  \end{proof}

\subsection{$a=p$}

\begin{proposition}
  \label{prop:QrEqp}
  If $r \geq 3p + 2$ and $r  \equiv p \mod (p-1)$ then the Jordan-Hölder series of $Q$ is:
  \[
    0 \to W \to Q \to U \to 0
  \]
  where $U = 0$ and:
  \begin{enumerate}
    \item If $r \equiv p \mod p$, then $W = V_r^*/V_r^{***}$.
    \item  If $r \not \equiv p, p-1 \mod p$, then the Jordan-Hölder factors of $W$ are  $V_{1}$ and  $V_{3} \ox \rD^{p - 2}$.
    \item If $r \equiv p-1 \mod p$, then the Jordan-Hölder factors of $W$ are $V_{1}$ and possibly $V_{p-4} \ox \rD^2$ and $ V_{3} \ox \rD^{p - 2}$.
  \end{enumerate}
\end{proposition}
\begin{proof}
  By \eqref{eq:VrXr2VrAst}, we have $U = 0$.
  \begin{enumerate}
    \item
      When $r \equiv p \mod p$, then by \cref{lem:Xr2r4p-1AstModAstAstEqualModP} we have $X_{r-2}^{*}/X_{r-2}^{***} = 0$, hence $W = V_r^* / V_r^{***}$.
    \item
      If $r \not \equiv p, p-1 \mod p$, then by \cref{lem:Xr2AstModXr2AstAstUnequalModP} and \cref{lem:Xr2AstAstModXr2AstAstAstUnequalModP} we have $X_{r-2}^*/X_{r-2}^{**} \ne 0$ and $X_{r-2}^{**}/X_{r-2}^{***} \ne 0 $.
      By \cref{lem:analog47}, we have $X_{r''}^* = X_{r''}^{***}$, thus $\phi (X_{r''}^* \otimes V_2 ) \subseteq X_{r-2}^{***}$.
      By comparing with (the non-singular part of) the right-hand side of the short exact sequence of \cref{prop:JHXrppxV2} and using \cref{lem:VrAstQuotients}, we find that $W$ contains one Jordan-Hölder factor of $V_r^{*} / V_r^{**}$ and one of $V_r^{**} / V_r^{***}$.
    \item
      If $r \equiv p-1 \mod p$, then by \cref{lem:Xr2AstModXr2AstAstUnequalModP} we have $0 \ne X_{r-2}^*/X_{r-2}^{**}$.
      By \cref{lem:analog47}, we have in particular $X_{r''}^* = X_{r''}^{**}$, thus $\phi (X_{r''}^* \otimes V_2 ) \subseteq X_{r-2}^{**}$.
      By comparing with (the non-singular part of) the right-hand side of the short exact sequence of \cref{prop:JHXrppxV2}, therefore $X_{r-2}^*/X_{r-2}^{**} = V_{p-2} \ox \rD$.
      Therefore $W$ contains only one Jordan-Hölder factor of $V_r^{*} / V_r^{**}$ and possibly both of $V_r^{**} / V_r^{***}$.
      \qedhere
  \end{enumerate}
\end{proof}

\subsection{$a=p+1$}

\begin{proposition}
  \label{prop:QrEqp+1}
  If $r \geq 3p + 2$ and $r  \equiv p+1 \mod (p-1)$ then the Jordan-Hölder series of $Q$ is:
  \[
    0 \to W \to Q \to U \to 0
  \]
  where $U = 0$ and:
  \begin{enumerate}
    \item  If $r \not \equiv 0, 1 \mod p$ then $W$ has only one Jordan-Hölder factor $V_{2}$.
  \end{enumerate}
\end{proposition}
\begin{proof}
  By \eqref{eq:VrXr2VrAst}, we have $U = 0$.
  \begin{enumerate}
    \item If $r \not \equiv 0,1 \mod p$, then by \cref{lem:Xr2AstModXr2AstAstUnequalModPaEqualP1}
      we have  $X_{r-2}^{*}/X_{r-2}^{**} = V_r^*/V_r^{**}$ while by \cref{lem:Xr2AstAstModXr2AstAstAstUnequalModPaEqualp1}, we know that $ X_{r-2}^{**}/X_{r-2}^{***} \ne 0$.
      As $ V_{p-3} \ox D^2 \subseteq X_{r-2}^{**}/X_{r-2}^{***}$ we see that $Q$ contains at most $V_2$ as a Jordan-Hölder factor.
      Section 4.2.1 of \cite{GhateVangala} shows that $X_{r-2}^{**}/X_{r-2}^{***} \ne V_r^{**}/V_r^{***}$.
      The Jordan-Hölder series of $Q$ follows.
  \end{enumerate}
\end{proof}

 \subsection{$r$ has the same representative mod $(p-1)$ and $p$}

  \begin{proposition}
    \label{prop:QrEqA}
    Let $a$ in $\left\{ 5, \ldots, p-1 \right\}$ such that $r  \equiv a \mod (p-1)$.
    If $r \equiv a \mod p$, then the Jordan-Hölder series of $Q$ is:
    \[
      0 \to W \to Q \to U \to 0
    \]
    where $W = V_r^*/V_r^{***}$ and $U = V_{p -  a - 1} \ox \rD^{a}$.
  \end{proposition}
  \begin{proof}
    By \eqref{eq:VrXr2VrAst}, we have $U =V_{p - a - 1} \ox \rD^{ a}$.
    By \cref{lem:Xr2r4p-1AstModAstAstEqualModP}, we know $X_{r-2}^*/X_{r-2}^{***} = 0$.
    Hence, $W = V_r^*/V_r^{***}$.
  \end{proof}

  \subsection{$r$ does \emph{not} have the same representative mod $(p-1)$ and $p$}

  \begin{proposition}
    \label{prop:Qrnota}
    Let $a$ in $\left\{ 5, \ldots, p - 1 \right\}$ be such that $r  \equiv a \mod (p-1)$.
    If $r \geq 3p + 2$ and $r \not \equiv a, a-1 \mod p$, then the Jordan-Hölder series of $Q$ is:
    \[
      0 \to W \to Q \to U \to 0
    \]
    where $W$ has the two Jordan-Hölder factors $V_{p - a + 1} \ox \rD^{a - 1}$ and  $V_{p - a + 3} \ox \rD^{a - 2}$ and $U = V_{p - a  - 1} \ox \rD^{a}$.
  \end{proposition}
  \begin{proof}
    By \eqref{eq:VrXr2VrAst}, we have $U = V_{p - a  - 1} \ox \rD^{a}$.

    To compute the left-hand side $W$, we compare $X_{r-2}^*/X_{r-2}^{**}$ and $X_{r-2}^{**}/X_{r-2}^{***}$ with the Jordan-Hölder series of $V_r^*/V_r^{**}$ and $V_r^{**}/V_r^{***}$ in \cref{lem:VrAstQuotients}:
    By \cref{lem:Xr2AstModXr2AstAstUnequalModP} and \cref{lem:Xr2AstAstModXr2AstAstAstUnequalModP} we have $X_{r-2}^*/X_{r-2}^{**} \ne 0$ and $X_{r-2}^{**}/X_{r-2}^{***} \ne 0 $.
    By \cref{lem:analog47}, we have $X_{r''}^* = X_{r''}^{***}$, thus $\phi (X_{r''}^* \otimes V_2 ) \subseteq X_{r-2}^{***} \subseteq X_{r-2}^{**}$.
    By comparing with (the non-singular part of) the right-hand side of \cref{prop:JHXrppxV2}, we find that $W$ contains exactly one Jordan-Hölder factor each of $V_r^{*} / V_r^{**}$ and of $V_r^{**} / V_r^{***}$.
  \end{proof}

  \begin{proposition}
    \label{prop:QrEqA-1}
    Let $a$ in $\left\{ 5, \ldots, p-1 \right\}$ such that $r  \equiv a \mod (p-1)$.
    If $r \geq 3p + 2$ and $r \equiv a-1 \mod p$, then the Jordan-Hölder series of $Q$ is given by:
    \[
      0 \rightarrow W
      \rightarrow Q
      \rightarrow U \rightarrow 0
    \]
    where the Jordan-Hölder factors of $W$ are $V_{p-a+1} \ox \rD^{a - 1}$ and possibly $V_{a-4} \ox \rD^2$ and $ V_{p - a + 3} \ox \rD^{a - 2}$, and $U = V_{p - a - 1} \ox \rD^{a}$.
  \end{proposition}
  \begin{proof}
    By \eqref{eq:VrXr2VrAst}, we have $U = V_{p - a - 1} \ox \rD^{a}$.

    By \cref{lem:Xr2AstModXr2AstAstUnequalModP} we have $X_{r-2}^*/X_{r-2}^{**} \ne 0$.
    By \cref{lem:analog47}, we have $X_{r''}^* = X_{r''}^{***}$, thus $\phi (X_{r''}^* \otimes V_2 ) \subseteq X_{r-2}^{***} \subseteq X_{r-2}^{**}$.
    By comparing with (the non-singular part of) the right-hand side of \cref{prop:JHXrppxV2}, we find that $X_{r-2}^*/X_{r-2}^{**} = V_{a-2} \otimes \rD$.
    Therefore $W$ contains only one Jordan-Hölder factor of $V_r^{*} / V_r^{**}$ and possibly both of $V_r^{**} / V_r^{***}$.
  \end{proof}


\section{Eliminating Jordan-Hölder factors}
\label{sec:ElimJH}

Throughout this section we assume that $p \geq 5$ and $r \geq 3p+2$ (so that the results of the preceding \cref{sec:Q} apply).
We refer the reader to \cite{BG} and \cite{B03} for details but summarize the formulae needed throughout this section.

For $m = 0$ we set $I_0 = \{ 0\}$ and for $m > 0$ we let $I_m = \{ [\lambda_0] + [\lambda_1] p + \cdots + [\lambda_{m-1}]p^{m-1} \from
\lambda_i \in \mathbb{F}_p \}$, where $[\cdot]$ denotes the Teichmüller representative.
For $m \geq 1$, there is a truncation map  $[ \cdot ]_{m-1} \from I_m \rightarrow I_{m-1}$ given by taking the first $m-1$ terms in the $p$-adic expansion above.
For $m = 1$, the truncation map is the $0$-map.
Let $\alpha = \begin{pmatrix}
  1 &0\\
  0&p
\end{pmatrix} $.
For $m \geq 0$ and $\lambda \in I_m$, let
\[
  g_{m,\lambda}^0 = \begin{pmatrix}
    p^m &\lambda\\
    0&1
  \end{pmatrix} \text{ and }
  g_{m,\lambda}^1 = \begin{pmatrix}
    1 & 0\\
    p \lambda &p^{m + 1}
  \end{pmatrix},
\]
where $g_{0,0}^0 = \id$ and $g_{0,0}^1 = \alpha$.
We have the decomposition $G =\amalg_{i = 0,1} KZ(g^i_{m,\lambda})^{-1} $.

An element in $\ind_{KZ}^G V$ is a finite sum of functions of the form $[g,v]$ where $g =g_{m,\lambda}^0$ or $g_{m,\lambda}^1$ for some $\lambda \in I_m$ and  $v = \sum_{i = 0}^r c_i  X^{r-i} Y^i \in V = \Sym^r R^2 \otimes \rD^s$.

The Hecke operator $T$ that acts on $\ind_{KZ}^G \Sym^r \overline{\mathbb{Q}}_p^2 $ can be written as $T = T^+ + T^-$, where:
\begin{align}
  T^ + ([g^0_{n,\mu}, v]) & = \sum_{\lambda \in I_1} \bigg[ g_{n + 1, \mu + p^n \lambda}^0, \sum_{j = 0}^r \bigg( p^j \sum_{i = j}^r c_i \binom{i}{j} (-\lambda)^{i-j}\bigg) X^{r-j} Y^j \bigg],\\
  T^-([g^0_{n,\mu}, v]) & = \bigg[g^0_{n-1, [\mu ]_{n-1} }, \sum_{j = 0}^r \bigg(\sum_{i = j}^r p^{r-i} c_i \binom{i}{j} \bigg(\frac{\mu - [\mu]_{n-1}}{p^{n-1}}\bigg)^{i-j}\bigg) X^{r-j} Y^j \bigg] \quad (n > 0),\\
  T^-([g^0_{n,\mu}, v]) & = [\alpha, \sum_{j = 0}^r p^{r-j} c_j X^{r-j} Y^j] \quad (n = 0).
\end{align}
We will use these explicit formulae for $T$ to eliminate all but one of the Jordan-Hölder factors from \cref{sec:Q} to be able to apply \cite[Proposition 3.3]{BG09}.

To explain the calculations using the $T^+$ and $T^-$ operators, we use the following heuristic:

\begin{itemize}
  \item For $T^+$, we note that the terms with $p^j$  appear depending on the valuation of $c_i$.
    For example if $c_i = \frac{1}{p a_p}$, then $v(c_i) > -4$, so we need to consider only the first $4$ values of $j$, while the terms for $j \geq 4$ vanish as $p^j$ kills $c_i$.
  \item For $T^-$ we typically consider the highest index $i$ for which $c_i \not \equiv 0$ as $p^{r-i}$ usually kills the other $c_i$ terms.
    For example, if $c_{r-1} \not \equiv 0$, then the terms in $T^-$, which we consider are $p c_{r-1} \binom{r-1}{j}(-\lambda)^{r-1-j}$.
\end{itemize}

\begin{lemma}
  \label{lem:Analog85}
  Let $5 \leq a \leq p + 3$.  We have the short exact sequence of $\Gamma$-modules:
  \[
    0 \rightarrow J_0:= V_{a-4} \ox \rD^2 \rightarrow V_r^{**}/V_r^{***} \rightarrow J_1 := V_{p-a + 3} \ox \rD^{a-2} \rightarrow 0,
  \]
  which splits for $a=p+3$ and
  \begin{enumerate}
    \item The monomials $X^{a-4}, Y^{a-4} \in J_0$ map to $\theta^2 X^{r-2p-2}$, $\theta^2 Y^{r-2p-2}$, respectively, in $V_r^{**}/V_r^{***}$.
    \item The polynomials  $\theta^2 X^{r-2p-2}$, $\theta^2 Y^{r-2p-2}$ map to $0 \in J_1$ and $\theta^2 X^{r-2p-a + 2}Y^{a-4}$, $\theta^2 X^{r-3p-1}Y^{p-1}$ map to $X^{p-a + 3}$, $Y^{p-a + 3}$, respectively in $J_1$.
  \end{enumerate}
\end{lemma}
\begin{proof}
  Following \cite[Lemma 8.5]{BG}, we have the following sequence:
  \begin{align}
    0 \to & V_{a-4} \ox \rD^2 \rightarrow V_r^{**}/V_r^{***} \iso V_{r-2p-2}/V_{r-2p-2}^* \ox \rD^2\\
    \overset{\psi^{-1}}\rightarrow & V_{p + a-5}/V_{p + a-5}^* \otimes \rD^2 \overset{\beta} \rightarrow V_{p-a + 3} \ox \rD^{a-2} \rightarrow 0.
  \end{align}
  where the map $\psi^{-1}$ is from \cite[(4.2)]{G} and $\beta$ from \cite[Lemma 5.3]{B03}.
  Under these maps $\psi^{-1}:  X^{r-2p-a + 2}Y^{a-4}  \mapsto X^{p-1}Y^{a-4}$ and $\beta : X^{p-1}Y^{a-4} \mapsto X^{p-a + 3}$. Similarly $\psi^{-1}:  X^{r-3p-1}Y^{p-1}  \mapsto X^{a-4}Y^{p-1}$ and $\beta : X^{a-4}Y^{p-1} \mapsto Y^{p-a + 3}$. The sequence splits for $a=p+3$ as $V_{p-1}$ is an injective module.
\end{proof}

\subsection{$r$ has the same representative mod $(p-1)$ and $p$}

\begin{proposition}
  \label{prop:elimrEqA}
  Let $a = 6, \ldots,p-1$.
  If  $r \equiv a \mod (p-1)$ and $r \equiv a \mod p^2$, then there is a surjection
  \[
    \ind^G_{KZ} (V_{p-a-1} \ox \rD^a)  \twoheadrightarrow \overline{\Theta}_{k,a_p}.
  \]
\end{proposition}
\begin{proof}
  By \cref{prop:QrEqA}, we have the following Jordan-Hölder series of $Q$:
  \[
    0  \rightarrow V_r^*/V_r^{***} \rightarrow Q \rightarrow V_{p-a-1} \ox \rD^a \rightarrow 0.
  \]
  To eliminate the factors coming from $V_r^{**} /V_r^{***}$ we consider $f = f_0 + f_1  \in \ind_{KZ}^G \Sym^r \overline{\mathbb{Q}}_p^2$, given by:
  \begin{align}
    f_1 & = \sum_{\lambda \in \mathbb{F}^*_p} \bigg[ g^0_{1,[\lambda]}, \dfrac{p^2}{a_p} [\lambda]^{p-3}(Y^r - X^{r-a}Y^a)\bigg] + \bigg[g_{1,0}^0, \dfrac{\binom{r}{2}(1-p)}{a_p}(X^2Y^{r-2} - X^{r-a + 2}Y^{a-2})\bigg],\\
    f_0 & = \bigg[\id,  \frac{p^2 (p-1)}{a_p^2} \sum_{\substack{0 < j < r-2 \\ j \equiv a-2\mod (p-1)}} \gamma_j X^{r-j}Y^j \bigg],
  \end{align}
  where the $\gamma_j$ are integers as in \cref{lem:analog7cg}.

  In $f_1$, for the first part we observe that $v (p^2/a_p) > -1$, so we consider only the term with $j=0$ for the first part of $T^+ f_1$.
  For $j=0$, we observe $\binom{r}{0} - \binom{a}{0} = 0 $.
  Regarding the second part, we note that $v (1/a_p) > -3$, so we consider the terms with $j=0,1,2$ for the second part of $T^+ f_1$.
  For $j=0$, we see that $\binom{r-2}{0} - \binom{a-2}{0} = 0 $.
  For $j=1,2$ we obtain $\frac{p^j}{a_p} (\binom{r-2}{j} - \binom{a-2}{j}) \equiv 0 \mod p$ as $r \equiv a \mod p^2$.
  Thus $T^+ f_1 \equiv 0 \mod p$.

  In $f_0$ we see that $v(p^2/a_p^2) > -4$.
  Due to the properties of $\gamma_j$ from  \cref{lem:analog7cg}, we have $\sum \binom{j}{n} \gamma_j \equiv 0 \mod p^{4-n}$ and $j \equiv a -2 \geq 4$, so the terms in $T^+ f_0 $ vanish $\mod p$.
  In $f_0$ the highest index $i$ for which $c_i \not \equiv 0 \mod p$ is $i= r-p-1 $.
  So we have $p^{r-i} = p^{p+1}$, which kills $p^2/a_p^2$ as $p \geq 5$.
  Thus $T^- f_0 \equiv 0 \mod p$ .

  For $T^- f_1$, we note that the highest terms for which $c_i \not \equiv 0$ are $i = r $ and $i= r-2$.
  In the case $i=r-2$ we note that it forces $j=r-2$ (as $\lambda =0$), so the non-zero term is $\dfrac{p^2(1-p)}{a_p} \binom{r}{2} X^{2}Y^{r-2}$.
  If $i=r$, then
  \[
    T^-f_1 = \bigg[\id, \frac{(p-1)p^2}{a_p} \sum_{\substack{0 < j \leq r-2 \\ j \equiv a-2 \mod (p-1)}} \binom{r}{j}X^{r-j}Y^j\bigg]
  \]
  The last term in the above expansion (when $j= r-2$) is
  \[
    \dfrac{p^2 \binom{r}{2}(p-1)}{a_p}X^2Y^{r-2},
  \]
  which is cancelled out by the term for $i=r-2$.
  Thus:
  \[
    T^-f_1 - a_p f_0 = \bigg[\id,  \frac{(p-1)p^2}{a_p} \sum_{\substack{0 < j < r-2 \\ j \equiv a-2 \mod (p-1)}}\bigg(\binom{r}{j} - \gamma_j \bigg)X^{r-j}Y^j\bigg]
  \]
  where the $\gamma_j \equiv \binom{r}{j} \mod p$ due to \cref{lem:analog7cg}, so $T^-f_1 - a_p f_0 \equiv 0 \mod p$.

  So $(T-a_p)f = -a_p f_1 \mod p $ and as $r \equiv a \mod p$ we have
  \begin{align}
      & (T-a_p) f \\
    \equiv{} & -\bigg[g_{1,0}^0, \binom{r}{2}(1-p) (X^2Y^{r-2} - X^{r-a + 2}Y^{a-2})\bigg]\\
    \equiv{} & -\bigg[g_{1,0}^0, \binom{a}{2} \theta \bigg(\sum_{i = 0}^{\frac{r-a}{p-1}-1}   X^{1+ i(p-1)}Y^{r-i(p-1) - p -2} \bigg)\bigg] \\
    \equiv{} & -\bigg[g_{1,0}^0, \binom{a}{2} \theta^2 \bigg(\sum_{i = 0}^{\frac{r-a}{p-1}-2} (i+1)   X^{ i(p-1)}Y^{r-i(p-1) -2 p -2} \bigg)\bigg] \\
    \equiv{} & \bigg[g_{1,0}^0, \binom{a}{2}\theta^2\big(X^{r-2p-a + 2}Y^{a-4}-Y^{r-2p-2}\big)\bigg] \mod V^{***}_r.
  \end{align}

 We follow the argument of \cite[Lemma 8.6]{BG} now. Let $v$ be the image of $\binom{a}{2}\theta^2(X^{r-2p-a + 2}Y^{a-4}-Y^{r-2p-2})$ in $V_r^{**}/V_r^{***}$.
  By \cref{lem:Analog85} the reduction $(\overline{T-a_p})f $ maps to $[g_{1,0}^0, \binom{a}{2}X^{p-a + 3}] \neq 0$  in $J_1 = V_{p-a + 3} \ox \rD^{a-2}$ .
  Because the short exact sequence for the Jordan-Hölder series of $V_r^{**}/V_r^{***}$ is non-split, the element $[g_{1,0}^0,v]$ generates $\ind_{KZ}^G (V_r^{**}/V_r^{***} )$ over $G$.

  To eliminate the factors coming from $V_r^{*} /V_r^{**}$ we consider $f = f_0 + f_1  \in \ind_{KZ}^G \Sym^r \overline{\mathbb{Q}}_p^2$, given by:
  \begin{align}
    f_1 & = \sum_{\lambda \in \mathbb{F}^*_p} \bigg[ g^0_{1,[\lambda]}, \dfrac{p}{a_p} [\lambda]^{p-2}(Y^r - X^{r-a}Y^a)\bigg] + \bigg[g_{1,0}^0, \dfrac{r(1-p)}{a_p}(XY^{r-1} - X^{r-a + 1}Y^{a-1})\bigg],\\
    f_0 & =\bigg [\id, \frac{p (p-1)}{a_p^2} \sum_{\substack{0 < j < r-1 \\ j \equiv a-1\mod (p-1)}}  \beta_jX^{r-j}Y^j \bigg],
  \end{align}
  where the $\beta_j$ are the integers from \cref{lem:analog7bg}' where thanks to the condition $r \equiv a \mod p^2$ we have $\beta_j \equiv \binom{r}{j} \mod p^2$ and $\sum \binom{j}{n}\beta_j \equiv 0 \mod p^{5-n}$.

  In $f_1$ for the first part we have $v (p/a_p) > -2$, so we consider the terms with $j=0,1$ for the first part of $T^+ f_1$.
  For $j=0$, we see that $\binom{r}{0} - \binom{a}{0} = 0 $ while for $j=1$, we see that $\frac{p}{a_p}\big(\binom{r}{1} - \binom{a}{1} \big) \equiv 0 \mod p $ as $r \equiv a \mod p^2$.
  Regarding the second part, we note that $v(1/a_p) > -3$, so we consider the terms in $T^+ f_1$ for $j=0,1,2$.
  For $j=0$ we see that $\binom{r-1}{0} - \binom{a-1}{0} = 0 $ while for $j=1,2$, we see that $\frac{p^j}{a_p} \big(\binom{r-1}{j} - \binom{a-1}{j}\big) \equiv 0 \mod p^2 $ as $r \equiv a \mod p^2$ .
  Thus $T^+ f_1 \equiv 0 \mod p$.

  In $f_0$ we see that $v(p/a_p^2)> -5$.
  Due to the properties of $\beta_j$, we have $\sum \binom{j}{n}\beta_j \equiv 0 \mod p^{5-n}$ (as $r \equiv a \mod p^2$) and $j \equiv a-1 \geq 5$, so the terms in $T^+ f_0 $ vanish $\mod p$.
  In $f_0$ the highest index $i$ for which $c_i \not \equiv 0 \mod p$ is $i= r-p $.
  Thus, $p^{r-i} = p^{p}$ but $p \geq 5$, so $T^- f_0 \equiv 0 \mod p$ .

  For $T^- f_1$, we note that the highest terms for which $c_i \not \equiv 0$ are $i = r $ and $i= r-1$.
  In case $i=r-1$, we note that it forces $j=r-1$ (as $\lambda =0$), so the nonzero term is $\frac{pr(1-p)}{a_p}  XY^{r-1}$.
  If $i=r$, then
  \[T^-f_1 = \bigg[\id, \frac{(p-1)p}{a_p} \sum_{\substack{0 < j \leq r-1 \\ j \equiv a-1 \mod (p-1)}}\binom{r}{j}X^{r-j}Y^j\bigg].\]
  The last term in the above expansion ($j= r-1$) is $\dfrac{p\binom{r}{1}(p-1)}{a_p}XY^{r-1}$, which is cancelled out by the term for $i=r-1$.
  Thus:
  \[
    T^-f_1 - a_p f_0 = \bigg[\id, \frac{(p-1)p}{a_p} \sum_{\substack{0 < j < r-1 \\ j \equiv a-1 \mod (p-1)}} \bigg( \binom{r}{j} - \beta_j \bigg)X^{r-j}Y^j \bigg]
  \]
  where the $\beta_j \equiv \binom{r}{j} \mod p^2$, so $T^-f_1 - a_p f_0 \equiv 0 \mod p$.
  Thus $(T-a_p)f = -a_p f_1 \mod p $, and
  \[
    (T-a_p) f  \equiv -\bigg[g_{1,0}^0, r(1-p) (XY^{r-1} - X^{r-a + 1}Y^{a-1})\bigg].
  \]
  The rest follows as in the proof of \cite[Lemma 8.6]{BG}, so we can eliminate the factors from $V_r^{*} /V_r^{**}$.
  Thus, the only remaining factor is $V_{p-a-1} \ox \rD^a$.
\end{proof}

  \begin{proposition}
    \label{prop:elimrEq5p}
    If  $r \equiv 5 \mod (p-1)$ and $r \equiv 5 \mod p^2$, and,
    when $v(a_p) = \frac{5}{2}$, assume that $v(a_p^2 - p^5) = 5$.
    Then
    \[
      \ind^G_{KZ} (V_{p-6} \ox \rD^5)  \twoheadrightarrow \overline{\Theta}_{k,a_p}.
    \]
  \end{proposition}
  \begin{proof}
    The Jordan-Hölder series of $Q$ is the same as in  \cref{prop:elimrEqA}.
    We will eliminate the factors from  $V_r^{*}/V_r^{**}$ and  $V_r^{**}/V_r^{***}$ leaving us with $V_{p-a-1} \ox \rD^a$ as in \cref{prop:elimrEqA}.

    To eliminate the terms from $V_r^{*}/V_r^{**}$, we distinguish two cases:
    \begin{itemize}
      \item
        If $v(a_p) \leq 5/2$ we use the functions from \cref{prop:elimrEqA}.

        For $\lambda \ne 0$ and $j \equiv 4 \mod (p-1)$, in $T^+ f_0$ we get the terms
        \[
          \frac{p^5(p-1)}{a_p^2} \sum_{\lambda} [g_{1, [\lambda]}, \sum_j \beta_j \binom{j}{4}X^{r-j}Y^j]
        \]
        which vanish, because $\sum_j \beta_j \binom{j}{4} \equiv \binom{r}{4} \sum \binom{r-4}{j-4} \equiv  0 \mod p$ by noting that $r-4 \equiv p \mod (p-1)$ while applying \cref{lem:analog7bg}'.

        Then $T^+ f_0$ also has the term $\frac{p^5(p-1)}{a_p^2} \beta_4 X^{r-4}Y^4$, which is integral as $v(a_p) \leq 5/2$.
        Noting that $\beta_4 \equiv 5 \mod p$, we can write
        $(T-a_p) f = T^+ f_0 - a_p f_1 $
        \[\equiv \bigg[g_{1,0}^0, \frac{5p^5(p-1)}{a_p^2}X^{r-4} Y^4 - 5(1-p) (XY^{r-1} - X^{r-4}Y^{4}) \bigg] \] and then follow the argument of \cite[Theorem 8.7]{BG}.

      \item
        If $v(a_p) > 5/2$, then consider $f' = \frac{a_p^2}{p^5}f$.
        All terms are zero except $T^+ f_0 = \bigg[g_{1,0}^0, \beta_4 X^{r-4} Y^4 \bigg]$ where $\beta_4 \equiv 5 \mod p$.
        By adding an appropriate term of $XY^{r-1}$, we can follow the argument as in \cref{prop:elimrEqA} to eliminate the factors from  $V_r^{*}/V_r^{**}$.

    \end{itemize}

    To eliminate the terms from $V_r^{**}/V_r^{***}$ we distinguish two cases:
    \begin{itemize}

      \item

        If $v(a_p) \leq 5/2$ we use the functions from \cref{prop:elimrEqA} but note that $T^+ f_0$ has the term $ \frac{p^5(p-1)}{a_p^2}\gamma_3 X^{r-3}Y^3$, which is integral as $v(a_p) \leq 5/2$.
        As $\gamma_3 \equiv 10 \mod p$, so we can write
        $(T-a_p) f = T^+ f_0 - a_p f_1$
        \[\equiv \bigg[g_{1,0}^0, \frac{10p^5(p-1)}{a_p^2} X^{r-3} Y^3 -\binom{5}{2}(1-p) (X^2Y^{r-2} - X^{r-3}Y^{3})  \bigg] \] and follow the argument as in the previous case.

      \item
        If $v(a_p) > 5/2$, then consider $f' = \frac{a_p^2}{p^5}f$.
        All terms are zero except $T^+ f_0 =\bigg[g_{1,0}^0, \gamma_3 X^{r-3} Y^3 \bigg]$.
        By adding an appropriate term of $X^2Y^{r-2}$, we can follow the argument as in the previous case to eliminate the factors from  $V_r^{**}/V_r^{***}$.
        \qedhere
    \end{itemize}
  \end{proof}

  \begin{proposition}
    \label{prop:elimrEqAp}
    Let $a = 5, \ldots,p-1$.
    If  $r \equiv a \mod p(p-1)$ but $r \not \equiv a \mod p^2$, (where in the case $a=5$ and $v(a_p ) = 5/2$ we assume $v(a_p^2 - p^5) =5$), then there is a surjection
    \[
      \ind^G_{KZ} (V_{p-a+1} \ox \rD^{a-1})  \twoheadrightarrow \overline{\Theta}_{k,a_p}.
    \]
  \end{proposition}
  \begin{proof}
    By \cref{prop:QrEqA}, we have the following Jordan-Hölder series of $Q$:
    \[
      0  \rightarrow V_r^*/V_r^{***} \rightarrow Q \rightarrow V_{p-a-1} \ox \rD^a \rightarrow 0.
    \]
    To eliminate the factors coming from $V_r^{**} /V_r^{***}$ we consider $f = f_0 + f_1  \in \ind_{KZ}^G \Sym^r \overline{\mathbb{Q}}_p^2$, given by:
    \begin{align}
      f_1 = & \sum_{\lambda \in \mathbb{F}^*_p}  \bigg[ g^0_{1, [\lambda]},  \dfrac{p}{a_p}[\lambda]^{p-2}(XY^{r-1}- X^{r-a+1}Y^{a-1}) \bigg]\\
            & + \bigg[ g^0_{1,0}, \dfrac{(1-r)}{a_p}\theta^2 (Y^{r-2p-2}- X^{p-1}Y^{r-3p-1}) \bigg], \quad \text{ and }\\
      f_0 = &  \bigg[\id, \frac{p^2(p-1)}{a_p^2}  \sum_{\substack{0 < j < r-2 \\ j \equiv a-2\mod (p-1)}} \beta_j  X^{r-j}Y^j \bigg],
    \end{align}
    where $\beta_j \equiv \binom{r-1}{j} \mod p$ and $\sum \binom{j}{n} \beta_j \equiv 0 \mod p^{4-n}$.
    We note that the existence of $\beta_j$ follows from \cref{lem:analog7bg} applied with $r-1$ instead of $r$.

    As $r \equiv a \mod p$, we see that $T^+ f_1 \equiv 0 \mod p$. For $f_0$ we see that the highest index $i = r-p-1 $ so $T^- f_0 \equiv 0 \mod p$. Also, as $v(p^2/a_p^2) < -4$ then by the properties of $\beta_j$ we have that $T^+ f_0 \equiv 0 \mod p$.

    For $T^- f_1$ we consider $i=r-1$ and $i = r-2$ to see that
    \[
      T^-f_1 - a_p f_0 = \bigg[\id,  \dfrac{(p-1)p^2}{a_p} \sum_{\substack{0 < j < r-2 \\ j \equiv a-2 \mod (p-1)}}\bigg( \binom{r-1}{j} - \beta_j \bigg)X^{r-j}Y^j \bigg]
    \]
    which vanishes modulo $p$ as $\beta_j \equiv \binom{r-1}{j} \mod p$.
    This means that
    \[
      (T-a_p) (f_1 + f_0) = -a_p f_1 = -(1-r)\bigg[ g^0_{1,0}, \theta^2 (Y^{r-2p-2}- X^{p-1}Y^{r-3p-1}) \bigg] .
    \]
    Thus, we can eliminate the factors from $V_r^{**}/V_r^{***}$.

    For $a=5$, we obtain $T^+ f_0 = [g^0_{2,0}, \frac{p^5}{a_p^2}\binom{r-1}{3}X^{r-3}Y^3]] \mod p$. If $v(a_p^2) < 5$ this term vanishes.

    If $v(a_p^2) > 5$, then consider $f' = \frac{a_p^2}{p^5}f$.
    All terms are zero except $T^+ f_0 =\bigg[g_{1,0}^0, \beta_3 X^{r-3} Y^3 \bigg]$.
    By adding an appropriate term of $X^2Y^{r-2}$, we can follow the argument as in \cref{prop:elimrEqA} to eliminate the factors from  $V_r^{**}/V_r^{***}$.

    In the case $v(a_p^2) = 5$, we assume that $v(a_p^2 - p^5) = 5$ as we get the extra non-zero term $a_p f'_1$. We then follow the same argument as \cref{prop:elimrEq5p}.

    To eliminate the factor $V_{p-a-1} \ox \rD^a$ we consider $f = f_0 + f_1  \in \ind_{KZ}^G \Sym^r \overline{\mathbb{Q}}_p^2$, given by:
    \begin{align}
      f_1 & = \sum_{\lambda \in \mathbb{F}^*_p} \bigg[ g^0_{1,[\lambda]}, \dfrac{1}{p^2}\big(Y^r - X^{r-a}Y^a\big) \bigg] + \bigg[g_{1,0}^0, \dfrac{(p-1)}{p}\big(Y^{r} - X^{r-a}Y^{a}\big)\bigg] \quad \text{ and }\\
      f_0 & = \bigg[\id, \frac{(p-1)}{p^2 a_p} \sum_{\substack{0 < j < r \\ j \equiv a\mod (p-1)}}  \alpha_jX^{r-j}Y^j \bigg],
    \end{align}
    where the $\alpha_j$ are the integers from  \cref{lem:analog7ag}' with the added conditions that $\alpha_j \equiv \binom{r}{j} \mod p^2$ and  $\sum \binom{j}{n} \alpha_j \equiv 0 \mod p^{5-n}$ as $r \equiv a \mod p$.

    For $T^+f_1$  in the first part of $f_1$ we note that $v(1/p^2) = -2$, so we need to consider $j=0,1,2$.
    For $j=0$, we see that $\binom{r}{0} - \binom{a}{0} = 0 $ while for $j=1$, we see that $\frac{p}{p^2}(\binom{r}{1} - \binom{a}{1}) = \frac{r - a}{p}$, which is integral as $p \mid r-a$, so the term involving $X^{r-1}Y$ maps to zero in $Q$.
    The term for $j=2$ is zero $\mod p$ as $r \equiv a \mod p$.
    For the second part, we note that $v(1/p) = -1$.
    The term with $j=0$ is identically zero while the coefficient of $X^{r-1}Y$  with $j=1$ is integral, which vanishes in $Q$.
    Thus $T^+ f_1 \equiv 0 \mod p$.

    In $f_0$ we see that $v(1/p^2 a_p)> - 5$.
    Due to the properties of $\alpha_j$, we have $\sum \binom{j}{n}\alpha_j \equiv 0 \mod p^{5-n}$ and $j \equiv a \geq 5$, so the terms in $T^+ f_0 $ vanish $\mod p$.
    Because the highest index $i$ for which $c_i \not \equiv 0 \mod p$ is $i= r-p+1$, we have $p^{r-i} = p^{p-1}$.
    Thus $T^- f_0 \equiv 0 \mod p$ for $p >5$.
    Note that $5 \leq a \leq p-1$ means that $p \geq 7$, so we do not need to worry about the case $p=5$.

    For $T^-f_1$ we note that the highest index of a nonzero coefficient is $i=r$, hence
    \[
      T^-f_1 = \bigg[\id,  \frac{(p-1)}{p^2} \sum_{\substack{0 < j \leq r \\ j \equiv a \mod (p-1)}}\binom{r}{j}X^{r-j}Y^j\bigg]
    \]
    The last term in the above expansion (when $j= r$) is $\dfrac{p(p-1)\binom{r}{r}}{p^2}Y^{r}$, which is cancelled out by the term for $i=r$ from the second part (where $\lambda=0$) which is $\dfrac{(1-p)}{p}Y^{r}$.
    We compute
    \[
      T^-f_1 - a_p f_0 = \bigg[\id,  \dfrac{(p-1)}{p^2} \sum_{\substack{0 < j < r \\ j \equiv a \mod (p-1)}}\bigg( \binom{r}{j} - \alpha_j \bigg)X^{r-j}Y^j \bigg]
    \]
    where the $\alpha_j \equiv \binom{r}{j} \mod p^2$, so $T^-f_1 - a_p f_0$ is integral.
    Now we follow the argument as in the proof of \cite[Theorem 8.3 ]{BG}.
    Applying Lemma 5(2) of \cite{A20} to our setting and using $r \equiv a \mod p$ yields
    \[
      \sum_{\substack{0 < j < r \\ j \equiv a \mod (p-1)}}\binom{r}{j}  \equiv p \frac{a-r}{a} \mod p^3.
    \]
    Thus, the expression maps to $\dfrac{a-r}{p a } X^{p-a-1}$, which is nonzero as $p^2 \not \mid a-r$.

    To eliminate the factor $V_{a-2} \ox \rD$, we consider $f = f_0 + f_1  \in \ind_{KZ}^G \Sym^r \overline{\mathbb{Q}}_p^2$, given by:
    \begin{align}
      f_1 & = \sum_{\lambda \in \mathbb{F}^*_p} \bigg[ g^0_{1,[\lambda]}, \dfrac{ [\lambda]^{p-2}}{p}\big(Y^r - X^{r-a}Y^a\big)\bigg]  + \bigg[g_{1,0}^0, \dfrac{-a}{p^2}\big(XY^{r-1} - X^{r-a+1}Y^{a-1}\big)\bigg] \quad \text{ and }\\
      f_0 & = \bigg[\id, \dfrac{(p-1)}{p a_p} \sum_{\substack{0 < j < r-1 \\ j \equiv a-1\mod (p-1)}}  \beta_jX^{r-j}Y^j \bigg],
    \end{align}
    where the $\beta_j$ are the integers from \cref{lem:analog7bg}.

    In $f_1$ we see that $v(1/p) = -1$, so we consider $j=0,1$.
    For $j=0$ we see $\binom{r}{0} -\binom{a}{0} = 0 $ while for $j=1$ we obtain $\frac{p}{p}(\binom{r}{1} -\binom{a}{1}) \equiv 0 \mod p$ as $ r-a \equiv 0 \mod p$.
    As $r-a \equiv 0 \mod p$ we see that $T^+ f_1 \equiv - \frac{ap(r-a)}{p^2} [g^0_{2,p[\lambda]}, X^{r-1}Y] \mod p$. Thus this term vanishes in $Q$.
    As $v(a_p) > 2$, we see that $a_p f_1 \equiv 0 \mod p $.

    For $f_0$, we note that $v(1/p a_p) > -4$ while $\sum \binom{j}{n}\beta_j \equiv 0 \mod p^{4-n}$ and $j \equiv a-1 \geq 4$, hence $T^+ f_0 \equiv 0 \mod p$.
    For $T^- f_0$ we note that the highest index is $i = r-p$, hence $p^{r-i} = p^{p}$, which kills $1/p a_p$ for $p \geq 5$.

    For $T^- f_1$, in the first part the highest index of a non-zero coefficient is $i=r$ while in the second part it is $i=r-1$, hence
    \[
      T^- f_1 \equiv{} \bigg[\id, \frac{(p-1)}{p} \bigg( \sum_{\substack{0 < j < r -1 \\ j \equiv a -1 \mod (p-1)}}  \binom{r}{j} X^{r-j}Y^j  + (r-a) XY^{r-1} \bigg)\bigg]
    \]
    We compute that
    \[
      T^-f_1 - a_p f_0 =\bigg[\id,  \frac{(p-1)}{p} \bigg(\sum_{\substack{0 < j < r -1 \\ j \equiv a -1 \mod (p-1)}}\bigg(\binom{r}{j} - \beta_j \bigg)X^{r-j}Y^j  + (r-a) XY^{r-1}\bigg)\bigg]
    \]

    As $p \mid \mid r-a$ and $\binom{r}{j} \equiv \beta_j \mod p $ we see that the above function is integral.
    As in \cite[Theorem 8.9(i)]{BG}, we change the above polynomial by a suitable $XY^{r-1}$ term so that it has the same image in $Q$ as
    \[\bigg[\id,(p-1) \big(F(X,Y)  + \frac{(a-r)}{p} \theta Y^{r-p-1} \big) \bigg],\]
    where:
  \[F(X,Y) =  \bigg[\id, \sum_{\substack{0 < j < r-1 \\ j \equiv a-1 \mod (p-1)}}\frac{1}{p} \bigg(\binom{r}{j} - \beta_j \bigg)X^{r-j}Y^j - \frac{(a-r)}{p} X^p Y^{r-p} \bigg)\bigg].\]

  We see that $F(X,Y)$ is integral as $\beta_j \equiv \binom{r}{j} \mod p$ and $r \equiv a \mod p$.
  By the conditions in \cref{lem:VrAstCriteria} and recalling that $\sum_j \beta_j \equiv 0 \mod p^4$, $\sum j \beta_j \equiv 0 \mod p^{3}$ and $r \equiv a \mod p$ we see that $F(X,Y) \in V_r^{**}$.
  Thus, $(T-a_p)f$ is equivalent to $ \dfrac{a-r}{p} \theta Y^{r-p-1}$, which, by \cite[Lemma 8.5]{BG}, maps to  $\dfrac{a-r}{p} Y^{a-2}$.
  This term is not zero as $r \not \equiv a \mod p^2$.
  Hence, the only surviving factor is $V_{p-a+1} \ox \rD^{a-1}$.
\end{proof}

\begin{proposition}
  \label{prop:elimaeqp}
  If $r \equiv p \mod (p-1)$ and $r  \equiv p \mod p$ (where in the case $p=5$ and $v(a_p ) = 5/2$ we assume $v(a_p^2 - p^5) =5$), then:
  \begin{enumerate}
    \item If $p^2 \nmid p-r$, then there is a surjection $ \ind^G_{KZ} (V_1)  \rightarrow \overline{\Theta}_{k,a_p}$.
    \item If $p^2 \mid p-r$, then there is a surjection $ \ind^G_{KZ} (V_{p-2} \otimes \rD)  \rightarrow \overline{\Theta}_{k,a_p}$.
  \end{enumerate}
\end{proposition}
\begin{proof}
  We follow the proof of \cite[Theorem 8.9]{BG}.
  By \cref{prop:QrEqp},
  \[
    0 \rightarrow V_r^*/V_r^{***} \rightarrow Q \rightarrow 0,
  \]
  that is, $Q \simeq V_r^* / V_r^{***}$.
  \begin{enumerate}
    \item

      To eliminate the factors from $ V_r^{**}/V_r^{***} $ we choose the functions as in  \cref{prop:elimrEqAp} putting $a=p$ and seeing that $p^2 \not \mid r-p$.

      To eliminate the factor $V_{p-2} \ox \rD$ we choose the functions $f = f_0 + f_1 +f_2 \in \ind_{KZ}^G \Sym^r \overline{\mathbb{Q}}_p^2$, given by:
      \begin{align}
        f_2 & = \sum_{\lambda \in \mathbb{F}_p} \bigg[ g_{2,p[\lambda]}^0, \dfrac{[\lambda]^{p-2}}{p} (Y^r - X^{r-p}Y^p)\bigg],\\
        f_1 & = \bigg[g_{1,0}^0, \dfrac{ (p-1)}{p a_p} \sum_{\substack{0 < j < r-1 \\ j \equiv 0 \mod (p-1)}}  \beta_jX^{r-j}Y^j \bigg], \quad \text{ and }\\
        f_0 & = \bigg[\id, \dfrac{(1-p)}{p}(X^r - X^p Y^{r-p})\bigg]
      \end{align}
      where the integers $\beta_j$ are those given in \cref{lem:analog7bg}.

      In $f_2$ we see that $v(1/p) = -1$, so we only consider $j=0,1$.
      For $j=0$ we see that $\binom{r}{0}-\binom{p}{0} = 0$ while for $j=1$ we obtain $\frac{p}{p}\big( \binom{r}{1} - \binom{p}{1} \big) \equiv 0 \mod p$ as $r \equiv p \mod p$.
      Thus $T^+ f_2 \equiv 0 \mod p$.
      Since $v(a_p) > 2$, we see that $a_p f_2 \equiv 0 \mod p$.

      In $f_1$ we see $v(1/pa_p) > -4$.
      Because $\sum \binom{j}{n} \beta_j \equiv 0 \mod p^{4-n}$, we have $T^+ f_1 \equiv 0 \mod p$.
      Since the highest index is $i = r-p$, we see that $p^{r-i} = p^{p}$ kills $1 /p a_p$ for $p \geq 5$, which means $T^- f_1 \equiv 0 \mod p$.

      In $f_0$ we have $v(1/p) = -1$, so we only consider $j=0,1$.
      For $\lambda \ne 0$ and $j=0$ we see that $\frac{(1-p)}{p}(\binom{0}{0}-\binom{r-p}{0})X^r =0 $ while for $j=1$, $\frac{p(1-p)}{p}\big( \binom{0}{1} - \binom{r-p}{1} \big) \equiv 0 \mod p$ as $r \equiv p \mod p$.
      However, if $\lambda =0$ we have that $\lambda^0=1$, so we consider $i=j=0$ and see that $T^+ f_0 \equiv  \bigg[g_{1,0}^0, \dfrac{(1-p)}{p}X^r \bigg]$.
      Since $v(a_p) > 2$, we see that $a_p f_0 \equiv 0 \mod p$.

      For $T^- f_2$, for $i=r$ we see that:
      \[T^-f_2  =  \bigg[g_{1,0}^0, \frac{(p-1)}{p}\sum_{\substack{0 \leq j \leq r-1 \\ j \equiv 0 \mod (p-1)}} \binom{r}{j}X^{r-j}Y^j \bigg].\]
      The last term above (when $j=r-1$) is $\dfrac{(p-1)r}{p}XY^{r-1} $ while the first term (when $j=0$) is cancelled out by $T^+ f_0 =  \bigg[g_{1,0}^0, \dfrac{(1-p)}{p}X^r \bigg]$.

      This yields
      \[ T^-f_2 - a_p f_1 +T^+ f_0 =  \bigg[g_{1,0}^0, \frac{(p-1)}{p}\bigg( \sum_{\substack{0 < j < r-1 \\ j \equiv 0 \mod (p-1)}} \bigg(\binom{r}{j} - \beta_j \bigg)X^{r-j}Y^j + r XY^{r-1}\bigg)\bigg]\]
      which is integral as $\beta_j \equiv \binom{r}{j} \mod p$ and $p \mid r$.

      Now, we follow the same argument as in the proof of \cite[Theorem 8.9(i)]{BG} to eliminate the factor $V_{p-2} \ox \rD$.
      Thus, we are left with the factor $V_1$.
    \item

      We first assume that $v(a_p^2) < 5$ if $p=5$. To eliminate the factors from $V_r^{**}/V_r^{***}$ we consider $f = f_0 + f_1 \in \ind_{KZ}^G \Sym^r \overline{\mathbb{Q}}_p^2$, given by:
      \begin{align}
        f_1 = & \sum_{\lambda \in \mathbb{F}_p^*} \bigg[ g_{1,[\lambda]}^0, \dfrac{p}{a_p} [\lambda]^{p-3}(Y^r - X^{r-p}Y^p)\bigg]\\
              & + \bigg[g_{1,0}^0,\dfrac{\binom{r}{2}(1-p)}{pa_p}(X^2Y^{r-2} - X^{r-p + 2}Y^{p-2})\bigg], \quad \text{ and }\\
        f_0 = & \bigg[\id, \frac{p (p-1)}{a_p^2} \sum_{\substack{0 < j < r-2 \\ j \equiv p-2 \mod (p-1)}}  \gamma_jX^{r-j}Y^j \bigg],
      \end{align}
      where the integers $\gamma_j$ are those given in \cref{lem:analog7cg} that satisfy $\gamma_j \equiv \binom{r}{j} \mod p^2$ due to the condition that $p^2 \mid p-r$.

      In $f_1$, we note that in the first part $v(p/a_p) > -2$, so for $T^+ f_1$ we consider $j=0,1$.
      For $j=0$ we see that $\binom{r}{0} - \binom{p}{0} = 0$ while for $j=1$ we see $\frac{p}{a_p} \big(\binom{r}{1} - \binom{p}{1} \big) \equiv 0 \mod p$ as $p^2 \mid r-p$.
      In the second part of $f_1$ we have $v(\binom{r}{2}/p a_p) > -3$, so we consider $j=0,1,2$.
      For $j=0$ we see that $\binom{r-2}{0} - \binom{p-2}{0} = 0$ while for $j=1,2$ we see $\frac{p^j \binom{r}{2}(1-p)}{pa_p}\big(\binom{r-2}{j} - \binom{p-2}{j} \big) \equiv 0 \mod p$ as $p^2 \mid r-p$.
      Thus, $T^+ f_1 \equiv 0 \mod p$.

      In $f_0$ we have $v(p/a_p^2) > - 5$.
      Because $\sum \binom{j}{n} \gamma_j \equiv 0 \mod p^{5-n}$, we have $T^+ f_0 \equiv 0 \mod p$.
      Note that for $p=5$, $T^+ f_0 = [g_{1,0}, \frac{p^4(p-1)}{a_p^2} \gamma_{3} \binom{3}{3}X^{r-3}Y^3 ]$.
      Because $v(a_p^2) < 5$ and $\gamma_3 \equiv \binom{r}{3} \equiv 0 \mod p$, we obtain $T^+ f_0 \equiv 0 \mod p$.
      Because the highest index is $i = r-p-1$, we see that $p^{r-i} = p^{p+1}$ kills $p / a_p^2$.
      Hence $T^- f_0 \equiv 0 \mod p$.

      For $T^- f_1$, for the first part ($i=r$) we see  that:
      \[T^-f_1  = \bigg[\id, \frac{(p-1)p}{a_p}\sum_{\substack{0 < j \leq r-2 \\ j \equiv p-2 \mod (p-1)}} \binom{r}{j}X^{r-j}Y^j \bigg].\]
      The last term when $j=r-2$ is $\dfrac{\binom{r}{2}p}{a_p}X^2Y^{r-2} $, which is cancelled out by the second part of $T^- f_1$ ($i=r-2$).
      This yields
      \[ T^-f_1 - a_p f_0 = \bigg[\id, \frac{(p-1)p}{a_p} \sum_{\substack{0 < j < r-2 \\ j \equiv p-2 \mod (p-1)}} \bigg(\binom{r}{j} - \gamma_j \bigg)X^{r-j}Y^j \bigg]\]
      which is zero $\mod p$ as $\gamma_j \equiv \binom{r}{j} \mod p^2$ while $v(p/a_p) > -2$.
      Hence
      \[
        (T-a_p) f
        \equiv - a_p f_1
        \equiv -\bigg[g_{1,0}^0, \dfrac{\binom{r}{2}(1-p)}{p}(X^2Y^{r-2} - X^{r-p + 2}Y^{p-2}) \bigg]
      \]
      By the hypothesis $\frac{r}{p} \equiv 1 \mod p$ and $r-1 \equiv p-1 \mod p$, so
      \[
        (T-a_p) f
        \equiv -\bigg[g_{1,0}^0, \dfrac{(p-1)(1-p)}{2}(X^2Y^{r-2} - X^{r-p + 2}Y^{p-2})\bigg].
      \]
      Therefore, as in \cref{prop:elimrEqA}
      \begin{align}
        & X^2Y^{r-2} - X^{r-p + 2}Y^{p-2} \\
        \equiv & - \theta^2(X^{r-3p + 2}Y^{p-4}-Y^{r-2p-2}) \mod V_r^{***}.
      \end{align}
      Thus, $\overline{(T-a_p) f}$ maps to $[g_{1,0}^0, X^3]$ by \cref{lem:Analog85}. Following previous arguments, this shows that we can eliminate the factors from $V_r^{**}/V_r^{***}$.

      In the case $p=5$ and $v(a_p^2 ) \geq 5$ we assume $v(a_p^2 - p^5) = 5$ if $v(a_p) = 5/2$ and follow the argument in the case $p=3$ in \cite [Theorem 8.9.(ii)]{BG}.

      We consider the function $f' = \frac{a_p^2}{p^5} f$ where $f$ is the function above, obtaining:
      \begin{align}
        f'_1 = & \sum_{\lambda \in \mathbb{F}_p^*} \bigg[ g_{1,[\lambda]}^0, \dfrac{a_p}{p^4} [\lambda]^{p-3}(Y^r - X^{r-p}Y^p) \bigg]\\
               & + \bigg[g_{1,0}^0, \dfrac{\binom{r}{2}(1-p)a_p}{p^6}(X^2Y^{r-2} - X^{r-p + 2}Y^{p-2}) \bigg], \quad \text{ and }\\
        f'_0  = & \bigg[\id, \sum_{\substack{0 < j < r-2 \\ j \equiv p-2 \mod (p-1)}} \frac{(p-1)}{p^4} \gamma_jX^{r-j}Y^j \bigg],
      \end{align}
      where the integers $\gamma_j$ are those given in \cref{lem:analog7cg} that satisfy $\gamma_j \equiv \binom{r}{j} \mod p^2$ due to the condition that $p^2 \mid p-r$.

      In $f'_1$  we have $v(a_p/p^4) > -2$ in the first part of $f'_1$, so we consider $j=0,1$.
      For $j=0$ we see that $\binom{r}{0} - \binom{p}{0} = 0$ while for $j=1$ we see $\frac{pa_p}{p^4} \big(\binom{r}{1} - \binom{p}{1} \big) \equiv 0 \mod p$ as $p^2 \mid r-p$.
      In the second part of $f'_1$ we see that $v(\binom{r}{2} a_p/p^6)> -3$, so we consider $j=0,1,2$.
      For $j=0$ we see that $\binom{r-2}{0} - \binom{p-2}{0} = 0$ while for $j=1,2$ we see $\frac{p^j \binom{r}{2}a_p}{p^6} \big(\binom{r-2}{j} - \binom{p-2}{j} \big) \equiv 0 \mod p$ as $p^2 \mid r-p$.
      Thus, the second part of $T^+ f'_1 \equiv 0 \mod p$ as well.

      In $f'_0$ we have $v(1/p^4) =-4$.
      The highest index in $f'_0$ is $i= r-p-1$, so $p^{r-i} = p^{p+1}$, which kills $1/p^4$.
      Hence, $T^- f'_0 \equiv 0 \mod p$.
      We obtain $T^+ f'_0 =[\id, \dfrac{\gamma_3 }{p} X^{r-3}Y^3]$ (observing that $p-2 =3$), which is integral as $\gamma_3 \equiv \binom{r}{3} \equiv 0 \mod p$.

      For $T^- f'_1$, for the first part (when $i=r$) we that:
      \[T^-f'_1  = \bigg[\id, \frac{(p-1)a_p}{p^4}\sum_{\substack{0 < j \leq r-2 \\ j \equiv p-2 \mod (p-1)}} \binom{r}{j}X^{r-j}Y^j \bigg].\]
      The last term above (when $j=r-2$) is $\dfrac{\binom{r}{2}a_p}{p^4}X^2Y^{r-2} $, which is cancelled out by the second part of $T^- f'_1$ (when $i=r-2$).
      This yields
      \[ T^-f'_1 - a_p f'_0 = \bigg[\id, \frac{(p-1)a_p}{p^4} \sum_{\substack{0 < j < r-2 \\ j \equiv p-2 \mod (p-1)}} \bigg(\binom{r}{j} - \gamma_j \bigg)X^{r-j}Y^j \bigg],\]
      which is zero $\mod p$ as $\gamma_j \equiv \binom{r}{j} \mod p^2$ while $v(a_p/p^4) > -2$.

      Hence $(T-a_p) f' \equiv - a_p f'_1 + T^+ f'_0$, which is equivalent to:
      \[
        -\bigg[g_{1,0}^0, \dfrac{\binom{r}{2}(1-p)(a_p^2)}{p^6}(X^2Y^{r-2} - X^{r-p + 2}Y^{p-2}) + \dfrac{\gamma_3 }{p} X^{r-3}Y^3\bigg].
      \]
      We note that as $r \equiv p \mod p^2$ and $p=5$, we have  $\gamma_3/p \equiv \binom{r}{3}/p \equiv 2 \mod p$ and $\binom{r}{2}/p \equiv 2 \mod p $.
      By adding a suitable term of $X^2Y^{r-2}$, we obtain
      \[
        (T-a_p) f'
        \equiv - a_p f'_1 + T^+ f'_0
        \equiv -\bigg[g_{1,0}^0, 2\bigg(1-\dfrac{a_p^2}{p^5}\bigg)(X^2Y^{r-2} - Y^{r-3}Y^{3})\bigg].
      \]
      We see that this is in $V_r^{**}/V_r^{***}$ as $r-5 \equiv 0 \mod p$ ($p=5$) and that its image under the projection $V_r^{**}/V_r^{***} \twoheadrightarrow V_{3} \otimes D^{p-3} $ is $2\bigg(1-\dfrac{a_p^2}{p^5}\bigg)X^3$. By the hypothesis we know that $1-\frac{a_p^2}{p^5} \ne 0$ so we can eliminate the factors from $V_r^{**}/V_r^{***}$.

      To eliminate the factor $V_{1}$ we choose the functions $f = f_0 + f_1 +f_2$ in $\ind_{KZ}^G \Sym^r \overline{\mathbb{Q}}_p^2$, given by:
      \begin{align}
        f_2 = & \sum_{\lambda \in \mathbb{F}_p^*} \bigg[ g_{2,p[\lambda]}^0, \dfrac{[\lambda]^{p-2}p}{a_p} (Y^r - X^{r-p}Y^p)\bigg]\\
              & + \bigg[g_{2,0}^0, \dfrac{r(1-p)}{pa_p}(XY^{r-1} - X^{r-p + 1}Y^{p-1})\bigg],\\
        f_1 = & \bigg[g_{1,0}^0, \dfrac{ (p-1)p}{a_p^2} \sum_{\substack{0 < j < r-1 \\ j \equiv 0 \mod (p-1)}}  \gamma_jX^{r-j}Y^j \bigg],
      \end{align}
      and
      \[
        f_0 = \bigg[\id, \dfrac{(1-p)p}{a_p}(X^r - X^p Y^{r-p})\bigg]
      \]
      where the integers $\gamma_j$ are those given in \cref{lem:analog7dg}.(i) that satisfy $\gamma_j \equiv \binom{r}{j} \mod p^2$ due to the condition that $p^2 \mid p-r$.

      In $f_2$  we see that $v(p/a_p) > -2$ in the first part of $f_2$, so we consider $j=0,1$.
      For $j=0$ we see that $\binom{r}{0} - \binom{p}{0} = 0$ while for $j=1$ we see $\frac{p^2}{a_p} \big(\binom{r}{1} - \binom{p}{1} \big) \equiv 0 \mod p$ as $p^2 \mid r-p$.
      In the second part of $f_2$ we see that $v(r/ p a_p)> -3$, so we consider $j=0,1,2$.
      For $j=0$ we see that $\binom{r-1}{0} - \binom{p-1}{0} = 0$ while for $j=1,2$ we see $\frac{p^j r}{pa_p} \big(\binom{r-1}{j} - \binom{p-1}{j} \big) \equiv 0 \mod p$ as $p^2 \mid r-p$.
      Thus $T^+ f_2 \equiv 0 \mod p$.

      In $f_1$ we have $v(p/a_p^2) > -5$.
      Since $\sum \binom{j}{n} \gamma_j \equiv 0 \mod p^{5-n}$, we see that $T^+ f_1 \equiv 0 \mod p$.
      Note that for $p=5$, $T^+ f_1 = [g_{1,0}, \frac{p^5}{a_p^2} \gamma_{4} \binom{3}{3}X^{r-4}Y^4] \mod p$ but $\gamma_4 \equiv \binom{r}{4} \equiv 0$, so $T^+ f_1 \equiv 0 \mod p$.
      Since the highest index is $i = r-p-1$, we see that $p^{r-i} = p^{p+1}$ kills $p / a_p^2$ hence $T^- f_0 \equiv 0 \mod p$.

      In $f_0$ we have $v(p/a_p) > -2$, so we only consider $j=0,1$.
      For $j=0$ we see that $\frac{(1-p)p}{a_p}(\binom{0}{0}-\binom{r-p}{0})X^r =\frac{(1-p)p}{a_p}X^r $ while for $j=1$ we obtain $\frac{p^2(1-p)}{a_p}\big( \binom{0}{1} - \binom{r-p}{1} \big) \equiv 0 \mod p$ as $r \equiv p \mod p$.
      Thus,
      \[
        T^+ f_0 = [g_{1,0}^0, \frac{(1-p)p}{a_p}X^r].
      \]
      For $T^- f_2$, for the first part ($i=r$) we that:
      \[
        T^-f_2  = \bigg[\id, \frac{(p-1)p}{a_p}\sum_{\substack{0 \leq j \leq r-1 \\ j \equiv 0 \mod (p-1)}} \binom{r}{j}X^{r-j}Y^j \bigg].
      \]
      The last term (when $j=r-1$) is $\dfrac{(p-1)r}{a_p}XY^{r-1} $, which is cancelled out by the second part of $T^- f_2$ (when $i=r-1$).
      The first term (when $j=0$) is cancelled out by $T^+ f_0$.
      This yields
      \[
        T^-f_2 - a_p f_1 +T^+ f_0 = \bigg[\id, \frac{(p-1)p}{a_p} \sum_{\substack{0 < j < r-1 \\ j \equiv 0 \mod (p-1)}} \bigg(\binom{r}{j} - \gamma_j \bigg)X^{r-j}Y^j \bigg]
      \]
      which is zero $\mod p$ as $\gamma_j \equiv \binom{r}{j} \mod p^2$ while $v(p/a_p) > -2$.
      Hence
      \[
        (T-a_p) f
        \equiv - a_p f_2
        \equiv - \bigg[g_{2,0}^0, \dfrac{r(1-p)}{p}(XY^{r-1} - X^{r-p + 1}Y^{p-1})\bigg].
      \]
      By assumption, $\frac{r}{p} \equiv 1 \mod p$.
      We then follow the same argument as in the proof of \cite[Thm 8.9(ii)]{BG} to eliminate the factor $V_1$.
      Thus, the only factor left is $V_{p-2} \ox \rD$.
      \qedhere
  \end{enumerate}
\end{proof}

\subsection{$r$ does \emph{not} have the same representative mod $(p-1)$ and $p$}

\begin{proposition}
  \label{prop:elimrnEqAa-1}
  If  $r \equiv a \mod (p-1)$ and $r \not \equiv a, a-1 \mod p$ for $5 \leq a \leq p$, then there is a surjection
  \[
    \ind^G_{KZ} (V_{p-a + 3} \ox \rD^{a-2})  \twoheadrightarrow \overline{\Theta}_{k,a_p}.
  \]
\end{proposition}
\begin{proof}
  By \cref{prop:Qrnota}, we have the following Jordan-Hölder series of $Q$:
  \[ 0 \rightarrow W \rightarrow Q \rightarrow V_{p-a-1} \ox \rD^a \rightarrow 0 \]
  where $W$ has $V_{p-a + 1} \ox \rD^{a-1}$ and $ V_{p-a + 3} \ox \rD^{a-2}$ as factors.

  To eliminate the factor $V_{p-a-1} \ox \rD^a$, we consider $f = f_0 + f_1 \in \ind_{KZ}^G \Sym^r \overline{\mathbb{Q}}_p^2 $, where:
  \begin{align}
    f_1 & = \sum_{\lambda \in \mathbb{F}_p}\bigg[g^0_{1, [\lambda]}, \frac{1}{p} (Y^r-X^{r-a}Y^{a})\bigg], \quad \text{ and }\\
    f_0 & = \bigg[\id, \frac{(p-1)}{p a_p} \sum_{\substack{0 < j < r \\ j \equiv a\mod (p-1)}} \alpha_j X^{r-j}Y^j \bigg],
  \end{align}
  where the $\alpha_j$ are chosen as in \cref{lem:analog7ag}.

  In $f_1$ we have $v(1/p)=-1$, so we consider only $j=0,1$ in $T^+ f_1$.
  For $j=0$, we obtain $\frac{1}{p}(\binom{r}{0}-\binom{a}{0}) = 0$ while for $j=1$ we obtain $\frac{p}{p}(\binom{r}{1}-\binom{a}{1})X^{r-1}Y$, which is integral and goes to zero in $Q$.
  Because $v(a_p) > 2$, we have $a_p f_1 \equiv 0 \mod p$.

  In $ f_0$ we note that $v(1/p a_p) > -4$.
  As $\sum \binom{j}{n} \alpha_j \equiv 0 \mod p^{4-n}$ and $j \equiv a \geq 4$, hence $T^+ f_0 \equiv 0 \mod p$.
  For $T^- f_0$ the highest index $i= r-(p-1)$ and $p^{r-i} = p^{p-1}$, which kills $1/pa_p$ for $p \geq 5$.
  For $T^- f_1$ we consider $i=r$, obtaining:
  \[
    T^-f_1= \bigg[\id, \frac{(p-1)}{p} \sum_{\substack{0 < j < r \\ j \equiv a \mod (p-1)}} \binom{r}{j}X^{r-j}Y^j  +p Y^r\bigg].
  \]
  Because $Y^r$ is sent to zero in $Q$:
  \[
    T^-f_1 - a_p f_0 = \bigg[\id, \frac{(p-1)}{p} \sum_{\substack{0 < j < r \\ j \equiv a \mod (p-1)}} \bigg(\binom{r}{j} - \alpha_j\bigg)X^{r-j}Y^j \bigg]
  \]
  which is integral as $\binom{r}{j} \equiv \alpha_j \mod p$.
  Following the same argument as in the proof of \cite[Theorem 8.3]{BG}, we see that $(T-a_p) f$ maps to $[\id, \frac{r-a}{a} X^{p-a-1}]$, which is nonzero as $r \not \equiv a \mod p$.
To eliminate the factor $V_{p-a+1} \ox \rD^{a-1}$, we consider $f = f_1 + f_0$, where
  \begin{align}
    f_1 = & \sum_{\lambda \in \mathbb{F}^*_p}\bigg[g^0_{1, [\lambda]},\dfrac{-1}{a-1}[\lambda]^{p-2}(Y^{r}-2X^{p-1}Y^{r-p+1} + X^{2p-2}Y^{r-2p+2})\bigg] \\
          & + \sum_{\lambda \in \mathbb{F}^*_p}\bigg[g^0_{1, [\lambda]} \frac{1}{p^2} (XY^{r-1}-2X^{p}Y^{r-p} + X^{2p-1}Y^{r-2p+1}) \bigg]\\
          & + \bigg[g_{1,0}^0, \dfrac{r}{p(a-1)}(XY^{r-1}- X^{r-a+1}Y^{a-1}) \bigg]
  \end{align}
  and
  \begin{align}
    f_0 = & \bigg[\id, \frac{(p-1)}{pa_p} \bigg( \dfrac{C_1}{p-1}(X^pY^{r-p} - X^{2p-1}Y^{r-2p+1})+ \sum_{\substack{0 < j < r-1 \\ j \equiv a-1\mod (p-1)}} D_j  X^{r-j}Y^j \bigg) \bigg]\\
    + & \bigg[\id, \bigg( \dfrac{C_1(2r-3p) - \sum\binom{j}{2} D_j }{pa_p}\bigg)(X^{r-a+1}Y^{a-1}-2 X^{r-a-p+2}Y^{p+a-2} + X^{r-a-2p+3}Y^{2p+a-3}) \bigg],
  \end{align}
  where
  \[
    D_j =\binom{r-1}{j} -  \bigg( \dfrac{p}{a-1} + O(p^2) \bigg)  \binom{r}{j}
  \]
  and $O(p^2)$ is chosen so that
  \[
    \sum_{\substack{0 < j < r-1 \\ j \equiv a-1\mod (p-1)}} D_j =0.
  \]
  We let $C_1 = -\sum j D_j$.
  By \cref{lem:binomialsumaeq2}:
  \begin{align}
    \sum_{\substack{0 < j < r-1 \\ j \equiv a-1\mod (p-1)}}j  D_j & \equiv \dfrac{p(r-a)(r-a+1)}{(a-1)(a-2)} \mod p^2,\\
    \sum_{\substack{0 < j < r-1 \\ j \equiv a-1\mod (p-1)}}\binom{j}{2}  D_j & \equiv \dfrac{p(r-1)(r-a)(r-a+1)}{(a-1)(a-3)} \mod p^2, \quad \text{ and }\\
    \sum_{\substack{0 < j < r-1 \\ j \equiv a-1\mod (p-1)}}\binom{j}{3}  D_j & \equiv \dfrac{p\binom{r-1}{2}(r-a)(r-a+1)}{(a-1)(a-4)} \mod p^2.
  \end{align}
  In the second part of $f_1$ we have $v(1/p^2) = -2$, so we consider $j=0,1,2$ for $T^+ f_1$.
  For $j=0$ we obtain $\binom{r-1}{0} - 2 \binom{r-p}{0} + \binom{r-2p+1}{0} = 0 $ while for $j=1$, we see that $\binom{r-1}{1} - 2 \binom{r-p}{1} +  \binom{r-2p+1}{1} = 0 $, too.
  For $j=2$ the term $X^{r-2}Y^2$ has integral coefficients, so it maps to zero in $Q$.
  In the first part, we only consider $j=0$ and see that $\binom{r}{0} - 2 \binom{r-p+1}{0} + \binom{r-2p+2}{0} = 0 $. In the third part of $f_1$ we see that $v(r/p(a-1) = -1)$ and it is clear that for $j=0$ we obtain $T^+ f_1 = 0$ while for $j =1$ the term is integral. Hence, $T^+ f_1 \equiv 0 \mod p$.
  As $v(a_p) > 2$, we see that $a_p f_1 \equiv 0 \mod p$.

  In $f_0$ we see that $v (1/p a_p) > -4$, so we need to consider $j=0,1,2,3$ for $T^+ f_0$.
    In the first part, we see that for $j=0,1 $ the terms with $X^r$ and $X^{r-1}Y$ vanish modulo $p$.
  For $j=2$, the term with $X^{r-2}Y^2$ has the coefficient $\frac{p(p-1)}{a_p}(C_1 (2r-3p)- \sum \binom{j}{2}D_j)$.

  In the second part, we see that for $j=0,1$ the terms vanish. For $j=2$, the term is $p^2/pa_p ( C_1(2r-3p) - \sum\binom{j}{2} D_j )(p^2-2p+1) \mod p$ which cancels the $j=2$ term from the first part.
  Finally, for $j=3$, in both terms we have that $v(C_1) \geq 1$ and $v(\sum \binom{j}{3} D_j) \geq 1$, so $T^+ f_0 \equiv 0 \mod p$ so the term vanishes. Hence $T^+ f_0 \equiv 0 \mod p$.
  As the highest $i= r-p$, we see that $T^- f_0 \equiv 0 \mod p$.

  For $T^- f_1$ we consider $i=r $ and $i=r-1$ from the first two parts and $i=r-1$ from the third part.
  For $i=r-1$ from the first part we obtain:
  \[
    \bigg[\id, \dfrac{(p-1)}{p} \sum_{\substack{0 < j \leq r-1 \\ j \equiv a-1\mod (p-1)}} \binom{r-1}{j} X^{r-j}Y^j \bigg],
  \]
  and for $i=r$:
  \[
    \bigg[\id, \dfrac{-(p-1)}{a-1} \sum_{\substack{0 < j < r-1 \\ j \equiv a-1\mod (p-1)}}   \binom{r}{j}  X^{r-j}Y^j \bigg].
  \]

We note that the term with $j=r-1$ above cancels with the term for $i=r-1$ from the third part.
  Thus:
  \[
    T^- f_1  \equiv  \sum_{\substack{0 < j < r-1 \\ j \equiv a-1\mod (p-1)}} \bigg[\id, \dfrac{(p-1)}{p} D_j X^{r-j}Y^j \bigg].
  \]
  We hence compute that $(T-a_p) f = T^- f_1 - a_p f_0 \equiv [\id, \dfrac{C_1}{p}\theta X^{p-1}Y^{r-2p}] $ (we note that $a_p f_0$ of the second part of $f_0$ vanishes $\mod V_r^{**}$). Thus $(T-a_p) f$ maps to $[\id, \dfrac{C_1}{p} X^{p-a+1}]$.
  Because $r \not \equiv a, a-1 \mod p$,
  \[
    \dfrac{C_1}{p} \equiv \dfrac{(r-a)(r-a+1)}{(a-1)(a-2)}  \not\equiv 0 \mod p.
  \]
  Hence, the only remaining factor is $V_{p-a + 3} \ox \rD^{a-2}$.
\end{proof}

\begin{proposition}
  \label{prop:elimrEqA-1}
  If  $r \equiv a \mod (p-1)$ and $p \mid r - a + 1 $ for $5 \leq a \leq p$, then there is a surjection
  \[
    \ind^G_{KZ} (V_{p-a + 1} \ox \rD^{a-1})  \twoheadrightarrow \overline{\Theta}_{k,a_p}.
  \]
\end{proposition}
\begin{proof}
  By \cref{prop:QrEqA-1}, we have the following Jordan-Hölder series of $Q$:
  \[ 0 \rightarrow W \rightarrow Q \rightarrow V_{p-a-1} \ox \rD^a \rightarrow 0 \]
  where $W$ has $V_{p-a + 1} \ox \rD^{a-1}, V_{a-4} \ox \rD^2$ and $ V_{p-a + 3} $ as factors.

  We can eliminate the factor $V_{p-a-1} \ox \rD^a$ by the functions in the proof of \cref{prop:elimrnEqAa-1} as $r \not \equiv a \mod p$.

  To eliminate the factors from $V_r^{**}/V_r^{***}$ we consider $f = f_1 + f_0$ where

  \begin{align}
    f_1 = & \sum_{\lambda \in \mathbb{F}^*_p}\bigg[g^0_{2, p[\lambda]},\dfrac{p^2}{a_p}[\lambda]^{p-3}(Y^{r}- X^{p-1}Y^{r-p+1})\bigg]\\
    - & \sum_{\lambda \in \mathbb{F}^*_p}\bigg[g^0_{2, p[\lambda]} \frac{rp}{a_p}[\lambda]^{p-2} (XY^{r-1}-2X^{p}Y^{r-p} + X^{2p-1}Y^{r-2p+1}) \bigg]\\
    + & \bigg[g_{2,0}^0, \dfrac{\binom{r}{2}(p-1)}{a_p}\theta^2(Y^{r-2p-2}- X^{p-1}Y^{r-3p-1}) \bigg], \quad \text{ and }\\
    f_0 = & \bigg[g_{1,0}^0,  \sum_{\substack{0 < j < r-2 \\ j \equiv a-2\mod (p-1)}}\dfrac{(p-1)p^2}{a_p}\alpha'_j  X^{r-j}Y^j \bigg],
  \end{align}
  where $\alpha'_j$ are chosen as in \cref{lem:analog7eg}.

  We see that $T^+ f_1 \equiv 0 \mod p$. For $f_0$, by the properties of the $\alpha'_j$, we obtain $T^+ f_0 \equiv 0 \mod p$ and as the highest index is $i = r-2-(p-1)$ we also have that $T^- f_0 \equiv 0 \mod p$.

  For $T^- f_1$ we consider $i = r, r-1, r-2$ to see that
  $$T^-  f_1 \equiv \bigg[g_{1,0}^0,  \sum_{\substack{0 < j < r-2 \\ j \equiv a-2\mod (p-1)}}\dfrac{(p-1)p^2}{a_p^2} \bigg( \binom{r}{j} - r \binom{r-1}{j}\bigg)  X^{r-j}Y^j \bigg) \bigg]$$.

  Since $\alpha'_j \equiv \binom{r}{j} - r \binom{r-1}{j} \mod p$ we that $T^- f_1 - a_p f_0 \equiv 0 \mod p$.

  Hence, $(T-a_p)(f_1 + f_0) = -a_p f_1 \equiv [g_{2,0}^0, \binom{r}{2}\theta^2(Y^{r-2p-2}- X^{p-1}Y^{r-3p-1}) ] $ which generates $V_r^{**}/V_r^{***}$ and the proposition follows.

  Hence, the only remaining factor is $V_{p-a + 1} \ox \rD^{a-1}$.
\end{proof}

\subsection{$r \equiv 3 \mod (p-1)$ }

In the following proposition we eliminate all but one Jordan-Hölder factor.
We note that while eliminating the factors from $V_r^{*}/V_r^{**}$ we  consider $a=3$ but while eliminating the factors from $V_r^{**}/V_r^{***}$, we consider $a=p+2$, following the convention set in the beginning of the paper in \cref{lem:VrAstCriteria}.

\begin{proposition}
  \label{prop:elimaeq3}
  If  $r \equiv 3 \mod (p-1)$, and:
  \begin{enumerate}
    \item If $r \not \equiv 0, 1, 2 \mod p$, then there is a surjection $\ind^G_{KZ} (V_{p-4} \ox \rD^3)  \twoheadrightarrow \overline{\Theta}_{k,a_p}$.
    \item If $r \equiv 0 \mod p$ then there is a surjection $\ind^G_{KZ} (V_{1} \ox \rD)  \twoheadrightarrow \overline{\Theta}_{k,a_p}$.
  \end{enumerate}
\end{proposition}
\begin{proof}
  \hfill
  \begin{enumerate}
    \item If $r \not \equiv 0,1, 2 \mod p$, then by \cref{prop:QrEq3} we already have the result.
    \item If $r \equiv 0 \mod p$, then to eliminate the factor $V_{p-4} \ox \rD^3$ we use the functions as in \cref{prop:elimrnEqAa-1} used to eliminate $V_{p-a-1} \ox \rD^a$ (for $a = 3$ ) but note that $T^+ f_0$ has the term $ \frac{p-1}{pa_p}p^3 \alpha_3 X^{r-3}Y^3= \frac{p-1}{pa_p}p^3 \binom{r}{3}X^{r-3}Y^3$ by \cref{lem:analog7ag}.
      As $p \mid r$ we see that $\binom{r}{3} =0$, so $T^+ f_1 =0$. The rest follows as in \cref{prop:elimrnEqAa-1}. Hence, the only remaining factor is $V_1 \ox D$.
      \qedhere
  \end{enumerate}
\end{proof}

\subsection{$r \equiv 4 \mod (p-1)$ }

In the following proposition we eliminate all but one Jordan-Hölder factor. We note that while eliminating the factors from $V_r^{*}/V_r^{**}$ we consider $a=4$ but while eliminating the factors from $V_r^{**}/V_r^{***}$, we consider $a=p+3$, following the convention set in the beginning of the paper in \cref{lem:VrAstCriteria}.

\begin{proposition}
  \label{prop:elimrEq4}
   Let $r \geq 3p+2$. If  $r \equiv 4 \mod (p-1)$ and:
   \begin{enumerate}
     \item If $r \equiv 4 \mod p$ (and $r \geq 5p$ for $p=5$) then there is a surjection $\ind^G_{KZ} (V_{p-5} \ox \rD^4)  \twoheadrightarrow \overline{\Theta}_{k,a_p}$.
     \item If $r \equiv 1 \mod p$ then there is a surjection $\ind^G_{KZ} (V_{0} \ox \rD^2)  \twoheadrightarrow \overline{\Theta}_{k,a_p}$.
     \item If $r \not\equiv 1,2,3,4 \mod p$ then there is a surjection  $\ind^G_{KZ} (V_{p-3} \ox \rD^3)  \twoheadrightarrow \overline{\Theta}_{k,a_p}$.
   \end{enumerate}
\end{proposition}
\begin{proof}

 \begin{enumerate}
    \item Let $r \equiv 4 \mod p$. To eliminate the factors from $V_r^{*}/V_r^{**}$ we consider $f = f_0 + f_1 \in \ind_{KZ}^G \Sym^r \overline{\mathbb{Q}}_p^2$, given by:
      \begin{align}
        f_1 & = \sum_{\lambda \in \mathbb{F}_p^*} \bigg[ g_{1,[\lambda]}^0, \dfrac{a_p[\lambda]^{p-2}}{p^3} (Y^r - X^{r-4}Y^4) \bigg] + \bigg[g_{1,0}^0, \dfrac{r a_p(1-p)}{p^4}(XY^{r-1} - X^{r-3}Y^{3})\bigg],\\
        f_0 & = \bigg[\id, \dfrac{1}{p^3}  \sum_{\substack{0 < j < r-1 \\ j \equiv 3 \mod (p-1)}} \beta_j X^{r-j}Y^j \bigg],
      \end{align}
      where the $\beta_j$ are chosen as in \cref{lem:analog7bg}.

      Using the properties of $\beta_j$ and $ r \equiv 4 \mod p$ we see that $T^+ f_1 \equiv 0 \mod p$. As $v(a_p) > 2$, we obtain $a_p f_1 \equiv 0 \mod p$.
      We compute that:

      \[ T^-f_1 - a_p f_0 = \bigg[\id, \frac{a_p(p-1)}{p^3} \sum_{\substack{0 < j < r-1 \\ j \equiv 3 \mod (p-1)}} \bigg(\binom{r}{j} - \beta_j \bigg)X^{r-j}Y^j \bigg]\]
      which is zero $\mod p$ as $\beta_j \equiv \binom{r}{j} \mod p$ and $v(a_p/p^3) > -1$.

      In $f_0$ as the highest $i = r-p$, we see that $T^- f_0 \equiv 0 \mod p$. However, for $j=3$, we obtain $T^+ f_0 =  [g_{1,0}^0, \beta_3 X^{r-3}Y^3] \equiv  [g_{1,0}^0, 4 X^{r-3}Y^3] \mod p$.
      Hence, $(T-a_p)f = T^+ f_0 = [g_{1,0}^0, 4 X^{r-3}Y^3]$.
      Since $XY^{r-1}$ maps to zero in $Q$, we see that $(T-a_p)f = T^+ f_0 \equiv [g_{1,0}^0,4 (X^{r-3}Y^3 - XY^{r-1})]$.
      Now, we follow the argument as in \cite[Theorem 8.6]{BG} (for $a=4$) and see that we can eliminate the factors from $V_r^{*}/V_r^{**}$.

      Thus, $ V_{p-5} \otimes \rD^4$ is the only remaining factor by \cref{prop:QrEq4}.

\item If $r \equiv 1 \mod p$, to eliminate the factors from $V_r^{*}/V_r^{**}$ we consider $f = f_2 + f_1 \in \ind_{KZ}^G \Sym^r \overline{\mathbb{Q}}_p^2$, given by:
      \begin{align}
        f_2 & = \sum_{\lambda \in \mathbb{F}_p} \bigg[ g_{2,[\lambda]}^0, \dfrac{1}{p^2} ( XY^{r-1} - 2X^{p}Y^{r-p} + X^{2p-1}Y^{r-2p+1}) \bigg] ,\\
        f_1 & = \bigg[g^0_{1,0}, \dfrac{1}{p a_p}  \sum_{\substack{0 < j < r-1 \\ j \equiv 3 \mod (p-1)}} \beta_j X^{r-j}Y^j \bigg],
      \end{align}
      where the $\beta_j \equiv \binom{r-1}{j}$ are chosen as in \cref{lem:analog7ag}. The existence of $\beta_j$ follows with $r = r-1$.

      Using the properties of $\beta_j$ and $ r \equiv 1 \mod p$ we obtain
      \[
        T^+ f_1 \equiv
        \sum_{k \in \mathbb{F}_p} [g^0_{2,p[\lambda]}, \frac{p^3}{p a_p} \beta_3 X^{r-3}Y^3] \equiv
        \sum_{k \in \mathbb{F}_p} [g^0_{2,p[\lambda]}, \frac{p^3}{p a_p} \binom{r-1}{3} X^{r-3}Y^3] \mod p.
      \]
      As $v(a_p) > 2$, we obtain $a_p f_2 \equiv 0 \mod p$. Finally, for $T^+ f_2$ the terms for $j=0,1$ vanish while for $j=2$ the term is integral, hence vanishes in $Q$.
      We modify $T^-f_2 - a_p f_1$ by a suitable $XY^{r-1}$ term to obtain:

           \[ (T - a_p)f = \bigg[\id, \frac{(p-1)}{p} \sum_{\substack{0 < j < r-1 \\ j \equiv 3 \mod (p-1)}} \bigg(\binom{r-1}{j} - \beta_j \bigg)X^{r-j}Y^j - (p-1) XY^{r-1} \bigg] \mod <p,X^{r-2}Y^2> \]

By \cref{lem:analog25bg}, we see that

      \[  \frac{1}{p} \sum_{\substack{0 < j < r-1 \\ j \equiv 3 \mod (p-1)}} \binom{r-1}{j} \equiv 1 \mod p\]

Hence, by \cref{lem:VrAstCriteria} we obtain $(T - a_p)f \in V_r^*$. As the coefficient of $c_{r-1} \not \equiv 0 \mod p$, we see it is not in $V_r^{**}$. Then, we apply \cite[Lemma 2.12]{GhateVangala} to find that $(T-a_p)f$ maps to a non-zero element in $V_{p-3} \ox \rD^2$. Thus, we are left with $V_{0} \ox \rD^2$.

    \item If $r \not \equiv 1,2,3,4 \mod p$, we can use the functions from \cref{prop:elimrnEqAa-1} to eliminate the factor $V_{p-5} \ox \rD^4$. Hence, by \cref{prop:QrEq4}, we are left with $V_{p-3} \ox \rD^2$.
      \qedhere
  \end{enumerate}
\end{proof}

\subsection{$r \equiv p+1 \mod (p-1)$ }

In the following proposition we eliminate all but one Jordan-Hölder factor.
We note that while eliminating the factors from $V_r^{*}/V_r^{**}$ and $V_r^{**}/V_r^{***}$, we consider $a=p+1$, following the convention set in the beginning of the paper in \cref{lem:VrAstCriteria}.

\begin{proposition}
  \label{prop:elimaeqp+1}
  If  $r \equiv p +1 \mod (p-1)$ and
      if $r \not \equiv 0, 1 \mod p$, then there is a surjection $ \ind^G_{KZ} (V_{2} )  \rightarrow \overline{\Theta}_{k,a_p}$.
\end{proposition}
\begin{proof}
      If $r \not \equiv 0,1 \mod p$, then by \cref{prop:QrEqp+1} we know that $V_2$ is the only factor.
\end{proof}


\section{Separating Reducible and Irreducible cases}
\label{sec:sepRedIrred}

We follow the methods of \cite[Section 9]{BG} to separate the reducible and irreducible cases when $\overline{\Theta}_{k,a_p}$ is a quotient of $\ind(V_{p-2} \otimes \rD^n)$.
This happens in \cref{prop:elimrnEqAa-1} (for $a = 5$) and \cref{prop:elimaeqp} (for $a = p$ and $p^2 \mid p-r$).
By \cite[Lemma 3.2]{BG13}, we need to check whether the map $\ind_{KZ}^G V_{p-2}) \rightarrow \overline{\Theta_{k, a_p}}$ factors through the cokernel of $T$ (in which case $\overline{V}_{k,a_p}$ is irreducible) or the cokernel of $T^2 - cT +1$ for some $c \in \overline{F}_p$ (in which case $\overline{V}_{k,a_p}$ is reducible).

The following theorem is based on \cite[Theorem 9.1]{BG}:

\begin{theorem}
  Let $r \equiv 5 \mod (p-1)$ and $r \not \equiv 4,5 \mod p$.
  If $r \not\equiv 2, 3 \mod p$, then we further assume that $v(a_p^2) \neq 5$.
  Then $\overline{V}_{k, a_p}$ is irreducible.
\end{theorem}
\begin{proof}
  We consider $f = f_1 + f_0 \in \ind_{KZ}^G \Sym^r \overline{\mathbb{Q}}_p^2 $, where
  \[
    f_1 = \sum_{\lambda \in \mathbb{F}_p}\bigg[g^0_{1,[\lambda]}, \frac{\theta^2}{a_p} (X^{r-2p-3}Y-Y^{r-2p-2})\bigg],
  \]
  and
  \[
    f_0 = \bigg[\id, \frac{p^2(p-1)}{a_p^2} \sum_{\substack{0 < j < r-2 \\ j \equiv 3 \mod (p-1)}} \alpha_j X^{r-j}Y^j \bigg],
  \]
  where the $\alpha_j$ are chosen similar to \cref{lem:analog7ag} with the condition that $\alpha_j \equiv \binom{r-2}{j} \mod p$.

  In the first part of $f_1$ as $v(1/a_p) > -3$, we consider $j=0,1,2$ for $T^+ f_1$.
  We see that $\theta^2 (X^{r-2p-3}Y-Y^{r-2p-2})= X^{r-2p-1}Y^{2p+1}-2 X^{r-p-2}Y^{p+2} + X^{r-3}Y^{3} + X^{2}Y^{r-2}-2 X^{p+1}Y^{r-p-1}+ X^{2p}Y^{r-2p}$.
  For $j=0,1$ we obtain that $T^+ f_1$ is identically zero.
  For $j=2$ we see that $\sum a_i \binom{i}{2} \equiv 0 \mod p$ where $a_i$ is the coefficient of $X^{r-i}Y^i$ in  $\theta^2 (X^{r-2p-3}Y-Y^{r-2p-2})$, so $T^+ f_1 \equiv 0 \mod p$.

  In $f_0$ we have $v(p^2/a_p^2) > -4$.
  As $\sum_j \binom{j}{n} \alpha_j \equiv 0 \mod p^{4-n}$ we obtain $T^+ f_0 \equiv \frac{p^5(p-1)}{a_p^2}\binom{r-2}{3}X^{r-3}Y^3 \equiv 0 \mod p$ since  $r \equiv 2,3 \mod p$ or $v(a_p^2) < 5$.
  Finally, in $f_0$ the highest $i= r-p-1$, so $p^{r-i} = p^{p+1}$ kills $p^2/a_p^2$ for $p \geq 5$.
  Thus, $T^- f_0 \equiv 0 \mod p$.

  For $T^- f_1$ we consider $i = r-2$ and obtain:
  \[ T^-f_1= \bigg[\id, \sum_{\substack{0 < j < r-2\\ j \equiv 3 \mod (p-1)}} \frac{(p-1)p^2}{a_p}\binom{r-2}{j}X^{r-j}Y^j + \frac{p^3}{a_p}X^2Y^{r-2}\bigg]. \]
  As $v(a_p) < 3$ we obtain:
  \[ T^-f_1 - a_p f_0 = \bigg[\id, \dfrac{(p-1)p^2}{a_p} \sum_{\substack{0 < j < r-2\\ j \equiv 3 \mod (p-1)}} \bigg(\binom{r-2}{j} - \alpha_j \bigg)X^{r-j}Y^j\bigg], \]
  which dies mod $p$ as $\alpha_j \equiv \binom{r-2}{j} \mod p$.

  Hence, $(T-a_p) f = -a_p f_1 =  \sum_{\lambda \in \mathbb{F}_p}[g^0_{1,[\lambda] }, \theta^2 (X^{r-2p-3}Y-Y^{r-2p-2})]$.

  By  \cref{lem:Analog85} this maps to $  \sum_{\lambda \in \mathbb{F}_p}[g^0_{1,[\lambda] }, - X^{p-2}]$, which equals $- T[\id, X^{p-2}]$.
  Thus, the reducible case cannot occur.

  If $r \not \equiv 2,3,4,5 \mod p$ and $v(a_p^2) > 5$, we consider the function $f'= \frac{a_p^2}{p^5}f$. We see that $a_p f'_1, T^+ f'_1, T^- f'_0 \equiv 0 \mod p$, while we see that:

  \[ T^-f'_1 - a_p f'_0 = \bigg[\id, \dfrac{(p-1)a_p}{p^3} \sum_{\substack{0 < j < r-2\\ j \equiv 3 \mod (p-1)}} \bigg(\binom{r-2}{j} - \alpha_j \bigg)X^{r-j}Y^j\bigg], \]

  Since $\binom{r-2}{j} \equiv \alpha_j \mod p$, the above function vanishes.

  Hence, we are left with

  $$T^+ f'_0 = (p-1) \sum_{\lambda \in \mathbb{F}_p}[g^0_{1,[\lambda] }, \binom{r-2}{3} X^{r-3}Y^3]  $$.

  Let $G(X,Y) = -\frac{1}{(p-1)}\sum_k k^{p-4}(\frac{1}{2}(X+kY)^r - X(X+kY)^{r-1}) \in X_{r-2}$.
  Working modulo $p$ we see that:

  \[G(X,Y) = \dfrac{-1}{2} \sum_{\substack{0 < j \leq r-2\\ j \equiv 3 \mod (p-1)}} \bigg(\binom{r}{j} - 2\binom{r-1}{j}\bigg)X^{r-j}Y^j.\]

  Let $ F(X,Y) = X^{r-3}Y^3 + G(X,Y) = \sum c_j X^{r-j} Y^j$. By \cref{lem:binomialsumaeq2} with $i=2$ we see that for $0 < j < r-2$, we have $\sum \binom{r}{j} \equiv \frac{(5-r)(4+r)}{2}$ while $\sum \binom{r-1}{j} \equiv 5-r $.
  The term for $j=r-2$ gives us $\binom{r}{r-2} -2 (r-1)$.
  Thus, for $0 < j < r-2$, we obtain $\sum c_j \equiv 0 \mod p$.
  By \cref{lem:binomialsumaeq2} for $0 < j < r-2$, we obtain $\sum j \binom{r}{j} \equiv \frac{r(5-r)(2+r)}{2}$ and $\sum j \binom{r-1}{j} \equiv (r-1)(5-r)$.
  For the term $j=r-2$ we have $(r-2)(\binom{r}{r-2} -2 (r-1))$.
  Thus, we see that $\sum j c_j \equiv 0 \mod p$. Hence, $X^{r-3}Y^3 + G(X,Y) \in V_r^{**}$.
  Using \cite[Lemma 2.12]{GhateVangala} we find that $F(X,Y)$ maps to a  non-zero element under the projection $V_r^{**}/V_r^{***}  \twoheadrightarrow V_{p-2} \ox \rD^3$. Hence, we see that $V_{p-2} \ox \rD^3$ contributes irreducibly.
\end{proof}

\begin{theorem}[Extension of {\cite[Theorem 9.2]{BG}}]
  Let $r \equiv p \mod (p-1)$ and $p^2 \mid p-r$. If $p = 5$ and $v(a_p^2) = 5$ then assume that $v(a_p^2 - p^5 )= 5$.
  Then:
  \begin{enumerate}
    \item If $p^3 \nmid p-r$, then $\overline{V}_{k, a_p}$ is irreducible.
    \item If $p^3 \mid p-r$, then $\overline{V}_{k, a_p} \cong \unr(\sqrt{-1})\omega \oplus \unr(-\sqrt{-1})\omega$ is reducible.
  \end{enumerate}
\end{theorem}
\begin{proof}
  \hfill
  \begin{enumerate}
    \item  
      Consider the function $f = f_0 + f_1 +f_2 \in \ind_{KZ}^G \Sym^r \overline{\mathbb{Q}}_p^2$, given by:
      \begin{align}
        f_2 & = \sum_{\lambda \in \mathbb{F}_p, \mu \in \mathbb{F}_p^*}  \bigg[ g_{2,p[\mu] + [\lambda]}^0, \dfrac{1}{p^2}(Y^r - X^{r-p}Y^p)\bigg]\\
            & + \sum_{\lambda \in \mathbb{F}_p}\bigg[g_{2,[\lambda]}^0, \dfrac{(1-p)}{p}(Y^r-X^{r-p}Y^p)\bigg],\\
        f_1 & =\sum_{\lambda \in \mathbb{F}_p} \bigg[g_{1,[\lambda]}^0, \sum_{\substack{1 < j < r \\ j \equiv 1 \mod (p-1)}} \frac{(p- 1)}{p^2 a_p} \gamma_jX^{r-j}Y^j \bigg], \quad \text{ and }\\
        f_0 & = \bigg[\id, \frac{r}{p^3}(X^{r-1}Y-X^{r-p}Y^p) \bigg],
      \end{align}
      where the integers $\gamma_j$ are given in \cref{lem:analog7dg}.

      In the first part of $f_2$ we have $v(1/p^2) =-2$, so we consider $j=0,1,2$ for $T^+ f_2$.
      For $j=0$ we have $\binom{r}{0} - \binom{p}{0} = 0$ while for $j=1,2$ we see that $\frac{p^j}{p^2}(\binom{r}{j}-\binom{p}{j}) \equiv 0 \mod p^2$ as $p^2 \mid r-p$.
      In the second part of $f_2$ we have $v(1/p) =-1$, so we consider $j=0,1$ for $T^+ f_2$.
      For $j=0$ we have $\binom{r}{0} - \binom{p}{0} = 0$ while for $j=1$ we see that $\frac{p}{p}(\binom{r}{1}-\binom{p}{1})\equiv 0 \mod p^2$ as $p^2 \mid r-p$.
      Thus $T^+ f_2 \equiv 0 \mod p$.

      In $f_1$ we have $v(1/p^2 a_p) > -5$.  By the properties of the $\gamma_j$ we have $\sum \binom{j}{n}\gamma_j \equiv 0 \mod p^{5-n}$, so $T^+ f_1 \equiv 0 \mod p$.
      We see that $a_p f_2$ and $a_p f_0$ die mod $p$ as $v(a_p) > 2$.

      In $f_0$, we have $v(r/p^3) = -2$.
      For $T^+ f_0$ we consider $j = 0,1,2$.
      For $j=0$ we obtain $\frac{r}{p^3} (\binom{1}{0} - \binom{p}{0}) =0$.
      For $j=1$, we obtain 
      \begin{align}
        & \sum_{\lambda \in  \mathbb{F}_p^*}\bigg[g^0_{1,[\lambda] },\frac{pr}{p^3} \bigg(\binom{1}{1} - \binom{p}{1}\bigg)X^{r-1}Y\bigg]+\bigg[g^0_{1,0}, \frac{r}{p^2}X^{r-1}Y\bigg]\\
        = & \sum_{\lambda \in  \mathbb{F}_p^*}\bigg[g^0_{1,[\lambda] },\frac{r(1-p)}{p^2} X^{r-1}Y\bigg]+\bigg[g^0_{1,0}, \frac{r}{p^2}X^{r-1}Y\bigg]\\
        = & \sum_{\lambda \in  \mathbb{F}_p}\bigg[g^0_{1,[\lambda] }, \frac{r(1-p)}{p^2} X^{r-1}Y]+\bigg[g^0_{1,0}, \frac{rp}{p^2}X^{r-1}Y\bigg]
        \end{align}
        The last term is integral so it vanishes in $Q$.
        For $j=2$, we obtain
        \[
          \frac{p^2r(1-p)}{p^3} (\binom{1}{2} - \binom{p}{2})X^{r-2}Y^2,
        \]
        which is integral, hence vanishes in $Q$.
        Hence
        \[
          T^+ f_0 = \sum_{\lambda \in  \mathbb{F}_p}\bigg[g^0_{1,[\lambda] },\frac{r(1-p)}{p^2} (1-p)X^{r-1}Y\bigg].
        \]
        In $T^- f_1$, the highest index of a nonzero coefficient is $i= r-p+1$.
        Therefore $p^{r-i} = p^{p-1}$ kills $1/p^2 a_p$ for $p \geq 7$.
        If $p=5$, we note that $T^- f_1 $ has the term $\frac{(p-1)p^4}{p^2a_p} \gamma_{4}$.
        As $r \equiv p \mod p^2$, we see that $\gamma_4 \equiv  \binom{r}{4} \equiv 0 \mod p$ and hence  $T^- f_1 \equiv 0$.

        For $T^- f_2$ we consider $i=r$ in the first part of $f_2$, obtaining:
        \[
          \sum_{\lambda \in \mathbb{F}_p} \bigg[g_{1,\lambda}^0, \dfrac{(p- 1)}{p^2}  \sum_{\substack{1 < j < r \\ j \equiv 1 \mod (p-1)}}  \binom{r}{j} X^{r-j}Y^j  + \dfrac{1}{p}Y^r + \dfrac{r(p-1)}{p^2}X^{r-1}Y \bigg].
        \]
        The term $\dfrac{1}{p}Y^r$ is cancelled out by the second part of $T^- f_2$, while the term $ \dfrac{r(p-1)}{p^2}X^{r-1}Y $ is cancelled out by $T^+ f_0$.
        Thus $(T-a_p) f \equiv T^- f_2 -a_p f_1 + T^+ f_0$ is equivalent to:
        \[
          \sum_{\lambda \in \mathbb{F}_p} \bigg[g_{1,[\lambda]}^0, \frac{(p- 1)}{p^2}  \sum_{\substack{1 < j < r \\ j \equiv 1 \mod (p-1)}}  \bigg( \binom{r}{j}-\gamma_j \bigg) X^{r-j}Y^j \bigg].
        \]
        As $\binom{r}{j} \equiv \gamma_j \mod p^2$ the above function is integral.
        Because each of the monomials $X^{r-j}Y^j$ maps to $X^{p-2}$ under the projection $V_r / V_r^* \twoheadrightarrow V_{p-2} \otimes \rD^3$ if $j \equiv 1 \mod (p-1)$, by the properties of $\sum_j \gamma_j$ the expression above maps to $cX^{p-2}$, where $c = \frac{(p-1)(p-r)}{p^2}$ due to \cref{lem:analog25bg}.

        As $p^2 \mid p-r$ this sum is integral, but is nonzero as $p^3 \nmid p-r$.
        Thus $(T-a_p)f = \sum_{\lambda \in \mathbb{F}_p} [g_{1,\lambda}^0, cX^{p-2}] = cT[\id, X^{p-2}]$, which means that  $\overline{V}_{k, a_p}$ is irreducible.

      \item 

       Assume $v(a_p) < 5/2  $ if $p=5$. We consider the function $f = f_0 + f_1 + f_2$, where:
        \begin{align}
          f_2 & = \sum_{\lambda \in \mathbb{F}_p,\mu \in \mathbb{F}_p^*} \bigg[ g_{2,p[\mu] + [\lambda]}^0, \dfrac{1}{a_p}(Y^r - X^{r-p}Y^p) \bigg] + \sum_{\lambda \in \mathbb{F}_p}\bigg[g_{2,[\lambda]}^0, \dfrac{(1-p)}{a_p}(Y^r-X^{r-p}Y^p)\bigg],\\
          f_1 & = \sum_{\lambda \in \mathbb{F}_p} \bigg[g_{1,[\lambda]}^0,  \dfrac{(p-1)}{a_p^2} \sum_{\substack{1 < j < r \\ j \equiv 1 \mod (p-1)}} \gamma_jX^{r-j}Y^j \bigg],
        \end{align}
        where the $\gamma_j \equiv \binom{r}{j} \mod p^3$ are chosen as in \cref{lem:analog7dg} and
        \[
          f_0 = \bigg[\id, \dfrac{r}{p a_p}(X^{r-1}Y-X^{r-p}Y^p) \bigg] + \bigg[\id, \dfrac{-r}{a_p}(X^{r-1}Y-2X^{r-p}Y^p + X^{r-2p+1}Y^{2p-1}) \bigg].
        \]
        In the first and second part of $f_2$ we have $v(1/a_p)> -3$, so we consider $j=0,1,2$ for $T^+ f_2$.
        For $j=0$ we have $\binom{r}{0} - \binom{p}{0} = 0$ while for $j=1,2$ we see that $\frac{p^j}{a_p}(\binom{r}{j}-\binom{p}{j}) \equiv 0 \mod p$ as $p^3 \mid r-p$. Thus, we  see that $T^+ f_2 \equiv 0 \mod p$.

        In $f_1$ we see that $v(1/a_p^2) > -6$.
        By the properties of the $\gamma_j$ we have $\sum \binom{j}{n} \gamma_j \equiv 0 \mod p^{6-n}$, so $T^+ f_1 \equiv 0 \mod p$.
        In $T^- f_1$ the highest index of a non-zero coefficient is $i= r-p+1$, and $p^{r-i} = p^{p-1}$ kills $1/a_p^2$ for $p \geq 7$.
        For $p=5$ we see that $T^- f_1 $ has the terms $ \frac{p^4}{a_p^2}  \binom{r}{r-4}  \equiv 0 \mod $ as $p^3 \mid r-p$ and $v(a_p^2) < 5$, so $T^- f_1 \equiv 0 \mod p$.

        In $f_0$, in the first part we have $v(r/p a_p) > -3$ so we consider $j = 0,1,2$.
        For $j=0$, we obtain $\frac{r}{p a_p} (\binom{1}{0} - \binom{p}{0}) =0$ and for $j=2$, we obtain  $\frac{p^2r}{pa_p} (\binom{1}{2} - \binom{p}{2}) \equiv 0 \mod p$.
        For $j=1$, we obtain
        \begin{align}
         & \sum_{\lambda \in \mathbb{F}_p^*} \bigg[g_{1,\lambda}^0, \frac{pr}{pa_p} \bigg(\binom{1}{1} - \binom{p}{1}\bigg)X^{r-1}Y \bigg] + \bigg[g^0_{1,0}, \frac{r}{a_p} X^{r-1}Y \bigg]\\
          = & \sum_{\lambda \in \mathbb{F}_p^*} \bigg[g_{1,\lambda}^0, \frac{r(1-p)}{pa_p} X^{r-1}Y \bigg] + \bigg[g^0_{1,0}, \frac{r}{a_p} X^{r-1}Y \bigg]\\
          = & \sum_{\lambda \in \mathbb{F}_p} \bigg[g_{1,\lambda}^0,\frac{r(1-p)}{a_p}X^{r-1}Y\bigg]  + \bigg[g^0_{1,0}, \frac{rp}{a_p} X^{r-1}Y \bigg]
        \end{align}
        In the second part, we see that for $j=0,2$, the term $T^+ f_0$ vanishes but for $j=1$ the term $[g^0_{1,0}, \frac{-rp}{a_p} X^{r-1}Y]$ prevails.
        Hence, $T^+ f_0 = \sum_{\lambda \in \mathbb{F}_p} [g_{1,\lambda}^0, \frac{r(1-p)}{a_p}X^{r-1}Y]$.

        For $T^- f_2$ we consider $i=r$ in the first part and obtain:
        \[
          \sum_{\lambda \in \mathbb{F}_p} \bigg[g_{1,\lambda}^0, \dfrac{(p- 1)}{a_p}  \sum_{\substack{1 < j < r \\ j \equiv 1 \mod (p-1)}}   \binom{r}{j} X^{r-j}Y^j + \dfrac{(p-1)}{a_p}Y^r + \dfrac{r(p-1)}{a_p}X^{r-1}Y \bigg].
        \]
        The term $\dfrac{(p-1)}{a_p}Y^r$ is cancelled out by the second part of $T^- f_2$, while the term $ \dfrac{r(p-1)}{a_p}X^{r-1}Y $ is cancelled out by $T^+ f_0$.

        Then $T^- f_2 -a_p f_1 + T^+ f_0$ is equivalent to:
        \[
          \sum_{\lambda \in \mathbb{F}_p} \bigg[g_{1,\lambda}^0, \dfrac{(p- 1)}{a_p}  \sum_{\substack{1 < j < r \\ j \equiv 1 \mod (p-1)}}  \bigg( \binom{r}{j}-\gamma_j \bigg) X^{r-j}Y^j \bigg],
        \]
        which is zero as $\gamma_j \equiv \binom{r}{j} \mod p^3$.

        Thus, $(T-a_p) f \equiv -a_p f_2 - a_p f_0 \mod p$.
        Following the argument given in the proof of \cite[Theorem 9.2]{BG}, this turns out to be the same as $(T^2+ 1)[\id, -X^{p-2}]$.
        Therefore the representation is reducible.

        If $p=5$ and $v(a_p) \geq 5/2$, then we are in a situation similar to \cite[Theorem 9.2]{BG} for $p=3$ and $v(a_p) \geq 3/2$.
        We consider the function $f' = \dfrac{a_p^2}{p^5}f$.
        Then $(T-a_p) f'$ is integral and has reduction equal to the image of $c(T^2 + 1) [\id, X^{p-2}]$ where $c = \overline{1-a_p^2/p^5}$, which by the extra hypothesis is not zero.
        Thus, the representation is reducible.
        \qedhere
  \end{enumerate}
\end{proof}

  \bibliographystyle{amsalpha}
  \providecommand{\bysame}{\leavevmode\hbox to3em{\hrulefill}\thinspace}
\providecommand{\mr}[1]{MR~\href{http://www.ams.org/mathscinet-getitem?mr=#1}{#1}}
\providecommand{\zbl}[1]{zbMATH~\href{http://www.zentralblatt-math.org/zmath/en/search/?q=an:#1}{#1}}
\providecommand{\jfm}[1]{JFM~\href{http://www.emis.de/cgi-bin/JFM-item?#1}{#1}}
\providecommand{\arxiv}[1]{arXiv~\href{http://www.arxiv.org/abs/#1}{#1}}
\providecommand{\doi}[1]{DOI~\href{http://dx.doi.org/#1}{#1}}
\providecommand{\MR}{\relax\ifhmode\unskip\space\fi MR }
\providecommand{\MRhref}[2]{%
  \href{http://www.ams.org/mathscinet-getitem?mr=#1}{#2}
}
\providecommand{\href}[2]{#2}

\end{document}